\documentclass[english]{smfart}
\usepackage{sabbah_susy}
\def\subsubsections#1{\subsubsection*{#1}}
\begin{document}
\frontmatter
\title[Fourier-Laplace transform II]{Fourier-Laplace transform of a~variation of polarized~complex Hodge structure, II}

\author[C.~Sabbah]{Claude Sabbah}
\address{UMR 7640 du CNRS\\
Centre de Math\'ematiques Laurent Schwartz\\
\'Ecole polytechnique\\
F--91128 Palaiseau cedex\\
France}
\email{sabbah@math.polytechnique.fr}
\urladdr{http://www.math.polytechnique.fr/~sabbah}

\begin{abstract}
We show that the limit, by rescaling, of the `new supersymmetric index' attached to the Fourier-Laplace transform of a polarized variation of Hodge structure on a punctured affine line is equal to the spectral polynomial attached to the same object. We also extend the definition by Deligne of a Hodge filtration on the de~Rham cohomology of a exponentially twisted polarized variation of complex Hodge structure and prove a $E_1$-degeneration property for it.
\end{abstract}

\thanks{Part of this work has been achieved during and shortly after the conference ``New developments in Algebraic Geometry, Integrable Systems and Mirror symmetry'', while the author was visiting RIMS, Kyoto. The author thanks this institution for hospitality and Professor M.-H.~Saito for the invitation and financial support.}
\maketitle
\enlargethispage{-\baselineskip}
\tableofcontents
\mainmatter

\section*{Introduction}
The purpose of this article, mainly concerned with exhibiting properties of the Fourier-Laplace transform of a variation of Hodge structure, is twofold.
\enlargethispage{\baselineskip}
\begin{enumerate}
\item\label{enum:intro1}
Let~$X$ be a compact Riemann surface, let~$S$ be a finite set of points on~$X$. We will denote by $j:U=X\moins S\hto X$ the inclusion. Let $f:X\to\PP^1$ be a meromorphic function on~$X$ which is holomorphic on~$U$ and let $(V,\Vnablaf)$ be a holomorphic bundle on~$U$ equipped with a holomorphic connection. We denote by $\ccM$ the locally free $\cO_X(*S)$-module of finite rank with a connection having regular singularities at each point of~$S$ and such that $\ccM_{|U}=(V,\Vnablaf)$ (Deligne's meromorphic extension of $(V,\Vnablaf)$). In \cite{Deligne8406}, P\ptbl Deligne defines a Hodge filtration on the de~Rham cohomology of the exponentially twisted connection $H^*_{\DR}(X,\ccM\otimes\ccE^f)$, \ie that of the meromorphic bundle $\ccM$ with the twisted connection $\Vnablaf+df\wedge$, at least when the monodromy of $(V,\Vnablaf)$ is unitary (and thus corresponds to a variation of polarized Hodge structure of type $(0,0)$).

This Hodge filtration is indexed by real numbers, and Deligne proves a $E_1$-degeneration property for the de~Rham complex. It has a good behaviour with respect to duality.

One is naturally led to the following questions:
\begin{itemize}
\item
In what sense do we get a Hodge filtration, \ie what are the underlying Hodge properties?
\item
Why are the jumps of this Hodge filtration related to the eigenvalues of the monodromy of $f$ around $f=\infty$ (more precisely, the spectrum of~$f$ at infinity relative to $(V,\Vnablaf)$)?
\item
Is there a possible extension of this construction without the unitarity assumption, when $(V,\Vnablaf)$ is only assumed to underlie a polarized variation of Hodge structure?
\end{itemize}

In \S\ref{sec:Delignefilt}, we extend the construction by Deligne of a filtration on the twisted de~Rham cohomology $H^*_{\DR}(X,\ccM\otimes\nobreak\ccE^f)$ when $(V,\Vnablaf)$ underlies a variation of polarized complex Hodge structure and give an answer to the previous questions. However, for simplicity, we restrict to the case where $X=\PP^1=\Afu\cup\{\infty\}$ and $f$ is the coordinate function on~$\Afu$.

\medskip
\item\label{enum:intro2}
Let $H$ be a complex vector space equipped with a positive definite Hermitian form $h$ (that we call a Hermitian metric) and two endomorphisms $\cU$ and $\cQ$, where $\cQ$ is selfadjoint with respect to $h$. The other purpose of this article is to give a relation between polynomials of degree $\dim H$ attached to this situation:
\begin{itemize}
\item
On the one hand, the characteristic polynomial of $\cQ$, denoted by $\Susy_{(H,h,\cU,\cQ)}(T)$.
\item
On the other hand, the spectral polynomials. The spectral polynomial at infinity, as defined in \S\ref{subsec:specinfty} below, is attached to the holomorphic bundle with a meromorphic connection having a pole of order two associated to $\cU$ and $\cQ$ (\cf \S\ref{subsec:inttwst}), and denoted by $\SP^\infty_{(H,h,\cU,\cQ)}(T)$. With a supplementary assumption called ``no ramification'', one can also define the spectral polynomial at the origin $\SP^0_{(H,h,\cU,\cQ)}(T)$ (\cf\S\ref{subsec:specorigin}).
\end{itemize}
There is a rescaling operator $\mu_\tau^*$, parametrized by $\tau\in\CC^*$, acting on the data $(H,h,\cU,\cQ)$ (more precisely and more accurately, on the associated integrable twistor structure, \cf Appendix~\ref{sec:appendiceB}).

The other main result of this article (Theorem \ref{th:conjHertling}) is to prove, under some conditions made explicit below (namely, $(H,h,\cU,\cQ)$ is the de~Rham cohomology of the exponential twist of a variation of polarized Hodge structure on a punctured line, in particular, the ``no ramification'' condition holds), a relation conjectured by C\ptbl Hertling:
\begin{equation}\label{eq:conjHertling}
\begin{split}
\lim_{\tau\to0}\Susy_{\mu_\tau^*(H,h,\cU,\cQ)}(T)&=\SP^\infty_{(H,h,\cU,\cQ)}(T),\\
\lim_{\tau\to\infty}\Susy_{\mu_\tau^*(H,h,\cU,\cQ)}(T)&=\SP^0_{(H,h,\cU,\cQ)}(T).
\end{split}
\end{equation}
A similar relation was first proved by C\ptbl Hertling (\cf \cite[Th\ptbl7.20]{Hertling01}) when the connection $\nabla$ has a regular singularity at $\hb=0$.
\end{enumerate}

The relation between the two approaches \eqref{enum:intro1} and \eqref{enum:intro2} above is made explicit in Remark \ref{rem:relations1-2} below. Both questions rely on a detailed analysis of the Fourier-Laplace transform of a variation of polarized complex Hodge structure on the punctured affine line.

\begin{remarque*}
In the recent preprint \cite{Mochizuki08b}, T\ptbl Mochizuki extends the limit theorems \ref{th:limSusytame} and \ref{th:limSusywild} in the higher dimensional case and gives applications to a characterization of nilpotent orbits.
\end{remarque*}

\subsubsections{Acknowledgements}
I thank Claus Hertling for useful discussions on this subject and for his comments on a preliminary version of this article. I thank the referee for his careful reading of the manuscript and useful comments.

\section{Connections with a pole of order two}
Let $\Omega$ be an open disc centered at the origin of $\CC$ with coordinate~$\hb$ and let~$\cH$ be a $\cO_\Omega$-locally free sheaf with a meromorphic connection~$\nabla$ having a pole of order two at the origin and no other pole (one can consider a more general situation, but we will restrict to this setting). We will moreover assume that the eigenvalues of the monodromy operator and of the formal monodromy operator have absolute value equal to one (so that the $V$-filtrations below are indexed by real numbers; here also, a more general situation could be considered, but we will restrict to this setting).

\subsection{Spectrum at infinity}\label{subsec:specinfty}
There exists a unique locally free $\cO_{\PP^1}(*\infty)$-module~$\wt\cH$ equipped with a meromorphic connection $\nabla$ with poles at~$0$ and~$\infty$ only, such that $\infty$ is a regular singularity, and which coincides with~$\cH$ when restricted to~$\Omega$ (it is called the \emph{meromorphic Deligne extension} of $(\cH,\nabla)$ at infinity, \cf\cite{Deligne70}). Let us denote by $\hb'=1/\hb$ the coordinate at infinity on $\PP^1$. For any $\gamma\in\RR$, the $\gamma$-Deligne extension of $\wt\cH$ at infinity is the locally free $\cO_{\PP^1}$-module $V^\gamma\wt\cH$ on which the connection has a logarithmic pole at infinity with residue having eigenvalues in $[\gamma,\gamma+1)$. According to the Birkhoff-Grothendieck theorem, $V^\gamma\wt\cH$ decomposes as the direct sum of rank-one locally free $\cO_{\PP^1}$-modules $V^\gamma\wt\cH=\cO_{\PP^1}(a_1)\oplus\cdots\oplus\cO_{\PP^1}(a_{\rk\cH})$ with $a_1\geq a_2\geq\cdots$. We denote by $v_\gamma$ the number of such line bundles which are $\geq0$ and by $\nu_\gamma$ the difference $v_\gamma-v_{>\gamma}$.

We can express these numbers a little differently. We have a natural morphism
\begin{equation}\label{eq:Vg}
\cO_{\PP^1}\otimes_\CC\Gamma(\PP^1,V^\gamma\wt\cH)\to V^\gamma\wt\cH
\end{equation}
whose image is denoted by $\wt\cV^\gamma$. This is a subbundle of $V^\gamma\wt\cH$ in the sense that $V^\gamma\wt\cH/\wt\cV^\gamma$ is also a locally free sheaf of $\cO_{\PP^1}$-modules; more precisely, fixing a Birkhoff-Grothendieck decomposition as above, we have $\wt\cV^\gamma=\bigoplus_{i\mid a_i\geq0}\cO_{\PP^1}(a_i)$ (indeed, for any line bundle $\cO_{\PP^1}(k)$, $\cO_{\PP^1}\otimes_\CC\Gamma(\PP^1,\cO_{\PP^1}(k))\to\cO_{\PP^1}(k)$ is onto if $k\geq0$ and~$0$ if $k<0$) so $\wt\cV^\gamma$ is a direct summand of $V^\gamma\wt\cH$ of rank $v_\gamma$. Restricting to~$\Omega$, we get a decreasing filtration $\cV^\cbbullet$ of~$\cH$ indexed by $\RR$. The graded pieces $\gr^\gamma_\cV\cH\defin\cV^\gamma/\cV^{>\gamma}$ are locally free $\cO_\Omega$-modules (being isomorphic to the kernel of $\cH/\cV^{>\gamma}\to\cH/\cV^\gamma$), and $\nu_\gamma=\rk\gr^\gamma_\cV\cH$.

Let us recall the definition of the spectral polynomial $\SP^\infty_\cH$ (\cf \cite{Bibi96bb} or \cite[\S III.2.b]{Bibi00}).

\begin{definition}[Spectrum at infinity]\label{def:spinfty}
The spectral polynomial of~$\cH$ at infinity is the polynomial $\SP^\infty_{\cH}(T)=\prod_\gamma(T-\gamma)^{\nu_\gamma}$ with (for any $\hbo\in\Omega$)
\[
\nu_\gamma=\rk\gr^\gamma_\cV\cH=\dim i^*_\hbo\gr^\gamma_\cV\cH=\dim\cV^\gamma/(\cV^{>\gamma}+(\hb-\hbo)\cV^\gamma).
\]
\end{definition}

In the following, we will often use an algebraic version of the previous construction, which is obtained as follows: set $G_0=\Gamma(\PP^1,\wt\cH)$, which is a free $\CC[\hb]$-module of finite rank; the (decreasing) Deligne $V$-filtration of $G\defin\CC[\hb,\hbm]\otimes_{\CC[\hb]}G_0$ at $\hb=\infty$ is a filtration $V^\cbbullet G$ of $G$ by $\CC[\hb']$-free submodules; in particular, $\hb' V^{\gamma}G=V^{\gamma+1}G$ and $\hb\partial_\hb+\gamma=-\hb'\partial_{\hb'}+\gamma$ is nilpotent on $\gr^\gamma_VG\defin V^{\gamma}G/V^{>\gamma}G$. Then, $\Gamma(\PP^1,V^\gamma\wt\cH)=V^{\gamma}G\cap G_0=:V^\gamma G_0$.

When tensored with $\cO_{\PP^1}(*\infty)$ and after taking global sections, \eqref{eq:Vg} is the inclusion morphism
\[
\CC[\hb]\cdot V^\gamma G_0\hto G_0.
\]
As $\CC[\hb]\cdot V^\gamma G_0$ is a direct summand in $G_0$, this inclusion induces an inclusion of fibres at~$\hbo$ for any $\hbo\in\CC$, and so
\begin{align*}
V^\gamma\big(G_0/(\hb-\hbo)G_0\big)&\defin V^\gamma G_0/[(\hb-\hbo)G_0\cap V^\gamma G]\\
&=\big[\CC[\hb]\cdot V^\gamma G_0+(\hb-\hbo)G_0\big]/(\hb-\hbo)G_0
\end{align*}
has dimension $v_\gamma$. Then we also have
\begin{equation}\label{eq:nugamma}
\begin{split}
\nu_\gamma(G_0)&=\dim\gr_V^\gamma(G_0/(\hb-\hbo)G_0)\\
&=\dim (G_0\cap V^\gamma G)\big/\big[(G_0\cap V^{>\gamma}G)+((\hb-\hbo)G_0\cap V^\gamma G)\big].
\end{split}
\end{equation}

\begin{exemple}\label{ex:cohomtame}
In \cite{Bibi96bb} (where an increasing version of the $V$-filtration is used, hence the change of sign below), this polynomial is denoted by $\SP_\psi(G,G_0)$. Correspondingly, the set of pairs $(-\gamma,\nu_\gamma)$ above is called the spectrum (at infinity) of $(G,G_0)$. When $G_0$ is the Brieskorn lattice attached to a cohomologically tame function on an affine smooth variety of dimension $n+1$ (\cf \loccit), the spectrum at infinity is symmetric with respect to $(n+1)/2$ and the numbers $-\gamma$ belong to $[0,n+1]\cap\QQ$.
\end{exemple}

\subsection{Spectrum at the origin}\label{subsec:specorigin}
We now make a supplementary assumption on $(\cH,\nabla)$. Let us denote by $\cH[1/\hb]$ the locally free $\cO_\Omega[1/\hb]$-module $\cO_\Omega[1/\hb]\otimes_{\cO_\Omega}\cH$, with its natural meromorphic connection. By the Levelt-Turrittin theorem, the associated formal module $\CC\lcr\hb\rcr[1/\hb]\otimes_{\cO_\Omega}\cH$ can be decomposed, after a suitable ramification of~$\hb$, as the direct sum of meromorphic connections which are tensor product of a rank-one irregular connection with a regular one. Here, we make the assumption that \emph{no ramification is needed to get the Levelt-Turrittin decomposition} (\cf Appendix~\ref{sec:appendiceB} for the need of such a condition). One can formulate this condition in terms of Laplace transforms, in the coordinate $\hb'\defin1/\hb$:

\begin{lemme}
The ``no ramification'' condition is equivalent to saying that the Laplace transform of the $\CC[\hb']\langle\partial_{\hb'}\rangle$-module $G$ associated with~$\cH$ has only regular singularities (included at infinity).
\end{lemme}

\begin{proof}[Sketch of proof]
This follows from the slope correspondence in the Fourier-Laplace transform (\cf \cite[Chap\ptbl V]{Malgrange91}). We will not distinguish between the Laplace transform and the inverse Laplace transform. Assume that $G$ is the Laplace transform of $M$. The formal part of $M$ at the origin produces the formal part of slope $<1$ of $G$ at $\hb'=\infty$. By assumption, only the slope~$0$ can appear, so $M$ is regular at the origin. A similar reasoning can be done at each singular point of $M$ by twisting by a suitable exponential term, showing that~$M$ has only regular singularities at finite distance. The part of slope $<1$ of~$M$ at infinity produces the formal part of $G$ at $\hb'=0$, and as~$G$ is regular at $\hb'=0$, only the slope~$0$ occurs as a slope $<1$ for $M$ at infinity. Slopes equal to~$1$ for $M$ at infinity would produce singular points of $G$ at finite distance, and not equal to $\hb'=0$. There are none. Lastly, slopes $>1$ for $M$ at infinity would produce slopes $>1$ for $G$ at $\hb'=\infty$. There are none. Hence $M$ has to be regular (slope~$0$) at infinity.
\end{proof}

Let us set $\cH^\wedge=\CC\lcr\hb\rcr\otimes_{\cO_\Omega}\cH$. When the ``no ramification'' condition is fulfilled, the Levelt-Turrittin decomposition for $\cH^\wedge[1/\hb]$ already exists for~$\cH^\wedge$. There exists then a finite number of pairwise distinct complex numbers $c_i$ ($i\in\nobreak I$) and a finite number of free $\CC\{\hb\}$-modules~$\cH_i$ with a regular meromorphic connection having a pole of order at most two, such that
\begin{equation}\label{eq:Levelt-Turrittin}
\cH^\wedge\simeq\tbigoplus_{i\in I}(\cH_i\otimes\cE^{c_i/\hb})^\wedge,
\end{equation}
where $\cE^{c_i/\hb}$ is $\CC\{\hb\}$ equipped with the connection $d-c_id\hb/\hb^2$. Each~$\cH_i$ is equipped with a regular meromorphic connection $\nabla$. The free $\CC\{\hb\}[1/\hb]$-module $\cH_i[1/\hb]$ has a canonical decreasing Deligne filtration $V^\bbullet\cH_i[1/\hb]$ indexed by real numbers (by our assumption) so that $\hb\partial_\hb-\gamma$ is nilpotent on the vector space $\gr^\gamma_V\cH_i[1/\hb]$.

\begin{definition}[Spectrum at the origin]\label{def:sporigin}
For any $i\in I$, the spectral polynomial of the regular meromorphic connection~$\cH_i$ at the origin is the polynomial $\SP^0_{\cH_i}(T)=\prod_\gamma(T+\gamma)^{\mu_{i,\gamma}}$, with
\[
\mu_{i,\gamma}=\dim\frac{\cH_i\cap V^\gamma\cH_i[1/\hb]}{\cH_i\cap V^{>\gamma}\cH_i[1/\hb]+\hb\cH_i\cap V^\gamma\cH_i[1/\hb]},
\]
and we set $\SP^0_{\cH}(T)=\prod_i\SP^0_{\cH_i}(T)$.
\end{definition}
(The choice $T+\gamma$ is done in order to have similar formulas for $\SP^0$ and $\SP^\infty$.)

\begin{exemple}\label{ex:cohomtameloc}
When~$\cH$ is the analytization of the Brieskorn lattice $G_0$ of a cohomologically tame function on a smooth affine variety (\cf Example \ref{ex:cohomtame}), then $\cH_i\neq0$ only if $-c_i$ is a critical value of this function, and $\cH_i^\wedge$ is the formal (with respect to $\partial_t^{-1}=\hb$) local Brieskorn lattice at this critical value (this follows from \cite[Prop\ptbl V.3.6]{Bibi00} for instance). The set of pairs $(\gamma,\mu_{i,\gamma})$ is the spectrum at this critical value (with a shift by one with respect to the definition of \cite{S-S85}): it is symmetric with respect to $(n+1)/2$ and the numbers $\gamma$ belong to $(0,n+1)$. (See also \cite{S-S85}, \cite{MSaito89} and \cite[Chap\ptbl10--11]{Hertling00} and the references therein for detailed results in the case of a singularity germ.)
\end{exemple}

\begin{remarque}\label{rem:modifnabla}
Assume that $\nabla$ has a pole of order one on~$\cH$. Then $\SP^0_\cH$ is the characteristic polynomial of $-\Res\nabla$ (residue of the connection at $\hb=0$). We also have $\SP^\infty_\cH=\SP^0_\cH$ (\cf \eg \cite[Ex\ptbl III.2.6]{Bibi00}). We conclude that in an exact sequence of logarithmic connections, $\SP^\infty_\cH$ and $\SP^0_\cH$ behave multiplicatively.
\end{remarque}

\subsection{Connection with a pole of order two by Laplace transform}\label{subsec:Laplacetr}
Let us recall the notion of Laplace transform of a filtered $\Clt$-module (\cf \cite[\S V.2.c]{Bibi00} or \cite[\S1.d]{Bibi05}. Let $\Afu$ be the affine line with coordinate~$t$ and let $M$ be a holonomic $\Clt$-module. We set $G\defin M[\ptm]=\CC[t]\langle\partial_t,\ptm\rangle\otimes_{\Clt}M$ (it is known that $G$ is also holonomic as a $\Clt$-module) and we denote by $\muloc:M\to G$ the natural morphism (the kernel and cokernel of which are isomorphic to powers of $\CC[t]$ with its natural structure of left $\Clt$-module). For any lattice~$L$ of~$M$, \ie a $\CC[t]$-submodule of finite type such that $M=\CC[\partial_t]\cdot L$, we set
\begin{equation}\label{eq:satdtm}
G_0^{(L)}=\sum_{j\geq0}\partial_t^{-j}\muloc(L).
\end{equation}
This is a $\CC[\ptm]$-submodule of $G$. Moreover, because of the relation $[t,\ptm]=(\ptm)^2$, it is naturally equipped with an action of $\CC[t]$. If $M$ has a regular singularity at infinity, then $G_0^{(L)}$ has finite type over $\CC[\ptm]$ (\cf \cite[Th\ptbl V.2.7]{Bibi00}). We have $G=\CC[\partial_t]\cdot G_0^{(L)}$.

Let us now assume that $M$ is equipped with a good filtration $F_\bbullet M$. In the following, in order to keep the correspondence with Hodge theory, we will work with decreasing filtrations $F^\cbbullet M$, the correspondence being given by $F^pM\defin F_{-p}M$. Let $p_0\in\ZZ$. We say that $F^\cbbullet M$ \emph{is generated by $F^{p_0}M$} if, for any $\ell\geq0$, we have $F^{p_0-\ell}M=F^{p_0}M+\cdots+\partial_t^\ell F^{p_0}M$. The $\CC[\ptm]$-module $\partial_t^{p_0}G_0^{(F^{p_0})}$ does not depend on the choice of the index~$p_0$, provided that the generating assumption is satisfied (\cf \cite[\S1.d]{Bibi05}). We thus define the \emph{Brieskorn lattice of the filtration $F^\cbbullet M$} as
\begin{equation}\label{eq:defG0F}
G_0^{(F)}=\partial_t^{p_0}G_0^{(F^{p_0})}\quad\text{for some (or any) index $p_0$ of generation.}
\end{equation}
If we also set $\hb=\partial_t^{-1}$, then $G_0^{(F)}$ is a free $\CC[\hb]$-module which satisfies $G=\CC[\hb,\hbm]\otimes_{\CC[\hb]}G_0^{(F)}$ and which is stable by the action of $\hb^2\partial_\hb\defin t$.

For any~$p$, we have
\begin{equation}\label{eq:muloc}
\muloc(F^pM)\subset \hb^pG_0^{(F)}.
\end{equation}
Indeed, $\hb^pG_0^{(F)}=\sum_{j\geq0}\partial_t^{-j-p+p_0}\muloc(F^{p_0}M)$; if $p\geq p_0$, we have $\partial_t^{p-p_0}F^pM\subset F^{p_0}M$, hence the desired inclusion after applying $\muloc$; if $p\leq p_0$, we have $F^pM=F^{p_0}+\cdots+\partial_t^{p_0-p}F^{p_0}M$, and the result is clear.

\subsection{Integrable twistor structures}\label{subsec:inttwst}
Let $H$ be a finite dimensional complex vector space equipped with a Hermitian metric $h$ and of two endomorphisms $\cU$ and $\cQ$, with $\cQ$ being selfadjoint with respect to $h$. Let $\cU^\dag$ be the $h$-adjoint of $\cU$. Let~$\Omega_0$ be an open neighbourhood of the closed disc $|\hb|\leq1$ in~$\CC$ and let us set $\cH'=\cO_{\Omega_0}\otimes_\CC H$, equipped with the meromorphic connection $\nabla=d+(\hb^{-2}\cU-\hb^{-1}\cQ-\cU^\dag)d\hb$.

We will denote by $\cT=(\cH',\cH',\cCS)$ the associated twistor structure (as defined in \S\ref{subsec:vartw} below, by taking~$X$ to be a point).

We will denote by $\SP^\infty_\cT(T)$ or by $\SP^\infty_{(H,h,\cU,\cQ)}(T)$ the spectral polynomial at infinity $\SP^\infty_{\cH'}(T)$. On the other hand, if $\nabla$ has no ramification at the origin, we will denote by $\SP^0_\cT(T)$ or by $\SP^0_{(H,h,\cU,\cQ)}(T)$ the spectral polynomial at the origin $\SP^0_{\cH'}(T)$.

\section{A review on integrable twistor $\cD$-modules}\label{se:revtw}
In this section and in Section \ref{sec:exptwFL}, we gather the notation and results needed for the proofs of the main theorems of this article. We refer to \cite{Bibi01c,Bibi05,Bibi05b} for details.

\subsection{Integrable harmonic Higgs bundles}\label{subsec:harmHiggs}
Let~$X$ be a complex manifold and let $E$ be a holomorphic bundle on~$X$, equipped with a Hermitian metric~$h$. For any operator~$P$ acting linearly on~$E$, we will denote by $P^\dag$ its adjoint with respect to~$h$.

Let $\theta$ be a holomorphic Higgs field on $E$, that is, an $\cO_X$-linear morphism $E\to\Omega^1_X\otimes_{\cO_X}E$ satisfying the ``integrability relation'' \hbox{$\theta\wedge\theta=0$}. We then say that $(E,\theta)$ is a Higgs bundle (\cf \cite{Simpson92}).

Let $E$ be a holomorphic bundle with a Hermitian metric $h$ and a holomorphic Higgs field $\theta$. Let $H$ be the associated $C^\infty$ bundle, so that $E=\ker d''$, let $D=D'+D''$, with $D''=d''$, be the Chern connection of~$h$. We say, after \cite{Simpson92}, that $(E,h,\theta)$ is a \emph{harmonic Higgs bundle} (or that~$h$ is \emph{Hermite-Einstein} with respect to $(E,\theta)$) if $\VD\defin D+\theta+\theta^\dag$ is an integrable connection on~$H$. The holomorphic bundle $V=\ker(d''+\theta^\dag)$ is then equipped with a flat holomorphic connection $\Vnablaf$, which is the restriction of $\VD'\defin D'+\theta$ to $V$.

We say (\cf \cite{Hertling01}, see also \cite[Chap\ptbl7]{Bibi01c}) that it is \emph{integrable} if there exist two endomorphisms $\cU$ and $\cQ$ of $H$ such that
\begin{align}
\text{$\cU$ is holomorphic, \ie }d''(\cU)&=0,\label{eq:higgsint1}\\
\cQ^\dag&=\cQ,\label{eq:higgsint2}\\
[\theta,\cU]&=0,\label{eq:higgsint3}\\
D'(\cU)-[\theta,\cQ]+\theta&=0,\label{eq:higgsint4}\\
D'(\cQ)+[\theta,\cU^\dag]&=0\label{eq:higgsint5}.
\end{align}

\begin{remarque}\label{rem:QHert}
Let us note that $(\cU+c\id,\cQ+\lambda\id)$ satisfy the same equations for any $c\in\CC$ and $\lambda\in\RR$. One way to fix $\lambda$ is to impose a compatibility condition with a given supplementary real structure. This would impose that $\cQ$ is purely imaginary (\cf \cite{Hertling01}). We then denote by $\cQ^\Hert$ this choice, which is the only one among the $\cQ+\lambda\id$, $\lambda\in\RR$, to be purely imaginary. With respect to the symmetric nondegenerate bilinear form deduced from the Hermitian metric and the real structure, $\cQ^\Hert$ is skewsymmetric, hence its characteristic polynomial satisfies $\Susy_{(H,h,\cU,\cQ^\Hert)}(-T)=(-1)^{\dim H}\Susy_{(H,h,\cU,\cQ^\Hert)}(T)$.
\end{remarque}

For any $x\in X$, the Hermitian vector space $(H_x,h_x)$ decomposes with respect to the eigenvalues of $\cQ_x$. However, these eigenvalues, which are real, may vary with $x$.

\begin{exemple}[Polarized variation of complex Hodge structure]\label{exem:PVCHS}
If $\cU=0$ (or $\cU=c\id$), then, according to \eqref{eq:higgsint5} and \eqref{eq:higgsint5}$^\dag$, $D(\cQ)=0$ and, working in a local $h$-orthonormal frame where $\cQ$ is diagonal, this implies that the eigenvalues of~$\cQ$ are constant. Let $H^p$ denote the eigen subbundle corresponding to the eigenvalue $p\in\RR$. Then $D'H^p\subset \Omega_X^1\otimes H^p$, $D''H^p\subset \Omega_{\ov X}^1\otimes H^p$ and \eqref{eq:higgsint4} implies $\theta H^p\subset\nobreak \Omega_X^1\otimes\nobreak H^{p-1}$. The decreasing filtration (indexed by $\RR$) defined by $F^p H=\bigoplus_{p'\geq p}H^{p'}$ is stable by~$\VD''$, hence induces a filtration $F^\cbbullet V$ of the holomorphic bundle $V\defin \ker \VD''$ by holomorphic subbundles, which satisfies $\Vnablaf F^p V\subset F^{p-1}V$. Moreover, if we choose a sign $\epsilon_p\in\{\pm1\}$ in such a way that, for any $p\in\RR$, $\epsilon_{p+1}=-\epsilon_p$, the nondegenerate sesquilinear form~$k$ defined by the properties that the decomposition $\bigoplus_{p\in\RR}H^p$ is $k$-orthogonal and $k_{|H^p}=\epsilon_p h_{|H^p}$, is $\VD$-flat. We thus recover the standard notion of a polarized variation of complex Hodge structure of weight~$0$, if we accept filtrations indexed by real numbers, and if we set $H^p=H^{p,-p}$.
\end{exemple}

\subsection{Variations of twistor structures}\label{subsec:vartw}
The notion of an integrable variation of twistor structures (and, more generally, that of an integrable twistor $\cD$-module) is a convenient way to handle integrable harmonic Higgs bundles. It was introduced in \cite{Simpson97}. The presentation given here follows \cite{Bibi01c}, and the reader can also refer to \cite[Chap\ptbl3]{Mochizuki07}.

\subsubsections{Notation}
If~$X$ is a complex manifold, $\ov X$ will denote the conjugate manifold (with structure sheaf $\ov{\cO_X}$), and $X_\RR$ will denote the underlying real-analytic or $C^\infty$-manifold. We will denote by $\PP^1$ the Riemann sphere, covered by the two affine charts $\simeq\Afu$ with coordinate~$\hb$ and $1/\hb$, and by $p:X\times\nobreak\PP^1\to\nobreak X$ the projection.

The coordinate~$\hb$ being fixed, we denote by $\bS$ the circle $|\hb|=1$, by~$\Omega_0$ an open neighbourhood of the closed disc $\Delta_0\defin\{|\hb|\leq1\}$ and by~$\Omega_\infty$ an open neighbourhood of the closed disc $\Delta_\infty\defin\{|\hb|\geq1\}$.

We will denote by $\sigma:\PP^1\to\ov{\PP^1}$ the anti-holomorphic involution $\hb\mto-1/\ov\hb$. We assume that $\Omega_\infty=\sigma(\Omega_0)$. We denote by $\iota:\PP^1\to\PP^1$ the holomorphic involution $\hb\mto-\hb$.

It will be convenient to use the notation $\cX$ for $X\times\nobreak\Omega_0$ and $\ovv\cX$ for $\ov X\times\nobreak\Omega_\infty$. Let us introduce the notion of twistor conjugation. Let~$\cH''$ be a holomorphic vector bundle on $\cX$. Then $\ov{\cH''}$ is a holomorphic bundle on the conjugate manifold $\ov\cX\defin\ov X\times\nobreak\ov{\Omega_0}$ and $\sigma^*\ov{\cH''}$ is a holomorphic bundle on $\ovv\cX=\ov X\times\nobreak\Omega_\infty$ (\ie is an anti-holomorphic family of holomorphic bundles on $\Omega_\infty$). We will set $\ovv{\cH''}\defin\sigma^*\ov{\cH''}$.

By a $C^\infty$ family of holomorphic vector bundles on $\PP^1$ parametrized by $X_\RR$ we will mean the data of a triple $(\cH',\cH'',\cCS)$ consisting of holomorphic vector bundle $\cH',\cH''$ on $X\times\nobreak\Omega_0$ and a nondegenerate $\cO_{X\times\bS}\otimes_{\cO_\bS}\cO_{\ov X\times\bS}$-linear morphism
\[
\cCS:\cHS'\otimes_{\cO_\bS}\ovv{\cHS''}\to\cC^{\infty,\an}_{X_\RR\times\bS},
\]
where $\cC^{\infty,\an}_{X_\RR\times\bS}$ is the sheaf of $C^\infty$ functions on $X_\RR\times\nobreak\bS$ which are real analytic with respect to $\hb\in\bS$. The nondegeneracy condition means that $\cCS$ defines a $C^{\infty,\an}$-gluing between the dual $\cH^{\prime\vee}$ of~$\cH'$ and $\ovv{\cH''}$, giving rise to a $\cC^{\infty,\an}_{X_\RR\times\PP^1}$-locally free sheaf of finite rank that we denote by $\wt\cH$.

\subsubsections{Variations of twistor structures}
By a $C^\infty$ variation of twistor structure on~$X$ we mean the data of a triple $(\cH',\cH'',\cCS)$ defining a $C^\infty$ family of holomorphic bundles on $\PP^1$ as above, such that each of the holomorphic bundles $\cH',\cH''$ is equipped with a \emph{relative} holomorphic connection
\begin{equation}\label{eq:nabla}
\nabla:\cH'('')\to\frac1\hb\,\Omega^1_{\cX/\Omega_0}\otimes_{\cO_{X\times\Omega_0}}\cH'('')
\end{equation}
which has a pole along $\hb=0$ and is integrable. Moreover, the pairing $\cCS$ has to be compatible (in the usual sense) with the connections, \ie
\[
d'\cCS(m',\ovv{m''})=\cCS(\nabla m',\ovv{m''})\quad\text{and}\quad d''\cCS(m',\ovv{m''})=\cCS(m',\ovv{\nabla m''}).
\]
Let us note that we can define $\ovv\nabla$ as
\[
\ovv\nabla:\ovv{\cH''}\to \hb\Omega^1_{\ovv\cX/\Omega_\infty}\otimes_{\cO_{\ovv\cX}}\ovv{\cH''}.
\]
If we regard $\cCS$ as a $C^{\infty,\an}$-linear isomorphism
\begin{equation}\label{eq:cC}
\cCS:\cC^{\infty,\an}_{X_\RR\times\bS}\otimes_{\cO_{\ov X\times\bS}}\ovv{\cHS''}\isom\cC^{\infty,\an}_{X_\RR\times\bS}\otimes_{\cO_{X\times\bS}}\cHS^{\prime\vee},
\end{equation}
the compatibility with $\nabla$ means that $\cCS$ is compatible with the connection $d'+\ovv\nabla$ on the left-hand side and $\nabla^\vee+d''$ on the right-hand side, where $d',d''$ are the standard differentials with respect to~$X$ only.
\begin{itemize}
\item
The adjoint $(\cH',\cH'',\cCS)^\dag$ is defined as $(\cH'',\cH',\cCS^\dag)$, with
\[
\cCS^\dag(m'',\ovv{m'})\defin\ovv{\cCS(m',\ovv{m''})}.
\]
With respect to \eqref{eq:cC}, we can write $\cCS^\dag=\ovv\cCS{}^\vee$.
\item
If $k\in\hZZ$, the Tate twist $(k)$ is defined by $\cT(k)\defin(\cH',\cH'',(i\hb^{-2k})\cCS)$.
\end{itemize}

We say that the variation is
\begin{itemize}
\item
\emph{Hermitian} if $\cH''=\cH'$ and $\cCS$ is ``Hermitian'', \ie $\cCS^\dag=\cCS$,
\item
\emph{pure of weight~$0$} if the restriction to each $x\in X$ defines a \emph{trivial} holomorphic bundle on $\PP^1$,
\item
\emph{polarized and pure of weight~$0$} if it is pure of weight~$0$, Hermitian, and the Hermitian form on the bundle $\ov H\defin p_*\wt\cH$ is positive definite, \ie is a Hermitian metric.
\end{itemize}

\begin{lemme}[C\ptbl Simpson \cite{Simpson97}]\label{lem:simpson}
We have an equivalence between variations of polarized pure twistor structures of weight~$0$ and harmonic Higgs bundles, by taking $\PP^1$-global sections.\qed
\end{lemme}

\begin{remarque}
There is a natural notion of a (polarized) variation of twistor structure of weight~$w\in\ZZ$ (\cf \cite{Simpson97}). We say that $\cT=(\cH',\cH'',\cCS)$ is pure of weight~$w$ if the restriction to each $x\in X$ defines a bundle isomorphic to $\cO_{\PP^1}(w)^d$. A Hermitian duality is an isomorphism $\cS=(S',S''):\cT\to\cT^\dag(-w)$. We say that~$\cS$ is a polarization if the Tate twisted object $(\cT,\cS)(w/2)$ is polarized (\cf \cite{Bibi01c} for details, see also \cite{Mochizuki07}).
\end{remarque}

\subsection{Integrable variations of twistor structures}\label{subsec:intvartw}
A variation of twistor structure $(\cH',\cH'',\cCS)$ is \emph{integrable} if the relative connection $\nabla$ on $\cH',\cH''$ comes from an absolute connection also denoted by $\nabla$, which has Poincar\'e rank one (\cf \cite[Chap\ptbl7]{Bibi01c}). In other words, $\hb\nabla$ should be an integrable meromorphic~$\hb$-connection on $\cH',\cH''$ with a logarithmic pole along $\hb=0$. We also ask for a supplementary compatibility property of the absolute connection with the pairing in the following way:
\[
\hb\frac{\partial}{\partial\hb}\cCS(m',\ovv{m''})=\cCS(\hb\nabla_{\partial_\hb}m',\ovv{m''})-\cCS(m',\ovv{\hb\nabla_{\partial_\hb}m''}).
\]
[Here, we regard $\cC^{\infty,\an}_{X_\RR\times\bS}$ as the sheaf of germs along $X\times\nobreak\bS$ of $C^\infty$ functions which are holomorphic with respect to~$\hb$; when considering it as the sheaf of $C^\infty$ functions which are real analytic with respect to $\hb\in\bS$, one should replace the operator $\hb\frac{\partial}{\partial\hb}$ with $\hb\frac{\partial}{\partial\hb}-\ov\hb\frac{\partial}{\partial\ov\hb}$.]

Lemma \ref{lem:simpson} can be extended to integrable variations:

\begin{lemme}[C\ptbl Hertling \cite{Hertling01}, \cf also {\cite[Cor\ptbl7.2.6]{Bibi01c}}]\label{lem:hertling}
The equivalence of Lemma \ref{lem:simpson} specializes to an equivalence between integrable variations of pure polarized twistor structures of weight~$0$ and integrable harmonic Higgs bundles, \ie harmonic Higgs bundles equipped with endomorphisms $\cU,\cQ$ satisfying \eqref{eq:higgsint1}--\eqref{eq:higgsint5}.\qed
\end{lemme}

Let us indicate one direction of the correspondence. Starting from $(H,h,\theta,\cU,\cQ)$, we construct an integrable variation of Hermitian twistor structures $(\cH',\cH'',\cCS)$ by setting $\cH=p^*H$ on $\cX$, with the $d''$-operator $D''_\cH\defin D''+\hb\theta^\dag$ and we set $\cH'=\ker D''_\cH=\cH''$. The relative connection \eqref{eq:nabla} is defined as the restriction of $D'_\cH\defin D'+\hbm\theta$ to~$\cH'$. The absolute connection is obtained by adding to the relative connection $\nabla$ the connection in the~$\hb$-variable $d'_\hb+(\hb^{-2}\cU-\hbm\cQ-\cU^\dag)d\hb$ (\cf \cite{Hertling01} or \cite[\S7.2.c]{Bibi01c} for more details).

\begin{remarque}
Given an integrable variation of twistor structure $\cT=(\cH',\cH'',\cCS)$, we will say that the structure obtained by changing the action of $\hb^2\nabla_{\partial_\hb}$ on~$\cH'$ to $\hb^2\nabla_{\partial_\hb}-\nobreak\lambda\hb$ and that on~$\cH''$ to $\hb^2\nabla_{\partial_\hb}-\ov\lambda\hb$ ($\lambda\in\CC$) is \emph{equivalent} to the previous one. If the variation of twistor structure is Hermitian, the equivalent structure is still compatible with~$\cS$ iff $\lambda\in\RR$.
\end{remarque}

\begin{remarque}[Tate twist and integrability]\label{rem:tateint}
The effect of the Tate twist $(k)$ (with $k\in\hZZ$) on the action of $\hb^2\partial_\hb$ is a shift between~$\cH'$ and~$\cH''$ by $2k\hb$. In this article, it will be convenient to choose a nonsymmetric shift, namely, the action on~$\cH''$ is unchanged, and that on~$\cH'$ is changed into $\hb^2\partial_\hb-2k\hb$.
\end{remarque}

\begin{remarque}\label{rem:U=0}
In the previous correspondence, we identify the bundle $E$ on~$X$ with the restriction $\cH''/\hb\cH''$. Then $\cU$ is induced by the action of $\hb^2\nabla_{\partial_\hb}$. If $\cU=0$ (the case of a variation of Hodge structure), then~$\cH''$ is stable by $\hb\nabla_{\partial_\hb}$, and the characteristic polynomial of $\cQ$ is equal to that of the restriction of $-\hb\nabla_{\partial_\hb}$ to $E$.
\end{remarque}

\subsection{The `new supersymmetric index'}\label{subsec:susy}
We assume in this paragraph that~$X$ is a point, so we work with twistor structures.

\begin{definition}[\cf \cite{C-F-I-V92} and \cite{Hertling01}]\label{def:Susy}
Let $\cT=(\cH',\cH',\cCS)$ be an integrable polarized pure twistor structure of weight~$0$ with polarization $\cS=(\id,\id)$. Let $\cU,\cQ$ be the associated endomorphisms of the corresponding finite dimensional complex vector space with positive definite Hermitian form. The endomorphism $\cQ$ is called the `new supersymmetric index' attached to $\cT$. We denote by $\Susy_\cT(T)$ its characteristic polynomial.
\end{definition}

It will be convenient to extend to any weight the previous definition.

\begin{definition}\label{def:Susyw}
Let $\cT=(\cH',\cH'',\cCS)$ be an integrable pure twistor structure of weight~$w$ with polarization~$\cS$. We set
\[
\Susy_\cT(T)\defin\Susy_{\cT(w/2)}(T).
\]
\end{definition}

Similarly, we define the spectral polynomials:

\begin{definition}
Let $\cT=(\cH',\cH'',\cCS)$ be an integrable twistor structure (not necessarily pure or polarized). We define (if the ``no ramification'' condition is fulfilled, for $\SP^0$):
\[
\SP^\infty_\cT(T)=\SP^\infty_{\cH''}(T)\quad\text{and}\quad\SP^0_\cT(T)=\SP^0_{\cH''}(T).
\]
\end{definition}

According to Definition \ref{def:Susyw} and to Remark \ref{rem:tateint}, we have, for any $k\in\hZZ$,
\begin{equation}\label{eq:tateSusySp}
\Susy_{\cT(k)}=\Susy_\cT,\quad\SP^\infty_{\cT(k)}=\SP^\infty_\cT,\quad\SP^0_{\cT(k)}=\SP^0_\cT.
\end{equation}

\subsection{Twistor $\cD$-modules}
In order to allow singularities in variations of twistor structure, we have introduced in \cite{Bibi01c} the notion of polarizable twistor $\cD$-module (see also \cite{Mochizuki07} for an extension of this notion with parabolic weights). We will briefly recall this notion.

We first introduce the sheaf $\cR_\cX$ of differential operators, locally isomorphic to $\cO_\cX\langle\partiall_{x_1},\dots,\partiall_{x_n}\rangle$, by setting $\partiall_{x_i}=\hb\partial_{x_i}$. A left $\cR_\cX$-module is nothing else but a $\cO_\cX$-module with a flat~$\hb$-connection.

The category $\RTriples(X)$ consists of triples $(\cM',\cM'',\cCS)$, where $\cM',\cM''$ are left $\cR_\cX$-modules and $\cCS:\cMS'\otimes_{\cOS}\ovv{\cM''}\to\DbhS{X}$ is a $\cR_{\cX|\bS}\otimes_{\cOS}\ovv{\cR_{\cX|\bS}}$ linear morphism with values in the sheaf $\DbhS{X}$ of distributions on $X\times\nobreak\bS$ which are continuous with respect to $\hb\in\bS$. There is a natural notion of morphism. This category has Tate twists by $\hZZ$, and a notion of adjunction (\cf \cite[\S1.6]{Bibi01c}). Restricting $\cM'$ or $\cM''$ to $\hb=1$ gives left $\cD_X$-modules, while restricting them to $\hb=0$ gives $\cO_X$-modules with a Higgs field.

There is a notion of direct image, hence of de~Rham cohomology when taking the direct image by the constant map.

Supplementary properties are introduced in order to define the notion of polarized twistor $\cD$-module of some weight. We will not recall them here and refer to \cite{Bibi01c} for further details.

\subsection{Specialization and integrability (the tame case)}\label{subsec:speinttame}
In the remaining part of Section \ref{se:revtw}, we assume that~$X$ is a disc with coordinate $x$ and we denote by $j:X^*\hto X$ the inclusion of the punctured disc $X\moins\{0\}$ in~$X$.

Let $\cT=(\cM',\cM'',\cCS)$ be a regular twistor $\cD$-module of weight~$w$ on~$X$ polarized by $\cS=(S'=(-1)^wS'',S'')$ (\cf \cite{Bibi01c}) with only singularity at $x=0$, so that $\cT_{|X^*}$ is a polarized variation of twistor structures of weight~$w$, which has a tame behaviour near the singularity, in the sense of \cite{Simpson90} (\cf \cite{Bibi01c,Mochizuki07}).

For any $\beta\in\CC$ with $\reel\beta\in(-1,0]$, the nearby cycle functor\footnote{In \cite{Bibi01c} it is defined with the increasing convention for the $V$-filtration. Here we use the decreasing one. The correspondence is $\Psi_x^\beta=\Psi_{x,\alpha}$ with $\beta=-\alpha-1$.} $\Psi_x^\beta$ sends such a triple $\cT$ to a triple $\Psi_x^\beta\cT\in\RTriples(\{0\})$, equipped with a morphism $\cN:\Psi_x^\beta\cT\in\RTriples(\{0\})\to\Psi_x^\beta\cT(-1)\in\RTriples(\{0\})$. If $\rM_\bbullet$ denotes the monodromy filtration, then the graded object $\gr_\bbullet^\rM\Psi_x^\beta\cT$, equipped with the morphism $\gr_{-2}^\rM\cN$, is a graded Lefschetz twistor $\cD$-module of weight~$w$ and type $\epsilon=-1$. Moreover, $\Psi_x^\beta\cS$ induces, by grading, a polarization of this object.

We have a similar result for vanishing cycles: $(\gr_\bbullet^\rM\phi_x^{-1}\cT,\gr_{-2}^\rM\cN)$ is an object of $\MLTr(X,w;-1)$ and $\phi_x^{-1}\cS$ induces, by grading, a polarization (this follows from \cite[Cor\ptbl4.1.17]{Bibi01c}).

Let us moreover assume that $(\cT,\cS)$ is integrable (\cf \cite[Chap\ptbl7]{Bibi01c}, where one should modify (7.1.2) by using the operator $\hb\partial/\partial\hb-\ov\hb\partial/\partial\ov\hb$ (standard conjugation) on the left-hand side, as (7.1.2) was mistakenly written there for distributions which are holomorphic with respect to~$\hb$), so that, in particular, the corresponding variation is integrable on~$X^*$. Then from \loccit we know that $\Psi_x^\beta\cM\neq0$ only if~$\beta$ is real, so that $\beta\star\hb=\beta\hb$ above, there is no difference between the notations $\Psi_x^\beta\cT$ and $\psi_x^\beta\cT$ of \loccit, and moreover, using the induced action of $\hb^2\partial_\hb$, $\Psi_x^\beta\cT$ remains integrable. Going then to $\gr_\ell^\rM\Psi_x^\beta\cT$, we get an integrable polarized twistor structure of weight~$w+\ell$ (\cf \cite[Lemma 7.3.8]{Bibi01c}). The action of $\hb^2\partial_\hb$ on $\gr_\ell^\rM\Psi_x^\beta\cM''$ is the action naturally induced by $\hb^2\partial_\hb$ on~$\cM''$. A similar result holds for $\gr_\ell^\rM\phi_x^{-1}\cT$.

\subsection{Specialization and integrability (the wild case)}\label{subsec:speintwild}
We keep the notation of \S\ref{subsec:speinttame}, but we will consider the more general case of a polarized wild twistor $\cD$-module $\cT=(\cM',\cM'',\cCS)$ of weight~$w$ with polarization~$\cS$, for which we refer to~\cite{Bibi06b,Mochizuki08}. We will also assume that the ``no ramification'' condition is fulfilled, that is, we assume that \cite[Prop\ptbl4.5.4]{Bibi06b} holds with ramification index $q$ equal to one. Therefore, setting $\cM=\cM'$ or $\cM''$, we have a formal decomposition
\[\tag{DEC$^\wedge$}
\wt\cM^\wedge\isom \tbigoplus_i(\wt\cR_i^\wedge\otimes\cE^{\varphi_i/\hb}),\qquad\varphi_j\in x^{-1}\CC[x^{-1}].
\]
Let us also notice that, when we restrict to $\hb=0$, the decomposition holds at the level of $\wt\cM/\hb\wt\cM$ (\cf \cite[Rem\ptbl4.5.5]{Bibi06b}).

For any $\varphi\in\xm\CC[\xm]$ and any $\beta\in\CC$ with $\reel\beta\in(-1,0]$, we set $\Psi_x^{\varphi,\beta}\wt\cM\defin\Psi_x^\beta(\wt\cM\otimes\cE^{-\varphi/\hb})$. We can then define the objects $\Psi_x^{\varphi,\beta}\wt\cT$, equipped with $\cN:\Psi_x^{\varphi,\beta}\wt\cT\to\Psi_x^{\varphi,\beta}\wt\cT(-1)$. The condition of being a polarized wild twistor $\cD$-module of weight~$w$ at $x=0$ means that, for all $\varphi,\beta$ as above, $(\gr_\bbullet^\rM\Psi_x^{\varphi,\beta}\wt\cT,\gr_{-2}^\rM\cN)$, equipped with the naturally induced sesquilinear duality, is a graded Lefschetz twistor structure of weight~$w$, in the sense of \cite[\S2.1.e]{Bibi01c}. As a consequence, if $\varphi=0$, the vanishing cycles $(\gr_\bbullet^\rM\phi_x^{0,-1}\wt\cT,\gr_{-2}^\rM\cN)$ are of the same kind.

For any $\varphi\in\xm\CC[\xm]$, the exponentially twisted $\cR_\cX[\xm]$-module $\wt\cM\otimes\cE^{-\varphi/\hb}$ remains integrable, if $\wt\cM$ is so, hence, according to \cite[Prop\ptbl7.3.1]{Bibi01c}, so are the modules $\Psi_x^{\varphi,\beta}\wt\cM$ ($\beta\in(-1,0]$) and $\phi_x^{0,-1}\wt\cM$. Moreover, the formal module $\wt\cM^\wedge$ is clearly integrable.

\begin{lemme}
Each $\cR_i^\wedge$ entering in the decomposition \textup{(DEC$^\wedge$)} is integrable.
\end{lemme}

\begin{proof}
Firstly, the irregular part $\wt\cM_\irr^\wedge$ of $\wt\cM^\wedge$ (\ie corresponding in (DEC$^\wedge$) to the sum over the nonzero $\varphi_i$) remains integrable, as, near any $\hbo\in\Omega_0$, it can be realized as the intersection $\bigcap_kV^k_{(\hbo)}\wt\cM^\wedge$, where $V^\cbbullet\wt\cM^\wedge$ is the $V$-filtration of $\wt\cM^\wedge$ (defined near~$\hbo$), and we know that each step of the $V$-filtration is integrable (\cite[Prop\ptbl7.4.1]{Bibi01c}).

Let $i_0\in I$ be such that $\varphi_{i_0}=0$. We claim that $\wt\cR^\wedge_{i_0}$ is integrable: because of the previous remark applied to $\wt\cM^\wedge\otimes\cE^{-\varphi_i/\hb}$ for $i\neq i_0$, $(\hb^2\partial_\hb+\varphi_i)\wt\cR^\wedge_{i_0}$---hence $\hb^2\partial_\hb\wt\cR^\wedge_{i_0}$---has no component on the regular part of $\wt\cM^\wedge\otimes\cE^{-\varphi_i/\hb}$, that is, on $\wt\cR^\wedge_i$; thus $\wt\cR^\wedge_{i_0}$ is stable by $\hb^2\partial_\hb$. The same result applies to any $\wt\cR^\wedge_i$, by globally twisting $\wt\cM^\wedge$ by $\cE^{-\varphi_i/\hb}$, hence the lemma.
\end{proof}

By assumption on $\wt\cM$, each $\wt\cR_i^\wedge$ is strictly specializable. Applying \cite[Lemma 7.3.7]{Bibi01c}, we find that $\Psi_x^{\varphi,\beta}\wt\cM\neq0\implique\beta\in\RR$.

\section{Specialization of the new supersymmetric index}
In this section,~$X$ denotes a disc with coordinate $x$ and~$j$ denotes the inclusion of the punctured disc $X^*\defin X\moins\{0\}$ into~$X$.

\subsection{The tame case}
In this subsection, we keep the setting of \S\ref{subsec:speinttame} and we assume that $(\cT,\cS)$ is integrable.

\begin{theoreme}\label{th:limSusytame}
We have the following correspondence between $\Susy$ polynomials:
\[\tag*{(\protect\ref{th:limSusytame})$(*)$}\label{eq:limSusytame}
\lim_{x\to0}\Susy_{\cT_x}(T)=\prod_{\beta\in(-1,0]}\prod_{\ell\geq0}\Susy_{\gr_\ell^\rM\Psi_x^\beta\cT}(T)
\]
\end{theoreme}

\begin{remarque}
If $\cT_{|X^*}$ consists of a polarized variation of Hodge structures of weight~$w$, then $\Susy_{\cT_x}(T)$ is constant (\cf Lemma \ref{lem:susyspHodge} below). In general, however, the eigenvalues of $\cQ_x$ do vary (see the example in \cite[(7.115)]{Hertling01} for instance).
\end{remarque}

\begin{proof}[Proof of Theorem \ref{th:limSusytame}]
For simplicity, we will assume $w=0$ (this can be obtained by a Tate twist $(w/2)$) and that $\cM'=\cM''$ and $\cS=(\id,\id)$. Let us fix $\beta\in(-1,0]$. By assumption, $(\gr_\bbullet^\rM\Psi_x^\beta\cT,\cN)$ is a graded Lefschetz twistor structure which is polarized and of weight~$0$ (\cf \cite[\S2.1.e]{Bibi01c}). It thus corresponds to a Hermitian vector space $H$ with a $\SL_2(\RR)$-action (\cf \cite[Rem\ptbl2.1.15]{Bibi01c}), hence a graded vector space $H=\bigoplus_\ell H_\ell$ with a nilpotent endomorphism of degree $-2$. We denote the standard action of the generators of $\sld(\RR)$ by $\rX,\rY,\rH$, so that $H_\ell$ is the eigenspace of $\rH$ for the eigenvalue~$\ell$. Then, for $\ell\geq0$, a basis $\bme_{\beta,\ell,\ell}^o$ of the primitive subspace $\rP H_\ell$ defines a global frame of $\rP^\beta_\ell\cM\defin \rP\gr_\bbullet^\rM\Psi_x^\beta\cM$, which is orthonormal for $\rP^\beta_\ell\cCS$ (the sesquilinear form of $\rP\gr_\bbullet^\rM\Psi_x^\beta\cT$) and in which the matrix of $\hb^2\partial_\hb$ takes the form $\pcU_{\beta,\ell}-\pcQ_{\beta,\ell}\hb-\pcU_{\beta,\ell}^\dag\hb^2$. We can assume that $\pcQ_{\beta,\ell}$ is diagonal, being selfadjoint with respect to the positive definite Hermitian form on $\rP H_\ell$.

The construction done in \cite[\S5.4.c]{Bibi01c} extends this family of frames first to a frame $\bme_{\beta,\ell,k}^o$ ($\ell\in\NN$, $k=0,\dots,\ell$) of $\gr_{\ell-2k}^\rM\Psi_x^\beta\cM$ and then to a local frame $\bme_\beta$ of $V^\beta\cM$ (the local construction near each~$\hbo$ done in \loccit is not needed here as the $V$-filtration is globally defined with respect to~$\hb$, \cf \cite[Rem\ptbl3.3.6(2)]{Bibi01c}).

The action of $\hb^2\partial_\hb$ leaves the $V$-filtration invariant (\cf \cite[Prop\ptbl7.3.1]{Bibi01c}), as well as the lift of the $\rM$\nobreakdash-filtration on each $V^\beta$ (\cf \cite[Lemma 7.3.8]{Bibi01c}). Therefore, the matrix~$B$ of $\hb^2\partial_\hb$ in the frame $\bme$, which is holomorphic, is ``triangular'' up to powers of~$x$ with respect to $\rM_\bbullet V^\cbbullet$, \ie can be written as
\begin{equation}\label{eq:B}
B=\tbigoplus_{\beta}\Big[B_{\beta,\beta,0}\oplus B_{\beta,\beta,\leq-1}\oplus\tbigoplus_{\beta'\neq\beta}B_{\beta',\beta}\Big]
\end{equation}
with $B_{\beta',\beta}/x$ holomorphic if $\beta'<\beta$, and where the index $j$ in $B_{\beta,\beta,j}$ denotes the weight with respect to $\rH$, so that $[\rH,B_{\beta,\beta,j}]=jB_{\beta,\beta,j}$. Moreover, the matrix $B_{\beta,\beta,0}$ can be written as $\cU_{\beta,\beta,0}-\cQ_{\beta,\beta,0}\hb-\cU_{\beta,\beta,0}^\dag\hb^2$, and is block-diagonal with respect to the previous decomposition $(\ell,k)$ of the frame~$\bme_\beta$, and the diagonal $(\ell,k)$-block of $\cQ_{\beta,\beta,0}$ is $\pcQ_{\beta,\ell}+(-k+\nobreak\ell/2)\id$. In particular, the characteristic polynomial of $\bigoplus_{\beta\in(-1,0]}\cQ_{\beta,\beta,0}$ is the right-hand side in \ref{eq:limSusytame}.

Let us denote by $A(x,\hb)$ the matrix $\bigoplus_{\beta\in(-1,0]} \mx^\beta\Lx^{\rH/2}$, with $\Lx\defin\vert\log\mx^2\vert$. By \cite[Lemma 5.4.7]{Bibi01c},\footnote{In \loccit, the matrix $A$ is multiplied by $e^{-\hb X}$; this is in fact not needed in the argument.} there exists on $X^*\times\nobreak\Omega_0$ (up to shrinking~$X$ and~$\Omega_0$, as defined in \S\ref{subsec:inttwst}) a matrix $S(x,\hb)$ with $\lim_{x\to0}S(x,\hb)=0$ uniformly with respect to~$\hb$, such that the frame
\[
\varepsilong\defin\bme\cdot A(x,\hb)^{-1}(\id+S(x,\hb))
\]
is an orthonormal frame for $\cCS$. The matrix of $\hb^2\partial_\hb$ in this frame will enable us to compute the left-hand side in the theorem.

This matrix is equal to
\[
(\id+S)^{-1}ABA^{-1}(\id+S)+(\id+S)^{-1}A\hb^2\partial_\hb[A^{-1}(\id+S)].
\]
The second term is a multiple of $\hb^2$ and will not contribute to $\cQ_x$.

Let us note that the block $(ABA^{-1})_{\beta',\beta}$ ($\beta'\neq\beta$) is equal to
\begin{equation}
\mx^{\beta'-\beta}\Lx^{\rH/2}B_{\beta',\beta}\Lx^{-\rH/2},
\end{equation}
and tends to~$0$ when $x\to0$ (since, when $\beta'<\beta$ and $\beta,\beta'\in(-1,0]$, $B_{\beta',\beta}/x$ is locally bounded and $1+\beta'-\beta>0$). Then so does its conjugate by $\id+S$.

A similar reasoning can be done for $A_\beta B_{\beta,\beta,\leq-1}A_\beta^{-1}$, which gives a decay at least like $\Lx^{-1/2}$ when $x\to0$.

Now, $A_\beta B_{\beta,\beta,0}A_\beta^{-1}=B_{\beta,\beta,0}$, and the coefficient of $-\hb$ is $\cQ_{\beta,\beta,0}$, which is thus equal to $\lim_{x\to0}\cQ_x$. This gives the conclusion.
\end{proof}

\subsection{The wild case}

We now consider the setting of \S\ref{subsec:speintwild}. We then have the following generalization of Theorem \ref{th:limSusytame}:

\begin{theoreme}\label{th:limSusywild}
Let $\cT=(\cM',\cM'',\cCS)$ be a wild twistor $\cD$-module of weight~$w$ polarized by~$\cS$, satisfying the ``no ramification'' condition. We have the following correspondence between $\Susy$ polynomials:
\[\tag*{(\protect\ref{th:limSusywild})$(*)$}\label{eq:limSusywild}
\lim_{x\to0}\Susy_{\cT_x}(T)=\prod_{\varphi\in\xm\CC[\xm]}\prod_{\beta\in(-1,0]}\prod_{\ell\geq0}\Susy_{\gr_\ell^\rM\Psi_x^{\varphi,\beta}\cT}(T).
\]
\end{theoreme}

\begin{proof}
As in the proof of Theorem \ref{th:limSusytame}, we will assume $w=0$, $\cM'=\cM''$ and $\cS=(\id,\id)$. We will make an extensive use of \cite[\S\S5.2\,\&\,5.4]{Bibi06b}. As in the tame case, we start with a frame $\bme^o_{\varphi,\beta,\ell}$ of $\rP_\ell^{\varphi,\beta}\cM\defin \rP\gr^\rM_\ell\Psi_x^{\varphi,\beta}\cM$ which is orthonormal for $\rP^{\varphi,\beta}_\ell\cCS$, for any $\varphi,\beta,\ell$. The matrix of $\hb^2\partial_\hb$ in this frame takes the form $\pcU_{\varphi,\beta,\ell}-\pcQ_{\varphi,\beta,\ell}\hb-\pcU_{\varphi,\beta,\ell}^\dag\hb^2$. The constructions of \loccit produce a frame~$\wt\varepsilong$ of $\wt\cM_{|X^*}$ which is orthonormal with respect to $\cCS$ (\cf \cite[Cor\ptbl5.4.3]{Bibi06b}), and we wish to compute the matrix of $\hb^2\partial_\hb$ in this frame. Let us recall the steps going from $\bme$ to~$\wt\varepsilong$.

\begin{enumerate}
\item\label{enum:wildstep1}
We first lift, exactly as in the tame case, each of the frames $\bme^o_{\varphi_i,\beta}$ to a frame $(\wh\bme_{\varphi_i,\beta})_\beta$ of $\wt\cR_i^\wedge$. Arguing as in the tame case, the matrix $\wh B$ of $\hb^2\partial_\hb$ in the frame~$\wh\bme$ takes the form $\bigoplus_i\wh B_{ii}$, where $\wh B_{ii}$ decomposes as in \eqref{eq:B}. As remarked in \S\ref{subsec:speintwild} (after (DEC$^\wedge$)), we can assume that, when restricted to $\hb=0$, the frame $\wh\bme_{|\hb=0}$ is a frame of $\wt\cM/\hb\wt\cM$, and is compatible with the corresponding $\varphi$-decomposition, so $\wh B_{ii}(x,0)$ is convergent.

\item\label{enum:wildstep2}
We then work locally with respect to~$\hbo$ and in small sectors in the variable~$x$. Let $\pi:Y\to X$ be the real oriented blow up of~$X$ at the origin, with $S^1=\pi^{-1}(0)$, and let us set $\cY=Y\times\nobreak\Omega_0$. Let us denote by $\cA_{\cY}$ the sheaf of $C^\infty$ functions on $\cY$ which are holomorphic with respect to~$\hb$ and holomorphic on $X^*\times\nobreak\Omega_0$. We then lift the frame $\wh\bme$ to a $\cA_{\cY,\xi_o,\hbo}$-frame, for any $\xi_o\in S^1$ and $\hbo\in\Omega_0$, and we get frames $\bmea^{(\xi_o,\hbo)}=(\bmea^{(\xi_o,\hbo)}_{\varphi_i})_i$. We can assume that, when restricted to $\hb=0$, the frame $\bmea^{(\xi_o,0)}_{|\hb=0}$ comes from a frame of $\wt\cM/\hb\wt\cM$ compatible with the $\varphi$-decomposition. The matrix $\sAB^{(\xi_o,\hbo)}$ of $\hb^2\partial_\hb$ satisfies the following properties (according to \cite[Lemma 5.2.6]{Bibi06b}):
\begin{enumerate}
\item
if $i,j\in I$ are distinct, the term $\sAB^{(\xi_o,\hbo)}_{ij}$ is infinitely flat along $S^1\times\nobreak\Omega_0$ in a neighbourhood of $(\xi_o,\hbo)$ and, for $\hbo=0$, $\sAB^{(\xi_o,0)}_{ij}(x,0)\equiv0$,
\item
for any $i\in I$, the term $\sAB^{(\xi_o,\hbo)}_{ii}$ has an asymptotic expansion equal to $\wh B_{ii}$ when $x\to0$ near the direction $\xi_o$, uniformly with respect to $\hb\in\nb(\hbo)$ and, for $\hbo=0$, $\sAB^{(\xi_o,0)}_{ii}(x,0)$ does not depend on $\xi_o$ and is holomorphic with respect to $x$ (it takes the form \eqref{eq:B} at $\hb=0$).
\end{enumerate}
It is then clear (after the tame case) that the limit, when $x\to0$ in the neighbourhood of the direction~$\xi_o$ and $\hb\in\nb(\hbo)$, of the characteristic polynomial of the coefficient of $-\hb$ in $\sAB^{(\xi_o,\hbo)}$ is equal to the RHS in~\ref{eq:limSusywild}.
\item\label{enum:wildstep3}
We now define the local untwisted $C^\infty$ frame $\varepsilong^{(\xi_o,\hbo)}=\bmea^{(\xi_o,\hbo)}\cdot A^{-1}(x,\hb)$, where $A=\bigoplus_iA_{ii}$ and each $A_{ii}$ is as in the tame case. Let $\infB^{(\xi_o,\hbo)}$ be the matrix of $\hb^2\partial_\hb$ in this frame. The non-diagonal blocks $\infB^{(\xi_o,\hbo)}_{ij}$ for $i\neq j$ remain infinitely flat when $x\to0$ in the direction $\xi_o$, as $A$ and $A^{-1}$ have moderate growth. Moreover, the~$\hb$-constant term $\infB^{(\xi_o,0)}(x,0)$ of $\infB^{(\xi_o,0)}(x,\hb)$ still satisfies $\infB^{(\xi_o,0)}(x,0)_{ij}=0$ if $i\neq j$, as $A$ is diagonal with respect to the $\varphi$-decomposition. Moreover, as in the tame case, $\infB^{(\xi_o,0)}(x,0)_{ii}(x,0)$ has a limit when $x\to0$. Then the same argument as in the tame case shows that the limit of the characteristic polynomial of the coefficient of~$-\hb$ in~$\infB^{(\xi_o,\hbo)}$ is the same as for $\sAB^{(\xi_o,\hbo)}$, hence is equal to the RHS in~\ref{eq:limSusywild}.
\item\label{enum:wildstep4}
We globalize the construction, by using a partition of unity with respect to $\xi_o$ and by using the argument of \cite[lemma 5.4.6]{Bibi01c} (\cf \cite[Lemma 5.2.11]{Bibi06b}), to get a frame $\varepsilong$. The base change from any $\varepsilong^{(\xi_o,\hbo)}$ to $\varepsilong$ takes the form $\id+R^{(\xi_o,\hbo)}(x,\hb)$, with $R^{(\xi_o,\hbo)}$ satisfying $\lim_{x\to0}\Lx^\delta R^{(\xi_o,\hbo)}=0$ uniformly with respect to $\hb\in\nb(\hbo)$, for some $\delta>0$. We also note that we can achieve $R^{(\xi_0,0)}(x,0)\equiv0$ in the base change, as the frame $\bmea^{(\xi_o,0)}_{|\hb=0}$ is already globally defined with respect to~$\xi$, and so does the frame $\varepsilong^{(\xi_o,0)}_{|\hb=0}$; moreover, the argument of \cite[lemma 5.4.6]{Bibi01c} gives a contribution equal to $\id$ at $\hbo=0$ for the base change. Therefore, the conclusion of \eqref{enum:wildstep3} holds for the matrix $\infB$ of $\hb^2\partial_\hb$ in the frame $\varepsilong$.
\item\label{enum:wildstep5}
Now, the base change from $\varepsilong$ to~$\wt\varepsilong$ given by \cite[Prop\ptbl5.4.1 and (5.3.2)]{Bibi06b} takes the form
\[
\wt\varepsilong=\varepsilong\cdot(\id+S'(x,\hb))^{-1}(\id+U_0(x))^{-1}\diag(e^{\hb\ov{\varphi_i}}\id),
\]
where $S'(x,\hb)$ is continuous and holomorphic with respect to~$\hb$ on $X^*\times\nobreak\nb(\{\hb\leq\nobreak1\})$, and satisfies $S'(x,0)\equiv0$, and $U_0(x)$ is continuous with respect to $x\in X$, $U_0(0)=0$, and $U_0$ is diagonal with respect to the $\varphi$-decomposition. As we are only interested in the coefficient of $-\hb$ in the matrix $\wt{\infB}$ of $\hb^2\partial_\hb$ in the frame~$\wt\varepsilong$, and as the matrix of the base change is holomorphic with respect to~$\hb$, it is enough to consider the corresponding coefficients in the conjugate matrix
\begin{multline*}
\diag(e^{-\hb\ov{\varphi_i}}\id)(\id+U_0(x))(\id+S'(x,\hb))\cdot \infB(x,\hb)\\
{}\cdot(\id+S'(x,\hb))^{-1}(\id+U_0(x))^{-1}\diag(e^{\hb\ov{\varphi_i}}\id).
\end{multline*}
Let us set $\infB(x,\hb)=\infB^{(0)}(x)-\hb\infB^{(1)}(x)+\cdots$ and $S'(x,\hb)=\hb S^{\prime(1)}(x)+\cdots$. On the one hand, we know that $\infB^{(0)}(x)$ is diagonal with respect to the $\varphi$-decomposition and has a limit when $x\to0$, and the limit when $x\to0$ of the characteristic polynomial of $\infB^{(1)}$ is the RHS in~\ref{eq:limSusywild}.

On the other hand, $S'$ defines a continuous map $X\to \Mat_d(L^2(\bS))$, and, as such, $S'(0)=0$ (\cf \cite[Prop\ptbl5.4.1]{Bibi06b}). Therefore, the (Fourier) coefficient $S^{\prime(1)}(x)$ is a continuous function of $x$ and has limit~$0$ when $x\to0$. We thus have
\begin{multline*}
(\id+U_0(x))(\id+S'(x,\hb))\cdot \infB\cdot(\id+S'(x,\hb))^{-1}(\id+U_0(x))^{-1}\\\hspace*{-2cm}=(\id+U_0(x))\infB^{(0)}(\id+U_0(x))^{-1}\\
-\hb\cdot(\id+U_0(x))\big(\infB^{(1)}+[\infB^{(0)},S^{\prime(1)}]\big)(\id+U_0(x))^{-1}+\cdots
\end{multline*}
As the~$\hb$-constant term above is diagonal with respect to the $\varphi$-decomposition, it commutes with $\diag(e^{\hb\ov{\varphi_i}}\id)$ and therefore is not altered by the conjugation by this matrix. It follows that the coefficient of $-\hb$ in $\wt{\infB}$ is
\[
(\id+U_0(x))\big(\infB^{(1)}+[\infB^{(0)},S^{\prime(1)}]\big)(\id+U_0(x))^{-1}.
\]
As $\lim_{x\to0}[\infB^{(0)},S^{\prime(1)}]=0$, the limit, when $x\to0$, of its characteristic polynomial (that is, the LHS in~\ref{eq:limSusywild}), is thus equal to the limit, when $x\to0$, of the characteristic polynomial of $\infB^{(1)}(x)$, which we know to be the RHS in~\ref{eq:limSusywild}.\qedhere
\end{enumerate}
\end{proof}

\section{A review on exponential twist and Fourier-Laplace transform}\label{sec:exptwFL}

In this section, we review some results of~\cite{Bibi05}. The base manifold~$X$ will be $\PP^1$ with its two affine charts having coordinates $t$ and~$t'$. We will denote by $\cP^1$ the corresponding manifold $\cX$ as in the notation of \S\ref{subsec:vartw}.

\Subsection{De~Rham cohomology with exponential twist for twistor $\cD$-modules}\label{subsec:exptw}
Although we do not gain much by simply attaching a polarized variation of twistor structure to a polarized variation of Hodge structure, the advantage is clearer when we apply an exponential twist and integrate.

\subsubsections{De~Rham cohomology with exponential twist}
Let $\ccM$ be a $\cD_{\PP^1}$-module and $\wt\ccM=\ccM(*\infty)$. The exponentially twisted de~Rham cohomology is the hypercohomology on~$\PP^1$ of the complex
\[
\DR(\ccM\otimes\ccE^{-t})\defin\{0\to\wt\ccM\To{\nabla-dt}\wt\ccM\to0\},
\]
that we denote $H^*_{\DR}(\PP^1,\ccM\otimes\ccE^{-t})$. If we assume $\ccM$ to be $\cD_{\PP^1}$-holonomic, then $M\defin\Gamma(\PP^1,\wt\ccM)$ is a holonomic $\Clt$-module and the previous hypercohomology is the cohomology of the complex
\begin{equation}\label{eq:dRexptwist}
0\to M\To{\partial_t-1}M\to0
\end{equation}
and has cohomology in degree one only, this cohomology being a finite dimensional $\CC$-vector space. Its dimension is computed in \cite[Prop\ptbl1.5, p\ptbl79]{Malgrange91}. If $\ccM$ has a regular singularity at infinity, this dimension is equal to the sum (over the singular points at finite distance) of the dimension of vanishing cycles of $\DR\ccM$.

\subsubsections{Exponential twist of a twistor $\cD$-module}
We will use the notation of \S\ref{subsec:vartw}. Let~$\cM$ be a left $\cR_{\cP^1}$-module ($\cP^1=\PP^1\times\nobreak\Omega_0$). We denote by $\wt\cM$ the localized module $\cR_{\cP^1}(*\infty)\otimes_{\cR_{\cP^1}}\cM$. We set $\cE^{-t/\hb}=\cO_{\cP^1}(*\infty)$ with~$\hb$-connection $\hb d-dt$. The exponentially twisted $\cR_{\cP^1}$-module $\FcM$ is $\cE^{-t/\hb}\otimes_{\cO_{\cP^1}(*\infty)}\wt\cM$ equipped with its natural~$\hb$-connection.

It is useful to introduce the category $\wtRTriples(\PP^1)$, whose objects $(\wt\cM',\wt\cM'',\wt\cCS)$ consist of $\cR_{\cP^1}(*\infty)$-modules with a pairing taking values in the sheaf of distributions on $(\PP^1\moins\{\infty\})\times\nobreak \bS$ which have moderate growth at $\{\infty\}\times\nobreak\bS$ (\ie which can be extended as distributions on $\PP^1\times\nobreak\bS$) and depend continuously on $\hb\in\bS$.

If we remark that, for $\hb\in\bS$, the $C^\infty$ function $e^{-t/\hb}\cdot\ovv{e^{-t/\hb}}=e^{\hb\ov t-t/\hb}$ has moderate growth as well as all its derivatives with respect to $t,\ov t$ when $t\to\infty$, we can associate to an object $\cT=(\cM',\cM'',\cCS)$ of $\RTriples(\PP^1)$ the object $\FcT=(\FcM',\FcM'',e^{\hb\ov t-t/\hb}\cCS)$ of $\wtRTriples(\PP^1)$.

Let now $\cT$ be a polarized twistor $\cD$-module on $\PP^1$. Then the previous construction can be refined to give an object $\FcT=(\FcM',\FcM'',{}^F\!\cCS)$ of $\RTriples(\PP^1)$. The regularization ${}^F\!\cCS$ of $e^{\hb\ov t-t/\hb}\cCS$ is obtained by specializing $\cFCS$, whose construction is recalled in \S\ref{subsec:Fourier-Laplace}, at $\tau=1$. Let $a$ be the constant map on $\PP^1$. The following is proved in \cite{Bibi04} (and its erratum):

\begin{theoreme}[Exponentially twisted Hodge theorem]\label{th:exptwHodge}
If $(\cT,\cS)$ is a polarized regular twistor $\cD$-module of weight~$w$ on $\PP^1$, then $\cH^0a_+\FcT$ is a polarized twistor structure of weight~$w$.\qed
\end{theoreme}

\subsection{Fourier-Laplace transform (\cf \cite[Appendix]{Bibi01c})}\label{subsec:Fourier-Laplace}
We continue to work with the projective line $\PP^1$ equipped with its two charts having coordinates $t$ and~$t'$, and we consider another copy of it, denoted by $\wh\PP^1$, having coordinates $\tau,\tau'$. We will set $\infty=\{t'=0\}$ and $\wh\infty=\{\tau'=0\}$. We consider the diagram
\begin{equation}\label{eq:P1whP1}
\begin{array}{c}\xymatrix{
&\PP^1\times\wh\PP^1\ar[ld]_p\ar[rd]^{\wh p}&\\
\PP^1&&\wh\PP^1
}\end{array}
\end{equation}
Let $\cM$ be a good $\cR_{\cP^1}$-module (in the sense of \cite[\S1.1.c]{Bibi01c}). We set $\cFcM\defin p^+\wt\cM(*\wh\infty)\otimes\nobreak\ccE^{-t\tau/\hb}$ (\cf\cite[\S A.2]{Bibi01c}). We know (\cf\cite[Prop\ptbl A.2.7]{Bibi01c}) that $\cFcM$ is a good $\cR_\cZ(*\wh\infty)$-module, where $Z=\PP^1\times\nobreak\wh\PP^1$ and $\cZ=Z\times\nobreak\Omega_0$. Taking direct images, $\wh p_+\cFcM$ is a coherent $\cR_{\wh\cP^1}(*\wh\infty)$-module.

Let $\cT=(\cM',\cM'',\cCS)$ be an object of $\RTriples(\PP^1)$, such that $\cM',\cM''$ are $\cR_{\cP^1}$-good. Then $\cFcT$ is defined as $(\cFcM',\cFcM'',\cFCS)$, where $\cFcM',\cFcM''$ are as above and $\cFCS$ is defined in \cite[p\ptbl196]{Bibi01c} (note that the twist for $\cCS$ needs some care). The fibre at $\tau=1$ (suitably defined as nearby cycles) of $\cFcT$ is identified with $\FcT$. The Fourier-Laplace transform $\wh\cT$ of $\cT$ is defined as the direct image of $\cFcT$ by $\wh p$.

Let us assume that $\cT$ is integrable. Then (\cf \cite[Rems\ptbl A.2.9 \& A.2.15]{Bibi01c}), $\cFcT$ is also integrable.

\begin{lemme}\label{lem:hbdhbtdt}
The action of $\hb^2\partial_\hb$ on $\cFcM=p^+\wt\cM(*\wh\infty)\otimes\cE^{-t\tau/\hb}$ satisfies, for any local section of~$\wt\cM$,
\[
(\hb^2\partial_\hb+\tau\partiall_\tau)(m\otimes\cE^{-t\tau/\hb})=(\hb^2\partial_\hb m)\otimes\cE^{-t\tau/\hb}.
\]
\end{lemme}

\begin{proof}
This directly follows from the definition of the actions (\cf\cite[A.2.2 \& A.2.3]{Bibi01c}).
\end{proof}

Let us also notice that one gets a similar relation with the coordinate~$\tau'$ by using the relation $\tau'\partiall_{\tau'}=-\tau\partiall_\tau$.

\subsection{Fourier-Laplace transformation of twistor $\cD$-modules}\label{subsec:localFL}
Let $(\cT,\cS)$ be a polarized regular twistor $\cD$-module of weight~$w$ (in the sense of \cite{Bibi01c} or \cite{Mochizuki07}) on $\PP^1$. The Fourier-Laplace transform $(\wh\cT,\wh\cS)$ is an object of the same kind on the \emph{analytic} affine line $\Afuh$ with coordinate~$\tau$, after \cite{Bibi04} and \cite{Bibi05b}. Moreover, it is smooth on the punctured line $\Afuh\moins\{\tau=0\}$, and its restriction at $\tau=1$ is naturally identified with $\cH^0a_+\FcT$.

In \cite[Cor\ptbl5.20 \& Prop\ptbl5.23]{Bibi05b}, we also show an ``inverse stationary phase formula'' computing the nearby and vanishing cycles of $\wh\cT$ at $\tau=0$ in terms of the nearby cycles at $t'=0$ of $\cT$.

Let us moreover assume that $(\cT,\cS)$ is integrable. Recall that $a$ denotes the constant map $\PP^1\to\pt$. The basic comparison result \cite[Cor\ptbl 5.20 and Prop\ptbl5.23]{Bibi05b}, together with Lemma \ref{lem:hbdhbtdt}, gives (\cf \S\ref{subsec:speinttame} for the notation):

\begin{proposition}\label{prop:psiwtTT}
We have natural isomorphisms of integrable polarized pure twistor structures of weight $w+\ell$ ($\ell\in\ZZ$, $\beta\in(-1,0)$ for the first line):
\begin{equation}\tag*{(\protect\ref{prop:psiwtTT})$(*)$}\label{eq:psiwtTT}
\begin{split}
(\gr_\ell^\rM\Psi_\tau^\beta\wh\cT,\hb^2\partial_\hb+\beta\hb)&\simeq(\gr_\ell^\rM\Psi_{t'}^\beta\cT,\hb^2\partial_\hb),\\
(\gr_\ell^\rM\phi_\tau^{-1} \wh\cT,\hb^2\partial_\hb-\hb)&\simeq(\gr_\ell^\rM\Psi_{t'}^0\cT,\hb^2\partial_\hb)
\end{split}
\end{equation}
and we also have
\begin{equation}\tag*{(\protect\ref{prop:psiwtTT})$(**)$}\label{eq:psiwtTTP}
(\rP\gr_0^\rM\Psi_\tau^0\wh\cT,\hb^2\partial_\hb)\simeq(\cH^0a_+\cT,\hb^2\partial_\hb).
\end{equation}
\end{proposition}

In Appendix \ref{sec:appendiceA} (Theorem \ref{th:FourierlocalT}), we show that the Fourier-Laplace transform $\wh\cT$ on $\Afuh$ naturally extends as a wild twistor $\cD$-module (in the sense of \cite{Bibi06b}, \cf also \cite{Mochizuki08}) near $\wh\infty\in\wh\PP^1$ and we relate the corresponding nearby cycles with the vanishing cycles of $\cT$ at its critical points (``stationary phase formula''):

\begin{corollaire}[of Theorem \ref{th:FourierlocalT}, \eqref{eq:PsiintMicro} and \eqref{eq:phiintMicro}]\label{cor:psiwtTT}
For any $c\in\nobreak\CC$, we have natural isomorphisms of integrable polarized pure twistor structures of weight $w+\ell$ ($\ell\in\ZZ$, $\beta\in(-1,0)$ for the first line):
\begin{equation}\tag*{(\protect\ref{cor:psiwtTT})$(*)$}\label{eq:psiwtTTinf}
\begin{split}
(\gr_\ell^\rM\Psi_{\tau'}^{c/\tau',\beta}\wh\cT,\hb^2\partial_\hb-(\beta+1)\hb)&\simeq (\gr_\ell^\rM\Psi_{t+c}^\beta\cT,\hb^2\partial_\hb),\\
(\gr_\ell^\rM\Psi_{\tau'}^{c/\tau',0}\wh\cT,\hb^2\partial_\hb-\hb)&\simeq (\gr_\ell^\rM\phi_{t+c}^{-1}\cT,\hb^2\partial_\hb-\hb).
\end{split}
\end{equation}
\end{corollaire}

\section{Twistor structures and Hodge structures}

In this section, we make explicit the functor $\Tw$ which associates to any polarized complex Hodge structure (\resp variation of Hodge structure, \resp polarized complex mixed Hodge structure) an integrable polarized twistor structure (\resp ...). In this section, $Y$ will denote a complex manifold,~$X$ will denote a disc with coordinate~$x$ and $X^*$ will denote the punctured disc $X\moins\{0\}$.

\subsection{The integrable variation attached to a polarizable variation of Hodge structures}\label{subsec:VTSVHS}

Let $(V,\Vnablaf)$ be a holomorphic vector bundle with an integrable holomorphic connection on a complex manifold $Y$. Let us assume that $(V,\Vnablaf)$ underlies a polarized variation of Hodge structures of weight~$w$. The $C^\infty$-bundle $H$ associated to $V$ comes equipped with a flat $C^\infty$ connection $D=\Vnablaf+d''$ and a decomposition $H=\oplus_pH^{p,w-p}$ indexed by integers. There is a $D$-flat sesquilinear pairing~$k$ on~$H$ such that the decomposition is $k$-orthogonal, and the sesquilinear pairing $h$ such that the decomposition is $h$-orthogonal and $h=i^{-w}(-1)^pk$ on $H^{p,w-p}$ is Hermitian positive definite. As usual, we set $F^pV=\bigoplus_{r\geq p}H^{r,w-r}$.

For any $j\in\hZZ$, the Tate twist is defined as
\[
(V,\Vnablaf,F^\cbbullet V,k,w)(j)\defin(V,\Vnablaf,F^\cbbullet V,i^{-2j}k,w-2j).
\]

We denote by $\cH'=R_{F[w]}V$, $\cH''=R_FV$, the Rees modules associated to $F[w]^\cbbullet V\defin F^{w+\cbbullet}V$ and $F^\cbbullet V$, that is:
\begin{equation}\label{eq:HF}
\begin{split}
\cH'&\defin\tbigoplus_pF[w]^p\hb^{-p}=\tbigoplus_r\hb^{w-r}H^{r,w-r}[\hb],\\
\cH''&\defin\tbigoplus_pF^p\hb^{-p}=\tbigoplus_r\hb^{-r}H^{r,w-r}[\hb].
\end{split}
\end{equation}
We denote by $R_Fk$ the map naturally induced by $k$ on $R_{F[w]}V\otimes_{\CC[\hb,\hbm]}\ovv{R_FV}$ with values in $\cC^\infty_Y[\hb,\hbm]$. We associate to this variation the triple $\cT=(R_{F[w]}V,R_FV,R_Fk)$. We set $\cS=(S',S'')$ with $S',S'':\cH''\to\cH'$, $S''$ is the multiplication by $\hb^w$ and $S'$ by $(-\hb)^w$.

The integrable connection $\nabla$ is defined as $\Vnablaf+d'_\hb$. We note that $R_{F'}V$ and $R_{F''}V$ are stable by $\hb\nabla_{\partial_\hb}$ (reflecting the fact that $\cU=0$). In particular, the action of $\hb^2\partial_\hb$ enables one to recover the grading of $R_FV$, hence the filtration $F^\cbbullet V$. The following is easy:

\begin{lemme}\label{lem:vhstw}
The object
\[
\Tw(V,F^\cbbullet V,k,w)\defin\big(\cT=(R_{F[w]}V,R_FV,R_Fk),\cS,\hb^2\partial_\hb\big)
\]
is an integrable polarized variation of twistor structures of weight~$w$.\qed
\end{lemme}

\begin{remarque}
According to the convention made in Remark \ref{rem:tateint}, the functor $\Tw$ is compatible with Tate twist (that is, $[\Tw(V,F^\cbbullet V,k,w)](j)$ is canonically isomorphic to $\Tw[(V,F^\cbbullet V,k,w)(j)]$.
\end{remarque}

\begin{lemme}\label{lem:susyspHodge}
Let $(\cT,\cS)$ be the integrable polarized twistor structure of weight~$w$ attached to a polarized Hodge structure of weight~$w$ (in particular, the ``no ramification'' condition is fulfilled). Then
\[
\SP^\infty_\cT(T)=\Susy_\cT(T)=\SP^0_\cT(T).
\]
\end{lemme}

\begin{proof}
According to \eqref{eq:tateSusySp} one can assume $w=0$. On the one hand, $\cU=0$, so $\cQ$ is conjugate to the opposite of the residue at $\hb=0$ of $\partial_\hb$ acting on $R_FV$. As $\hb\partial_\hb$ acts as $-p\id$ on $F^p\hb^{-p}$, we find that $\Susy_\cT(T)=\prod_p(T-p)^{\dim \gr_F^p}$.

We now have $G=\CC[\hb,\hbm]\otimes_\CC H$ and $V^pG=\hb^{-p}\CC[\hbm]\otimes_\CC H$ (for the $V$-filtration at $\hb=\infty$). Then \eqref{eq:HF} shows that (using the notation in Definition \ref{def:spinfty}) $\nu_p=\dim H^{p,-p}=\dim \gr_F^p$, hence the first equality. On the other hand, the $V$-filtration at $\hb=0$ is given by $V^pG=\hb^p\CC[\hb]\otimes_\CC H$ and $G$ has a regular singularity at~$\hb=0$, so there is no nontrivial exponential term in the decomposition \eqref{eq:Levelt-Turrittin}. We then have (using the notation in Definition \ref{def:sporigin}) $\mu_{0,p}=\dim H^{-p,p}$, hence the second equality.
\end{proof}

\Subsection{Integrable twistor structure attached to a polarized complex mixed Hodge structure}
Let $V_o$ be a complex vector space equipped with a filtration $F^\cbbullet V_o$, a nilpotent endomorphism $\rN_o$ and a sesquilinear pairing $k_o$. We denote by $\rM_\bbullet$ the monodromy filtration of $V_o$ associated to $\rN_o$. Let $w\in\ZZ$. We say (\cf \cite{Schmid73,K-K87}) that $(V_o,F^\cbbullet V_o,k_o,\rN_o)$ is a polarized complex mixed Hodge structure of weight~$w$ if the following conditions are fulfilled:
\begin{enumerate}
\item
$k_o$ is $(-1)^w$-Hermitian and $\rN_o$ is skew-adjoint with respect to $k_o$,
\item
$\rN_oF^\cbbullet V_o\subset F^{\cbbullet-1} V_o$,
\item
if we set $\ov F{}^pV_o=\ov{(F^{w-p+1}V_o)^\perp}$ (which also satisfies $N\ov F{}^pV_o\subset\ov F{}^{p-1}V_o$), then $(F^\cbbullet V_o,\ov F{}^\cbbullet V_o,\rM_\cbbullet)$ is a mixed Hodge structure of weight~$w$,
\item
the object $(\rP\gr_\ell^\rM V_o,F^\cbbullet\rP\gr_\ell^\rM V_o,k_o(\cbbullet,\ov{\rN_o^\ell\cbbullet}))$ is a polarized complex Hodge structure of weight $w+\ell$.
\end{enumerate}

\begin{remarque}[{\cf \cite[Lemma 2.8]{MSaito89}}]\label{rem:Ipq}
If $(V_o,F^\cbbullet V_o,k_o,\rN_o)$ is a polarized complex mixed Hodge structure of weight~$w$, then there exists an increasing filtration $\wt F_\bbullet V_o$ which is opposite to $F^\cbbullet V_o$ (\ie $V_o$ decomposes as $\bigoplus_p F^p\cap\wt F_p$) and which satisfies $\rN_o\wt F_\cbbullet V_o\subset \wt F{}_{\cbbullet-1} V_o$ (in particular, $\wt F{}_\cbbullet V_o$ is stable by $\rN_o$). Indeed, $V_o$ is bigraded by Deligne's $I^{p,q}$ (\cf \cite{DeligneHII}) with
\[
I^{p,q}=(F^p\cap W_{p+q})\cap\Big(\ov F{}^q\cap W_{p+q}+\textstyle\sum_{j\geq1}\ov F{}^{q-j}\cap W_{p+q-j-1}\Big),
\]
where $W_\ell\defin \rM_{w+\ell}$, and $F^pV_o=\bigoplus_{p'\geq p}\bigoplus_qI^{p',q}$. We can set $\wt F{}_pV_o=\bigoplus_{p'\leq p}\bigoplus_qI^{p',q}$.
\end{remarque}

\begin{definition}\label{def:twnilporb}
For a polarized complex mixed Hodge structure $(V_o,F^\cbbullet,k_o,\rN_o)$ of weight~$w$, we set $\Tw(V_o,F^\cbbullet,k_o,\rN_o)\defin(\cT,\cS,\cN,\hb^2\partial_\hb)$ with
\begin{enumerate}
\item
$\cT=(R_{F[w]}V_o,R_FV_o,R_Fk_o)$ (an object of $\RTriples(\pt)$),
\item
$\cS=((-\hb)^w,\hb^w)$ (a sesquilinear duality of $\cT$ of weight~$w$),
\item
$\cN:\cT\to\cT(-1)$ defined as $\cN=(\hb \rN_o,-\hb \rN_o)$,
\item
$\hb^2\partial_\hb$ is the natural derivation on $R_{F[w]}V_o,R_FV_o$.
\end{enumerate}
\end{definition}

\begin{lemme}
If $(V_o,F^\cbbullet,k_o,\rN_o)$ is a polarized complex mixed Hodge structure of weight~$w$, the monodromy filtration of $\hb \rN_o$ on $R_FV_o$ is such that $\gr_\ell^{\rM(\hb \rN_o)}R_FV_o=R_F\gr_\ell^{\rM(\rN_o)}V_o$. Moreover, the object $(\gr_\cbbullet^\rM\cT,\gr_\cbbullet^\rM\cS)$ is a graded Lefschetz twistor structure of weight~$w$ (\cf \cite[\S2.1.e]{Bibi01c}). Last, we have a canonical isomorphism of objects of weight $w+\ell$ ($\ell\geq0$):
\[
\rP\gr_\ell^\rM\Tw(V_o,F^\cbbullet,k_o,\rN_o)\isom\Tw[\rP\gr_\ell^\rM(V_o,F^\cbbullet,k_o,\rN_o)].
\]
\end{lemme}

\begin{proof}
Let us indicate the proof for the last part. We can reduce to weight~$0$ by twisting by $w/2$, and also to $\cS=(\id,\id)$. The left-hand side in the formula is by definition (\cf \cite[Example 2.1.14]{Bibi01c}) given by
\[
\cT_\ell=((\hb \rN_o)^\ell R_F\rP\gr_\ell^\rM V_o,R_F\rP\gr_\ell^\rM V_o,R_Fk_o),\quad\cS_\ell=((\hb \rN_o)^\ell,(-\hb \rN_o)^\ell),
\]
and the action of $\hb^2\partial_\hb$ is the natural one. According to Lemma \ref{lem:vhstw}, the right-hand side is given by
\[
\wt\cT_\ell=(R_{F[\ell]}\rP\gr_\ell^\rM V_o,R_F\rP\gr_\ell^\rM V_o,R_Fk_o(\cbbullet,\ovv{\rN_o^\ell\cbbullet})),\quad\wt\cS_\ell=((-\hb)^\ell,\hb^\ell),
\]
and the action of $\hb^2\partial_\hb$ is the natural one. If one notices that $R_{F[\ell]}\rP\gr_\ell^\rM V_o=\hb^\ell R_F\rP\gr_\ell^\rM V_o$, then one checks that $\varphi\defin((-\rN_o)^\ell,\id):(\cT_\ell,\cS_\ell)\to(\wt\cT_\ell,\wt\cS_\ell)$ is an isomorphism.
\end{proof}

Let us note that, by definition, for $(\cT,\cS,\cN,\hb^2\partial_\hb)$ as in Definition \ref{def:twnilporb}, the object $(\gr_\bbullet^\rM\cT,\gr_\bbullet^\rM\cS,\gr_{-2}^\rM\cN,\hb^2\partial_\hb)$ is a polarized graded Lefschetz twistor structure of weight~$w$ and type $-1$ (\cf \cite[\S2.1.e]{Bibi01c}). For such an object, there is a reduction to weight~$0$ and type~$0$ (\cf \loccit) giving rise to a polarized (graded Lefschetz) twistor structure of weight~$0$ (and type~$0$). This structure remains integrable and therefore comes equipped with a $\Susy$ polynomial. We denote it by $\Susy_{\Tw(V_o,F^\cbbullet,k_o,\rN_o)}(T)$.

\begin{lemme}\label{lem:SP=Susy}
Let $(V_o,F^\cbbullet,k_o,\rN_o)$ be a polarized complex mixed Hodge structure of weight~$w$. Then
\[
\SP^\infty_{\Tw(V_o,F^\cbbullet,k_o,\rN_o)}(T)=\Susy_{\Tw(V_o,F^\cbbullet,k_o,\rN_o)}(T)=\SP^0_{\Tw(V_o,F^\cbbullet,k_o,\rN_o)}(T).
\]
\end{lemme}

\begin{proof}
Each $\gr^\rM_\ell\cT$ comes equipped with a polarization $\cS_\ell$ defined from that on the various $\rP\gr^\rM_{\ell'}\cT$ by using the Lefschetz decomposition, making it a polarized twistor structure of weight $w+\ell$ (\cf \cite[Rem\ptbl2.1.15]{Bibi01c}), and one has
\[
\Susy_{\Tw(V_o,F^\cbbullet,k_o,\rN_o)}(T)=\prod_{\ell\in\ZZ}\Susy_{\gr_\ell^\rM\Tw(V_o,F^\cbbullet,k_o,\rN_o)}(T).
\]
On the other hand, because each $\gr_\ell^\rM R_FV_o$ is a free $\CC[\hb]$-module (being equal to $R_F\gr_\ell^\rM V_o$) we have such a product formula for $\SP^0$ and $\SP^\infty$, according to Remark \ref{rem:modifnabla}. Then Lemma \ref{lem:susyspHodge} applies.
\end{proof}

\begin{definition}[Vanishing cycles, {\cf \cite[Prop\ptbl2.1.3]{K-K87}}]\label{def:vanishingcycles}
Consider a polarized complex mixed Hodge structure $(V_o,F^\cbbullet,k_o,\rN_o)$ of weight~$w$. The vanishing cycle polarized complex mixed Hodge structure $(\wt V_o,\wt F^\cbbullet,\wt k_o,\wt \rN_o)$ of weight $w+1$ attached to it is defined as follow:
\[
\wt V_o=\rN_oV_o,\quad \wt F^\cbbullet=\rN_oF^\cbbullet,\quad \wt k_o(\rN_ox,\ov{\rN_oy})=k_o(x,\ov{\rN_oy}),\quad \wt \rN_o=N_{o|\wt V_o}.
\]
\end{definition}

\subsection{Extension of $\Tw$ through a singularity}\label{subsec:Schmid}
Let $(V,\Vnablaf)$ be a holomorphic bundle with connection on the punctured disc~$X^*$ underlying a polarized variation of Hodge structure of weight~$w$. We are in the situation considered in \S \ref{subsec:VTSVHS}. According to Lemma \ref{lem:vhstw}, $\big(\cT=(R_{F[w]}V,R_FV,R_Fk),\cS\big)$ is a polarized variation of twistor structures of weight~$w$ on~$X^*$. We will indicate how to extend it as a polarized twistor $\cD$-module on~$X$.

According to Schmid \cite{Schmid73}, the $\cO_X[\xm]$-submodule $\wt\ccM$ of $j_*V$ consisting of sections whose $h$-norm has moderate growth at the origin is locally free and the connection~$\nabla$ extends to it with regular singularities. We denote by $\ccM$ the $\cD_X$-submodule of $\wt\ccM$ generated by local sections $v$ whose $h$-norm is bounded by $C|x|^{-1+\epsilon}$ for some $C>0$ and $\epsilon>0$. It is known that~$\ccM$ is a regular holonomic $\cD_X$-module, which coincides with the \emph{minimal extension} of~$\wt\ccM$, so $\DR\ccM$ is the \emph{intermediate extension} (or intersection complex) of the local system $\ker\Vnablaf$ on $X^*$. Moreover, the filtration $F^\cbbullet V$ extends to a filtration of $\wt\ccM$ by holomorphic locally free $\cO_X$-modules, and then to a filtration of $\ccM$, which is a good filtration when we consider it as an increasing filtration. Lastly, the flat sesquilinear form $k$ defined from the metric $h$ extends as a $\cD_X\otimes_\CC\ov{\cD_X}$-linear pairing $k:\ccM\otimes_\CC\ov\ccM\to\Db_X$.

We can apply the Rees construction $R_F$ to these data.

\begin{proposition}[{\cf \cite[\S3.g]{Bibi05}}]\label{prop:RFMtw}
The object
\[
\big(\cT=(R_{F[w]}\ccM,R_F\ccM,R_Fk),\cS\big)
\]
is an integrable polarized twistor $\cD$-module of weight~$w$ on~$X$.\qed
\end{proposition}

\begin{definition}\label{def:polHodge}
We will call such an object a \emph{polarized complex Hodge $\cD$-module of weight~$w$}.
\end{definition}

\subsection{Nearby and vanishing cycles}\label{subsec:nearbyvanishing}

We will set $F^pV^\beta\ccM\defin F^p\ccM\cap V^\beta\ccM$. In particular, for any $k\geq0$, $x^kF^pV^\beta\ccM\subset F^pV^{\beta+k}\ccM$.

Assume that $(\ccM,F^\cbbullet\ccM)$ underlies a polarized complex Hodge $\cD$-module (\cf Definition \ref{def:polHodge}), then (\cf \cite[\S3.2]{MSaito86} and \cite[\S3.d]{Bibi05})
\begin{equation}\label{eq:FVM}
\begin{split}
\forall\beta>-1,\;\forall p\in\ZZ,\quad F^pV^\beta\ccM&=j_*j^*F^p\cap V^\beta\ccM\\
F^p\ccM&=\sum_{j\geq0}\partial_x^jF^{p+j}V^{>-1}\ccM.
\end{split}
\end{equation}
In particular, as a consequence of the first line of \eqref{eq:FVM}, we have
\begin{equation}\label{eq:tFVM}
\forall\beta>-1,\,\forall p,\,\forall k\geq0,\quad x^kF^pV^\beta\ccM=F^pV^{\beta+k}\ccM.
\end{equation}
Moreover, $x\partial_x:\gr^\beta_V\ccM\to\gr^\beta_V\ccM$ strictly shifts the filtration $F^\cbbullet$ by $-1$. It is an isomorphism if $\beta\neq0$.

As a consequence of the results recalled in \S\ref{subsec:speinttame}, we find:

\begin{corollaire}\label{cor:psinilporb}
If $(\cT,\cS)$ is a polarized complex Hodge $\cD$-module of weight~$w$ then, for any $\beta\in(-1,0]$,
\begin{enumerate}
\item\label{cor:psinilporb1}
$\Psi_x^\beta\cT=(R_{F[w]}\psi_x^\beta\ccM,R_F\psi_x^\beta\ccM,R_F\psi_x^\beta k)$,
\item\label{cor:psinilporb2}
This equality is compatible with the natural actions of $\hb^2\partial_\hb$ on both terms, and therefore $\hb\partial_\hb$ acts on $\Psi_x^\beta\cT$,
\item\label{cor:psinilporb3}
$(\Psi_x^\beta\cT,\Psi_x^\beta\cS,\cN,\hb^2\partial_\hb)$ is a polarized complex mixed Hodge structure of weight~$w$.\qed
\end{enumerate}
\end{corollaire}

Let us notice that \ref{cor:psinilporb}\eqref{cor:psinilporb3} can be regarded as a reformulation of Theorem (6.16) in \cite{Schmid73}, and is similar to Corollary~1 of \cite{MSaito86} on a punctured disc. Notice also that, compared to \cite{MSaito86}, the choice of the behaviour of weights by taking nearby/vanishing cycles is not the same here and in \cite{Bibi01c}, as we are working with complex Hodge structures, and we can use Tate twists by half-integers (see also the vanishing cycles below, which also has to be compared with Corollary~1 of \cite{MSaito86}).

For vanishing cycles (\cf \cite[\S3.6.b]{Bibi01c}) we find:

\begin{corollaire}\label{cor:phinilporb}
For $(\cT,\cS)$ as above, $(\phi_x^{-1}\cT,\phi_x^{-1}\cS,\cN,\hb^2\partial_\hb-\hb)$ is a polarized complex mixed Hodge structure of weight~$w$.
\end{corollaire}

\begin{proof}[Sketch of proof]
According to \cite[Cor\ptbl4.1.17]{Bibi01c} (and to an easy consequence of \S4.2 of \loccit for the polarization), the object $(\phi_x^{-1}\cT(-1/2),\phi_x^{-1}\cS(-1/2),\cN)$ is isomorphic to the image of $\cN:\cT\to\cT(-1)$ and gives rise, after grading with respect to the monodromy filtration, to a graded Lefschetz twistor structure of weight $w+1$. In order that the morphism $\cCan$ of \cite[Lemma 3.6.21]{Bibi01c} is compatible with the action of~$\hb\partial_\hb$ (giving the grading), we should equip $\phi_x^{-1}\cT(-1/2)$ with the shifted naturally induced action $\hb\partial_\hb-1$. Then $(\phi_x^{-1}\cT(-1/2),\phi_x^{-1}\cS(-1/2),\cN)$ is isomorphic to $\Tw$ of the vanishing cycles (as defined in \ref{def:vanishingcycles}) of the polarized complex mixed Hodge structure $\Tw^{-1}(\Psi_x^\beta\cT,\Psi_x^\beta\cS,\cN)$. Applying a Tate twist $(1/2)$ we find, according to our convention on Tate twist (Remark \ref{rem:tateint}), that $(\phi_x^{-1}\cT,\phi_x^{-1}\cS,\cN,\hb^2\partial_\hb-\hb)$ is (isomorphic to $\Tw$ of) a polarized complex mixed Hodge structure of weight~$w$.
\end{proof}

\subsection{Exponential twist of an integrable twistor $\cD$-module}
Starting from a variation of polarized Hodge structures $(V,F^\cbbullet V)$ on $U\subset\Afu$, Proposition \ref{prop:RFMtw} produces a complex Hodge $\cD$-module $\cT=\big((R_{F[w]}\ccM,R_F\ccM,R_Fk),\cS\big)$. Localizing away from $\infty$ and taking global sections produces a filtered $\Clt$-module $(M,F^\cbbullet M)$ with a pairing taking values in tempered distributions on $\Afu$ depending continuously on $\hb\in\bS$.

As integrability is preserved by direct images (\cf \cite[Prop\ptbl7.1.4]{Bibi01c}), we can apply Theorem \ref{th:exptwHodge} together with Proposition \ref{prop:RFMtw}:

\begin{corollaire}\label{cor:FcT}
If $(\cT,\cS)$ is a polarized complex Hodge $\cD$-module of weight~$w$ on $\PP^1$, then $\cH^0a_+\FcT$ is an integrable polarized twistor structure of weight~$w$.\qed
\end{corollaire}

Although we did not give the precise definition of the direct image functor $a_+$ (\cf \cite[\S1.6.d]{Bibi01c} for more details), one can notice that, according to the strictness property of polarized twistor $\cD$-modules, the restriction to $\hb=1$ commutes with taking $a_+$. In other words, the vector space corresponding to the polarized pure twistor structure $\cH^0a_+\FcT$ is the cokernel of $\partial_t-1:M\to\nobreak M$.

\subsubsections{Another expression of the exponentially twisted de~Rham cohomology}\label{subsec:Fourierfiltered}
One can give a more explicit expression for $\cH^0a_+\FcT$ considered in Corollary \ref{cor:FcT}. We will recall it below. For simplicity, we will assume that $w=0$.

We can extend the correspondence of \S\ref{subsec:Laplacetr} to objects with a sesquilinear pairing as follows. Let $(M,F^\cbbullet M)$ be as in \S\ref{subsec:Laplacetr} and let us moreover assume that $M$ comes equipped with a $\Clt\otimes_\CC\ov{\Clt}$-linear pairing $k:M\otimes_\CC\ov M\to\cS'(\Afu)$ with values in the Schwartz space of temperate distributions on $\Afu$. To $(M,F^\cbbullet M,k)$ we associate a Hermitian twistor structure $(\cH',\cH',\cCS)$:
\begin{itemize}
\item
we set $\cH'=G_0^{(F),\an}$ (the analytization of the object defined by \eqref{eq:defG0F});
\item
composing $k$ with the Fourier transform of temperate distributions with kernel $e^{\ov{t\tau}-t\tau}\itwopi dt\wedge d\ov t$ induces, by restriction to $\bS\defin\{|\tau|=1\}=\{|\hb|=1\}$, a sesquilinear pairing $\cCS:\cHS'\otimes_{\cOS}\ovv{\cHS'}\to\cOS$.
\end{itemize}

Moreover, this twistor structure is integrable (by using the action of $t=\hb^2\partial_\hb$ on~$G_0^{(F)}$).

\begin{lemme}[{\cf \cite[Lemma 2.1 and \S2.c]{Bibi05}}]
The twistor structure $(\cH',\cH',\cCS)$ is the exponentially twisted de~Rham cohomology of the object $(R_FM,R_FM,R_Fk)$ of $\wtRTriples(\PP^1)$.\qed
\end{lemme}

In the case where $(M,F^\cbbullet M,k)$ comes from a polarized variation of Hodge structures of weight~$0$ on~$U$ as explained at the beginning of this paragraph, we get from Corollary \ref{cor:FcT}:

\begin{corollaire}[{\cf \cite[Cor\ptbl3.15]{Bibi05}}]\label{cor:intmfktw}
Under the previous assumption, the integrable twistor structure $(\cH',\cH',\cCS)$ associated to $(M,F^\cbbullet M,k)$ is pure of weight~$0$ and polarized.\qed
\end{corollaire}

\begin{remarque}\label{rem:U}
The previous description makes it clear how to compute the conjugacy class of the endomorphism $\cU$ of \S\ref{subsec:harmHiggs}: indeed, this is the conjugacy class of the restriction of $\hb^2\partial_\hb$ to $\cH'/\hb\cH'$. It is therefore equal to the conjugacy class of~$t$ acting on $G_0^{(F)}/\hb G_0^{(F)}$. Its eigenvalues are the singular points of $M$ (at finite distance). Therefore, in general, it is not a multiple of $\id$, and the integrable twistor structure $(\cH',\cH',\cCS)$ does not correspond in the usual way to a polarized Hodge structure.
\end{remarque}

\subsection{Fourier-Laplace transformation of variations of polarized Hodge structures}\label{subsec:FLVHS}

We will make explicit the behaviour of the functor $\Tw$ under Laplace transform. We will work with the associated $\Clt$-modules.

Let $(M,F^\cbbullet M)$ be a regular holonomic $\Clt$-module with good filtration. Let us consider the Rees module $R_FM$, which is a $\CC[t,\hb]\langle\partiall_t\rangle$-module, and its Laplace transform $\wh{R_FM}$, which is a $\CC[\tau,\hb]\langle\partiall_\tau\rangle$-module. Recall that $\wh{R_FM}=R_FM$ as a $\CC[\hb]$-module and that~$\tau$ acts as~$\partiall_t$ and~$\partiall_\tau$ as $-t$. Notice that $\wh{R_FM}$ can be obtained as the cokernel of
\[
\CC[\tau]\otimes_\CC R_FM\To{\partiall_t-\tau}\CC[\tau]\otimes_\CC R_FM
\]
by the map $\sum_{k\geq0}\tau^k\otimes m_k\mto\sum_{k\geq0}\partiall_t^km_k$, and the natural action of $\CC[\tau,\hb]\langle\partiall_\tau\rangle$, as well as the action of $\hb^2\partial_\hb$, are obtained by conjugating the usual actions by $\cE^{-t\tau/\hb}$. In particular, the action of $\hb\partial_\hb$ on $\wh{R_FM}$ coming from the identification with $R_FM$ and which gives the grading of this Rees module corresponds to the action denoted $\hb\partial_\hb\otimes 1$ in Lemma \ref{lem:hbdhbtdt}, and the natural action of $\hb^2\partial_\hb$ is that given by this lemma.

Let $G_0^{(F)}$ be the Brieskorn lattice of the filtration $F^\cbbullet M$ (\cf\eqref{eq:defG0F}), that we will denote by $G_0$ for short.

We will be mainly concerned with $\wh{R_FM}_\loc\defin\CC[\tau,\tau^{-1},\hb]\otimes_{\CC[\tau,\hb]}\wh{R_FM}$. Recall (\cf \cite[Lemma 2.1]{Bibi05}) that $\wh{R_FM}_\loc\simeq\CC[\tau,\tau^{-1}]\otimes_\CC G_0$, where, on the right-hand side, the $\CC[\tau,\tau^{-1},\hb]\langle\partiall_\tau\rangle$-action is given as follows:
\begin{itemize}
\item
the $\CC[\tau,\tau^{-1}]$-structure is the natural one,
\item
the action of~$\hb$ is by $\tau\otimes\partial_t^{-1}$,
\item
the action of~$\partiall_\tau$ is by $\hb\cdot (\partial_\tau\otimes 1)-1\otimes t$.
\end{itemize}

Recall that $G\defin\wh M[\partial_t^{-1}]$ is a $\CC[\partial_t,\partial_t^{-1}]$-module with connection. In the following we will use the notation $\theta=\partial_t$, $\theta'=\partial_t^{-1}$ (with the identification above, $\theta'=\hb\tau'$ with $\tau'=\tau^{-1}$). We will denote by $V_\theta^\cbbullet G$ the $V$-filtration of~$G$ at $\partial_t=0$ and we will set $\psi_\theta^\gamma G\defin \gr_{V_\theta}^\gamma G$. Similarly, we denote by $V_{\theta'}^\cbbullet G$ the $V$-filtration at $\partial_t^{-1}=0$ and we set $\psi_{\theta'}^{0,\gamma}G\defin\gr_{V_{\theta'}}^\gamma G$. For any $c\in\CC$, we also set $\psi_{\theta'}^{c/\theta',\gamma}G\defin\gr_{V_{\theta'}}^\gamma (G\otimes\ccE^{-c/\theta'})$. As $G$ has a regular singularity at $\theta=0$, each $V_\theta^\gamma G$ is $\CC[\theta]$-free of finite type, while, as the singularity at $\theta'=0$ is usually irregular, each $V_{\theta'}^\gamma G$ has finite type over $\CC[\theta']\langle\theta'\partial_{\theta'}\rangle$. If the $G_i^\wedge$ are the regular formal modules entering in the decomposition analogous to \eqref{eq:Levelt-Turrittin} for $G^\wedge$, and if $i_0$ is the index $i$ such that $c_i=0$, then $\gr_{V_{\theta'}}^\gamma G=\gr_{V_{\theta'}}^\gamma G_{i_0}^\wedge$.

\begin{lemme}\label{lem:calculVg}
The $\CC[\tau,\tau^{-1},\hb]\langle\partiall_\tau\rangle$-module $\wh{R_FM}_\loc$ is strictly specializable at $\tau=0$ and $\tau=\infty$. The $V$-filtration is given by
\begin{align*}
V_\tau^\gamma\wh{R_FM}_\loc&=\tbigoplus_{k\in\ZZ}\tau^k\otimes(G_0\cap V_\theta^{\gamma-k}G)\simeq R_{G^{(F)}}V_\theta^\gamma G,\tag*{$(*)_\infty$}\\
V_{\tau'}^\gamma\wh{R_FM}_\loc&=\tbigoplus_{k\in\ZZ}\tau^{\prime k}\otimes(G_0\cap V_{\theta'}^{\gamma-k}G)\simeq R_{G^{(F)}}V_{\theta'}^\gamma G.\tag*{$(*)_0$}
\end{align*}
\end{lemme}

\begin{proof}
Let us denote by $U_\tau^\gamma\wh{R_FM}_\loc$ the right-hand side in $(*)_\infty$. Let us set $G_k=\theta^{\prime-k}G_0\subset G$. This is an increasing filtration of $G$ by $\CC[\theta']$-submodules. Let us fix $\gamma\in\RR$. Then, for $k\ll0$, we have $G_k\cap V_\theta^\gamma G=\{0\}$ and, for $k\gg0$, $G_k\cap V_\theta^\gamma G=G_{k-1}\cap V_\theta^\gamma G+G_k\cap V_\theta^{\gamma+1} G$ (the last equality expresses that $\nu_\alpha(G_0)=0$ for $\alpha\ll0$, \cf \cite{Bibi96bb}). If we consider $G_\bbullet\cap V_\theta^\gamma G$ as a filtration of the $\CC[\theta]$-module $V_\theta^\gamma G$ compatible with the filtration $\deg_\bbullet\CC[\theta]$ by the degree in $\theta$, these two properties are equivalent to saying that $G_\bbullet\cap V_\theta^\gamma G$ is a \emph{good} filtration, or equivalently that the Rees module $R_{G^{(F)}}V_\theta^\gamma G\defin \bigoplus_{k\in\ZZ}(G_k\cap V_\theta^\gamma G)\hb^k$ is a $R_{\deg}\CC[\theta]$-module of finite type. According to the definition of the action of~$\hb$ above, we identify $R_{\deg}\CC[\theta]$ with $\CC[\tau,\hb]$ and $R_{G^{(F)}}V_\theta^\gamma G$ with $U_\tau^\gamma\wh{R_FM}_\loc$, hence the finiteness of $U_\tau^\gamma\wh{R_FM}_\loc$ over $\CC[\tau,\hb]$.

Moreover, we get in the same way an identification of $\gr_{U_\tau}^\gamma\wh{R_FM}_\loc$ with the Rees module $R_{G^{(F)}}\gr_{V_\theta}^\gamma G$. In particular it is $\CC[\hb]$-free of finite rank, hence the strictness property.

As $tV_\theta^\gamma G\subset \partial_t^{-1}V_\theta^\gamma G=V_\theta^{\gamma-1}G$, we have
\[
\tau\partiall_\tau (1\otimes [G_0\cap V_\theta^\gamma G])=-\tau\otimes t[G_0\cap V_\theta^\gamma G]\subset \tau\otimes[G_0\cap V_\theta^{\gamma-1}G],
\]
showing that $U_\tau^\gamma\wh{R_FM}_\loc$ is stable by $\tau\partiall_\tau$. Similarly, one shows that, for $N\gg0$, $(\tau\partiall_\tau-\gamma\hb)^NU_\tau^\gamma\wh{R_FM}_\loc\subset U_\tau^{>\gamma}\wh{R_FM}_\loc$. This gives $(*)_\infty$ (\cf \cite[Lemma 3.3.4]{Bibi01c}).

For $(*)_0$, the argument is similar. It is easy to check that $R_{G^{(F)}}V_{\theta'}^\gamma G$ is a $\CC[\tau',\hb]\langle\tau'\partiall_{\tau'}\rangle$-module, and that $\tau'\partiall_{\tau'}-\gamma\hb$ is nilpotent on $R_{G^{(F)}}\gr_{V_{\theta'}}^\gamma G$, which has no $\CC[\hb]$-torsion by definition. The only new point is to check that $R_{G^{(F)}}V_{\theta'}^\gamma G$ has finite type over $\CC[\tau',\hb]\langle\tau'\partiall_{\tau'}\rangle$.

Let $k_0$ be such that $G_{k_0}\subset V^\gamma_{\theta'}G$ and, for any $k\geq k_0$, let $\bme_k$ be a finite system of $\CC[\theta']$-generators of $G_k\cap V_{\theta'}^\gamma G$. Recalling that $\theta'$ acts as $\hb\tau'$ on $R_{G^{(F)}}V_{\theta'}^\gamma G$, we find that, for any $k_1\geq k_0$, $\bigoplus_{k\leq k_1}\hb^k(G_k\cap\nobreak V_{\theta'}^\gamma G)$ is contained in the $\CC[\tau',\hb]$-submodule of $R_{G^{(F)}}V_{\theta'}^\gamma G$ generated by the $\hb^j\bme_j$, $j=k_0,\dots,k_1$.

On the other hand, the formula for the action of~$\partiall_\tau$ given above implies that $\tau'\partiall_{\tau'}$ acts on $R_{G^{(F)}}V_{\theta'}^\gamma G$ by $\hb\cdot(\theta'\partial_{\theta'}+k)$ on $\hb^k(G_k\cap V_{\theta'}^\gamma G)$. We will show that, for $k_1$ large enough and any $k\geq k_1$, $\hb^k(G_k\cap V_{\theta'}^\gamma G)$ is contained in the $\CC[\tau',\hb]\langle\tau'\partiall_{\tau'}\rangle$-module generated by $\hb^{k_1}\bme_{k_1}$.

We claim that
\begin{enumeratea}
\item\label{claim:a}
There exists $k_1$ such that, for any $k\geq k_1$,
\[
G_1\cap V^{\gamma+k}_{\theta'}G=\theta'\partial_{\theta'}(G_0\cap V^{\gamma+k}_{\theta'}G)+G_0\cap V^{\gamma+k}_{\theta'}G.
\]
\end{enumeratea}

Note that this is equivalent to $G_{k+1}\cap V^\gamma_{\theta'}G=\theta'\partial_{\theta'}(G_k\cap V^\gamma_{\theta'}G)+G_k\cap V^\gamma_{\theta'}G$. If $k=k_1$, such an equality implies $\hb^{k_1+1}(G_{k_1+1}\cap V_{\theta'}^\gamma G)\subset(\CC[\tau',\hb]+\CC[\tau',\hb]\tau'\partiall_{\tau'})\hb^{k_1}\bme_{k_1}$. Iterating the argument gives the desired inclusion.

We will prove Claim \eqref{claim:a} by working at the formal level. As it is clearly true away from $\theta'=0$, it is enough to prove \eqref{claim:a}$^\wedge$, that is, \eqref{claim:a} after tensoring with $\CC\lcr\theta'\rcr$.

Firstly, by uniqueness of the $V_{\theta'}^\cbbullet$-filtration, we have $(V_{\theta'}^\gamma G)^\wedge=V_{\theta'}^\gamma(G^\wedge)$. Moreover, $(V_{\theta'}^\gamma G\cap G_0)^\wedge=V_{\theta'}^\gamma G)^\wedge\cap G_0^\wedge$ in $G^\wedge$ [indeed, use that this is clearly true for $+$ instead of $\cap$ and that $\CC\lcr\theta'\rcr$ is flat over $\CC[\theta']$, and apply this to $(V_{\theta'}^\gamma G+G_0)^\wedge/G_0^\wedge=(V_{\theta'}^\gamma G)^\wedge/(V_{\theta'}^\gamma G\cap G_0)^\wedge$]. Therefore, it is enough to prove \eqref{claim:a}$^\wedge$.

Notice now that Claim \eqref{claim:a} is equivalent to
\begin{enumeratea}\refstepcounter{enumi}
\item\label{claim:b}
There exists $k_1$ such that, for any $k\geq k_1$,
\[
\theta^{\prime2}\partial_{\theta'}:V^{\gamma+k}_{\theta'}(G_0/\theta'G_0)\to V^{\gamma+k+1}_{\theta'}(G_0/\theta'G_0)
\]
is onto.
\end{enumeratea}
Similarly, \eqref{claim:a}$^\wedge$ is equivalent to \eqref{claim:b}$^\wedge$. Recall that $G^\wedge$ decomposes as $G^\wedge_\reg\oplus G^\wedge_\irr$ and that $V_{\theta'}^\gamma G^\wedge$ and $G_0$ decompose correspondingly. It is thus enough to prove \eqref{claim:b}$^\wedge$ on each term. On the regular part, there exists~$k_1$ such that $V^{\gamma+k_1}_{\theta'}(G_{\reg,0}^\wedge/\theta'G_{\reg,0}^\wedge)=0$ (because $V^{\gamma+k}_{\theta'}G^\wedge_\reg$ has finite type over $\CC\lcr\theta'\rcr$), hence both terms are~$0$ in \eqref{claim:b}$^\wedge$. On the purely irregular part, $V_{\theta'}^\gamma G^\wedge_\irr=G^\wedge_\irr$ for any $\gamma$, and $\theta^{\prime2}\partial_\theta'$ does not have the eigenvalue~$0$ on $G_{\irr,0}^\wedge/\theta'G_{\irr,0}^\wedge$ (\cf Remark \ref{rem:U}), hence it is onto.
\end{proof}

For the remaining of this section, we assume that $(M,F^\cbbullet M)$ is also equipped with a sesquilinear pairing $k$ such that $(M,F^\cbbullet M,k)$ comes from a polarized complex Hodge $\cD$-module on $\PP^1$.

\begin{corollaire}\label{cor:psiHodgeinfty}
$\prod_{\beta\in(-1,0]}\prod_{\ell\in\ZZ} \SP^\infty_{\gr^\rM_\ell\Psi_\tau^\beta\wh{R_FM}}(T)=\SP^\infty_{G_0^{(F)}}(T)$.
\end{corollaire}

\begin{proof}
For $\beta\in(-1,0]$, let $G^{(F),\cbbullet}\psi_\theta^\beta G$ be the filtration naturally induced by $G^{(F),\cbbullet}$. As a consequence of the identification $V_\tau^\gamma\wh{R_FM}_\loc$ with $R_{G^{(F)}}V_\theta^\gamma G$, we get, for any $\beta\in(-1,0]$,
\begin{equation}\label{eq:psiHodgeinfty}
\Psi_\tau^\beta\wh{R_FM}=R_{G^{(F)}}\psi_\theta^\beta G.
\end{equation}
Note however that the natural action of $\hb^2\partial_\hb$ on the left-hand side differs by $\tau\partiall_\tau$ from the action on the right-hand side defined from the~$\hb$-grading (\cf Lemma \ref{lem:hbdhbtdt}). Nevertheless, by our assumption on $(M,F^\cbbullet M)$, the graded pieces of $\Psi_\tau^\beta\wh{R_FM}$ with respect to the monodromy filtration are strict (\ie $\CC[\hb]$-free), and at the level of $\gr^\rM_\bbullet$, $\hb^2\partial_\hb$ on the left-hand side differs by $\beta\hb$ from the action on the right-hand side. Because of freeness and uniqueness of the monodromy filtration, we have
\[
\gr^\rM_\ell\Psi_\tau^\beta\wh{R_FM}=R_{G^{(F)}}\gr_\ell^\rM\psi_\theta^\beta G.
\]
On the other hand, as the action of $\partial_\hb$ has a simple pole on $R_{G^{(F)}}\psi_\theta^\beta G$, we can apply Remark \ref{rem:modifnabla} to get
\[
\prod_{\ell\in\ZZ}\SP^\infty_{\gr_\ell^\rM\Psi_\tau^\beta\wh{R_FM}}(T+\beta)=\prod_{\ell\in\ZZ}\SP^\infty_{R_{G^{(F)}}\gr_\ell^\rM\psi_\theta^\beta G}(T)=\SP^\infty_{R_{G^{(F)}}\psi_\theta^\beta G}(T).
\]
Recall that, if we set
\[
\nu_{\beta,p}=\dim \frac{V_\theta^\beta G\cap G^{(F),p}}{V_\theta^{>\beta} G\cap G^{(F),p}+V_\theta^\beta G\cap G^{(F),p+1}}.
\]
we have, as in the proof of Lemma \ref{lem:SP=Susy}, and by Definition \ref{def:spinfty} and \eqref{eq:nugamma},
\begin{align*}
\SP^\infty_{R_{G^{(F)}}\psi_\theta^\beta G}(T)&=\prod_{p\in\ZZ}(T-p)^{\nu_{\beta,p}}\\
\tag*{and}
\SP^\infty_{G_0^{(F)}}(T)&=\prod_{\beta\in(-1,0]}\prod_{p\in\ZZ}(T-\beta-p)^{\nu_{\beta,p}}.
\end{align*}
This gives the desired equality.
\end{proof}

We now consider the specialization at $\partial_t^{-1}=0$. Let $G_i^\wedge$ be the formal microlocalized module attached to $M$ at $-c_i$. Let $p_i$ be such that $F^{p_i}\ccM$ generates $\ccM$ as a $\cD$-module near $-c_i$, let $G_{i,0}^{(F^{p_i})}$ be the saturation by $\theta'\defin\partial_t^{-1}$ of the image of $F^{p_i}\ccM$ in $G_i^\wedge$ (by tensoring with formal microlocal differential operators of order zero), and let us set $G_{i,0}^{(F)}=\theta^{\prime-p_i}G_{i,0}^{(F_{p_i})}$, which is independent of the generating index $p_i$ (\cf \S\ref{subsec:Laplacetr}). Then it is known (\cf \eg \cite[Prop\ptbl V.3.6]{Bibi00}) that the Levelt-Turrittin decomposition \eqref{eq:Levelt-Turrittin} for $G_0^{(F)}$ has components $\cH'_i=G_{i,0}^{(F)}$. We have $G^{(F),\cbbullet}_i\psi_{\theta'}^\gamma G_i^\wedge=G^{(F),\cbbullet}\psi_{\theta'}^{c_i/\theta',\gamma}G$.

\begin{corollaire}\label{cor:psiHodge0}
For any $i$,
\[
\prod_{\beta\in(-1,0]}\prod_{\ell\in\ZZ} \SP^0_{\gr^\rM_\ell\Psi_{\tau'}^{c_i/\tau',\beta}\wh{R_FM}}(T)=\SP^0_{G_{i,0}^{(F)}}(T).
\]
\end{corollaire}

\begin{proof}
We will show the result for $c_i=0$. The same argument applies for any $c_i$ after twisting by $\ccE^{-c_i/\theta'}$ or $\cE^{-c_i/\hb\tau'}$. We denote by $i_0$ the index $i$ such that $c_{i_0}=0$.

From $(*)_0$ we get, for any $\beta\in(-1,0]$,
\[
\Psi_{\tau'}^{0,\beta}\wh{R_FM}=R_{G^{(F)}}\psi_{\theta'}^{0,\beta}G=R_{G_{i_0}^{(F)}}\psi_{\theta'}^\beta G_{i_0}^\wedge,
\]
and this equality is compatible with the action of $\hb^2\partial_\hb-\tau'\partiall_{\tau'}$ on the left-hand side and that of $\hb^2\partial_\hb$ on the right-hand side (\cf Lemma \ref{lem:hbdhbtdt}). By our assumption on $(M,F^\cbbullet M)$, we know from Appendix \ref{sec:appendiceA} that the graded pieces of $\Psi_{\tau'}^{c_{i_0}/\tau',\beta}\wh{R_FM}$ with respect to the monodromy filtration are strict (\ie $\CC[\hb]$-free). The same property holds for the right-hand side above, and going to the graded pieces, we find that the equality holds with $\hb^2\partial_\hb$ action on the left-hand side shifted by $-\beta\hb$. Arguing as for Corollary \ref{cor:psiHodgeinfty}, we find
\[
\prod_{\ell\in\ZZ}\SP^0_{\gr_\ell^\rM\Psi_{\tau'}^{0,\beta}\wh{R_FM}}(T-\beta)=\SP^0_{R_{G_{i_0}^{(F)}}\psi_{\theta'}^\beta G_{i_0}^\wedge}(T).
\]
Setting now
\[
\mu_{i_0,\beta,p}=\dim \frac{V_{\theta'}^\beta G_{i_0}^\wedge\cap G_{i_0}^{(F),p}}{V_{\theta'}^{>\beta} G_{i_0}^\wedge\cap G_{i_0}^{(F),p}+V_{\theta'}^\beta G_{i_0}^\wedge\cap G_{i_0}^{(F),p+1}},
\]
we have
\begin{align*}
\SP^0_{R_{G_{i_0}^{(F)}}\psi_{\theta'}^\beta G_{i_0}^\wedge}(T)&=\prod_{p\in\ZZ}(T-p)^{\mu_{i_0,\beta,p}}\\
\tag*{and}
\SP^0_{G_{i_0,0}^{(F)}}(T)&=\prod_{\beta\in(-1,0]}\prod_{p\in\ZZ}(T+\beta-p)^{\mu_{i_0,\beta,p}},
\end{align*}
hence the result.
\end{proof}

\section{Deligne's filtration}\label{sec:Delignefilt}
In this section, we will be concerned with the first point considered in the introduction. Let us consider the setting of \S\ref{subsec:Laplacetr}, that is, a holonomic $\Clt$-module equipped with a good filtration $F^\cbbullet M$. Recall that $\wt\ccM$ denotes the associated $\cD_{\PP^1}(*\infty)$-module with connection and $\ccM$ denotes its minimal extension across $\infty$. We will now assume that $M$ has only regular singularities at finite distance and at infinity.

We will keep the notation of \S\ref{subsec:Laplacetr}, but we will simply denote by $G_0$ the $\CC[\partial_t^{-1}]$-module $G_0^{(F)}$ defined by \eqref{eq:defG0F}. The spectral polynomial $\SP^\infty_{G_0}(T)$ is determined as soon as we determine the number $\nu_\gamma(G_0)$ for any $\gamma\in\RR$.

In \S\ref{subsec:DelignefiltrdR}, we will define Deligne's filtration $F_\Del^\cbbullet$ (indexed by $\RR$) on $\ccM\otimes\ccE^{-t}$, then on the corresponding de~Rham complex, and then on its hypercohomology. The main result of this section will be:

\begin{theoreme}\label{th:Deligne}
Assume that $(M,F^\cbbullet M)$ underlies a polarized complex Hodge $\cD$-module (\cf Definition \ref{def:polHodge}). Then,
\begin{enumerate}
\item\label{th:Deligne1}
the spectral sequence associated to the hypercohomology of the filtered de~Rham complex $F_\Del^\cbbullet\DR(\wt\ccM\otimes\ccE^{-t})$ on $\PP^1$ degenerates at $E_1$;
\item\label{th:Deligne2}
for any $\gamma\in\RR$, $\nu_\gamma(G_0)=\dim\HH^1\big(\PP^1,\gr_{F_\Del}^{\gamma+1}\DR(\wt\ccM\otimes\ccE^{-t})\big)$.
\end{enumerate}
\end{theoreme}

Let us remark that the $\cO_{\PP^1}$-coherent sheaf $\gr_{F_\Del}^\gamma(\wt\ccM\otimes\ccE^{-t})$ is supported at infinity if $\gamma\not\in\ZZ$.

\subsection{Laplace transform}
We denote by $\wh\PP^1$ the projective line with coordinates $\theta,\theta'$ (that we do not denote by $\tau,\tau'$ as above at the moment) and by $\Afuh$ its chart with coordinate~$\theta$. Recall that $G=\CC[\theta',\theta^{\prime-1}]\otimes_{\CC[\theta']}G_0$ ($\theta'=\partial_t^{-1}$ as above) is equipped with a connection having a regular singularity at $\theta=0$, defined as the multiplication by $-t$. We denote by $V_\theta^\cbbullet G$ the corresponding $V$-filtration, that we assume to be indexed by $\RR$ (this assumption is implied by the assumption in Theorem \ref{th:Deligne} that $(M,F^\cbbullet M)$ underlies a polarized complex Hodge $\cD$-module).

The exponentially twisted de~Rham complex \eqref{eq:dRexptwist} is quasi-isomorphic to
\[
0\to G\To{\partial_t-1}G\to0,
\]
\enlargethispage{1.5\baselineskip}%
which is quasi-isomorphic to
\[
0\to G\To{\partial_t^{-1}-1}G\to0,
\]
which in turn is quasi-isomorphic to
\[
0\to G_0\To{\partial_t^{-1}-1}G_0\to0.
\]
In other words, the hypercohomology $H^1_{\DR}(\PP^1,\ccM\otimes\ccE^{-t})$ is identified with the fibre at $\theta'=1$ of the free $\CC[\theta']$-module $G_0$.

The $V$-filtration $V_\theta^\cbbullet G$ enables one to define, in a natural way, a filtration $V^\cbbullet H^1_{\DR}(\PP^1,\ccM\otimes\ccE^{-t})$ by setting, for any $\gamma\in\RR$,
\begin{align*}
V^\gamma H^1_{\DR}(\PP^1,\ccM\otimes\ccE^{-t})&=V^\gamma(G_0/(\theta'-1)G_0)\\
&\defin\image\big[G_0\cap V_\theta^\gamma G\rightarrow G_0/(\theta'-1)G_0\big].
\end{align*}
According to \eqref{eq:nugamma}, we have
\begin{equation}\label{eq:nugammaDR}
\nu_\gamma(G_0)=\dim\gr_V^\gamma H^1_{\DR}(\PP^1,\ccM\otimes\ccE^{-t}).
\end{equation}

\begin{exemple}\label{ex:alpha}
Let us consider the case where $M=\Clt/(t\partial_t-\alpha)$ for some $\alpha\in(0,1)$. We regard $M$ as corresponding to a variation of Hodge structure $V$ of type $(0,0)$ on $\Afu\moins\{0\}$ with filtration $F^\cbbullet V$ given by $F^0V=V$ and $F^1V=0$. Let $V_t^\cbbullet M$ denote the $V$-filtration of $M$ at $t=0$. Then we set (\cf \eqref{eq:FVM}) $F^0M=V_t^{>-1}M$, $F^1M=0$ and, for $\ell\geq0$,
$F^{-\ell}M\defin V_t^{>-1}M+\cdots+\partial_t^\ell V_t^{>-1}M=V_t^{>-\ell-1}M$.

Denoting by $[\cdot]$ the class in $M$, we have $[1]\in V_t^\alpha M$, and $F^0M=\CC[t]\cdot[\partial_t]$. We also have $\wh M=G=\Cltheta/(\partial_\theta\theta+\alpha)$ which is free of rank one over $\CC[\theta,\theta^{-1}]$, and $G_0^{(F)}=\CC[\theta^{-1}]\cdot[\theta]$. As $[\theta]$ is in $V^{-\alpha}_\theta G$, we finally get
\[
\nu_\gamma(G_0^{(F)})=
\begin{cases}
1&\text{if }\gamma=-\alpha,\\
0&\text{otherwise}.
\end{cases}
\]
\end{exemple}

Let us consider the $V$-filtration $V^\cbbullet\wh{R_FM}_\loc$ (\cf Lemma \ref{lem:calculVg}). Notice that, for any $\gamma\in\RR$, the multiplication by $\tau-\hb$ is injective on $V^\gamma\wh{R_FM}_\loc$. Indeed, let us use as in Lemma \ref{lem:calculVg} the identification $V^\gamma\wh{R_FM}_\loc=R_{G^{(F)}}V_\theta^\gamma G$. Then, the localization with respect to~$\hb$ gives $\CC[\hb,\hbm]\otimes_\CC V_\theta^\gamma G$, where the action of~$\tau$ is induced by $\hb\otimes\theta$. In particular, it is $\CC[\tau,\hb,\hbm]$-free and the multiplication by $\tau-\hb$ is injective on this module. Therefore, so is the multiplication by $\tau-\hb$ on the $\CC[\tau,\hb]$-submodule $R_{G^{(F)}}V_\theta^\gamma G$. We will compute its cokernel $V^\gamma\wh{R_FM}_\loc/(\tau-\hb)V^\gamma\wh{R_FM}_\loc$.

Recall (\cf \cite[Def\ptbl B.1]{D-S02a}) that a $V$-solution to the Birkhoff problem for $G_0$ is a free $\CC[\theta]$-submodule $G^{\prime0}$ of $G$, which is stable by $\theta\partial_\theta$, which generates $G$ over $\CC[\theta,\theta^{-1}]$, and such that, for any $\gamma\in\RR$,
\begin{equation}\label{eq:Vsol}
G_0\cap V_\theta^\gamma G=\tbigoplus_{j\geq0}\theta^{\prime j}(G_0\cap G^{\prime0}\cap V_\theta^{\gamma+j}G).
\end{equation}
(Each term in the sum, as well as $G_0\cap G^{\prime0}$ and $G_0\cap V_\theta^\gamma G$, is a finite dimensional $\CC$-vector space, and the sum is finite, as $G_0\cap V_\theta^{\gamma+j}G=0$ for $j\gg0$; moreover, $G^{\prime0}\subset V_\theta^\gamma G$ for $\gamma\ll0$.) By definition of a solution to Birkhoff's problem, a $\CC$-basis $G_0\cap G^{\prime0}$ is a $\CC[\theta']$-basis of $G_0$; therefore, the natural morphism $G_0\cap G^{\prime0}\to G_0/(\theta'-1)G_0$ is an isomorphism.

\begin{lemme}\label{lem:tau-hb}
If the Birkhoff problem for $G_0$ has a $V$-solution $G^{\prime0}$, then, for any $\gamma\in\RR$, $V^\gamma\wh{R_FM}_\loc/(\tau-\hb)V^\gamma\wh{R_FM}_\loc$ is identified with the Rees module of the filtration $V^{\gamma+\cbbullet}H^1_{\DR}(\PP^1,\ccM\otimes\ccE^{-t})$.
\end{lemme}

\begin{proof}
We note that $(*)_\infty$ in Lemma \ref{lem:calculVg} gives (as $\hb=\tau\otimes\theta'$):
\begin{multline*}
V^\gamma\wh{R_FM}_\loc/(\tau-\hb)V^\gamma\wh{R_FM}_\loc\\
=\tbigoplus_k\tau^k\otimes\Big[(G_0\cap V_\theta^{\gamma-k}G)/(\theta'-1)(G_0\cap V_\theta^{\gamma-k+1}G)\Big].
\end{multline*}
As $G^{\prime0}$ is a $V$-solution, the natural inclusion
\[
(\theta'-1)(G_0\cap V_\theta^{\gamma+1}G)\subset[(\theta'-1)G_0]\cap (G_0\cap V_\theta^{\gamma}G)
\]
is an equality for any $\gamma\in\RR$: indeed, since $G_0=\bigcup_{\ell\in\ZZ}(G_0\cap\nobreak V_\theta^{\gamma-\ell}G)$, an element in the RHS can be written both as a polynomial $(\theta'-\nobreak1)\sum_{j\geq0}a_j\theta^{\prime j}$ with $a_j\in G_0\cap\nobreak G^{\prime0}\cap\nobreak V_\theta^{\gamma-\ell+j}G$ for some fixed $\ell\in\ZZ$, and as a polynomial $\sum_{j\geq0}b_j\theta^{\prime j}$ with $b_j\in G_0\cap G^{\prime0}\cap V_\theta^{\gamma+j}G$, according to \eqref{eq:Vsol}; the assertion follows by considering the term of highest degree with respect to $\theta'$ and by a straightforward induction. As a consequence,
\[
(G_0\cap V_\theta^\gamma G)/(\theta'-1)(G_0\cap V_\theta^{\gamma+1}G)
\]
is equal to the image of $G_0\cap V_\theta^{\gamma}G$ in $G_0/(\theta'-1)G_0$ for any $\gamma$. The result follows.
\end{proof}

The previous proof also shows that the morphism $G_0\cap\nobreak V_\theta^\gamma G\to G_0/(\theta'-\nobreak1)G_0$ induces an isomorphism
\[
G_0\cap G^{\prime0}\cap V_\theta^\gamma G\isom V^\gamma(G_0/(\theta'-1)G_0)\simeq V^\gamma H^1_{\DR}(\PP^1,\ccM\otimes\ccE^{-t}).
\]
As a consequence of the lemma, for any $\beta\in(-1,0]$,
\begin{equation}\label{eq:VRF}
V^\beta\wh{R_FM}_\loc/(\tau-\hb)V^\beta\wh{R_FM}_\loc=R_{V^{\beta+\cbbullet}}H^1_{\DR}(\PP^1,\ccM\otimes\ccE^{-t}).
\end{equation}

\begin{lemme}\label{lem:HodgeVsol}
If $(M,F_\bbullet M)$ underlies a polarizable complex Hodge $\cD$-module, then the Birkhoff problem for the Brieskorn lattice of $(M,F_\bbullet M)$ has a $V$-solution.
\end{lemme}

\begin{proof}
We follow the argument of \cite[Lemma 2.8]{MSaito89}. According to \cite[Prop\ptbl B.3(1)]{D-S02a}, giving a $V$-solution to the Brieskorn problem for $G_0^{(F)}$ is equivalent to giving, for any $\beta\in(-1,0]$, a filtration of $\psi_\theta^\beta G$ which is opposite to $G^{(F),\cbbullet}\psi_\theta^\beta G$ and which is stable by the nilpotent operator~$\rN_\theta$ induced by $-(\theta\partial_\theta-\beta)$. According to \eqref{eq:psiHodgeinfty} and Corollary \ref{cor:psinilporb}\eqref{cor:psinilporb3}, $G^{(F),\cbbullet}\psi_\theta^\beta G$ is the Hodge filtration of a polarized complex mixed Hodge structure for which the nilpotent endomorphism is a nonzero multiple of $\rN_\theta$. Remark \ref{rem:Ipq} gives then a convenient opposite filtration.
\end{proof}

\Subsection{Deligne's filtration on the exponentially twisted $\cD_X$-module}\label{subsec:Delignefiltr}
In this subsection, we denote by~$X$ an open disc in~$\CC$ with a coordinate $x$ centered at its origin. Let $\ccM$ be a regular holonomic $\cD_X$-module equipped with a good filtration $F_\bbullet\ccM$. We will make the following assumptions:
\begin{enumerate}
\item
$\ccM$ has a singularity at $x=0$ at most and is the minimal extension of its localized module $\wt\ccM\defin\cO_X[1/x]\otimes_{\cO_X}\ccM$.
\item
The eigenvalues of the monodromy of the local system \hbox{$\ker\big[\partial_x:\ccM_{|X^*}\to\ccM_{|X^*}\big]$} have an absolute value equal to~$1$. (Hence, the decreasing Kashiwara-Malgrange filtration $V^\cbbullet\ccM$ of $\ccM$ at the origin is indexed by a finite set of real numbers translated by $\ZZ$.)
\end{enumerate}

We denote by $\ccM\otimes\ccE^{-1/x}$ the $\cO_X[1/x]$-module $\wt\ccM$ equipped with the twisted connection $\nabla-d(1/x)$ (\ie if $e$ is the generator of the rank one $\cO_X[1/x]$-module $\ccE^{-1/x}$, we have $\partial_x(e\otimes m)=e\otimes((\partial_x+x^{-2})m)$). For any $\gamma\in\RR$, we denote by $\lgamma$ the smallest integer $\geq\gamma$, so that $\gamma-\lgamma\in(-1,0]$. Deligne's filtration is defined for $\gamma\in\RR$ by:
\[
F_\Del^{\gamma}(\wt\ccM\otimes\ccE^{-1/x})\defin\sum_{k\geq0}\partial_x^kx^{-1}(F^{\lgamma+k}V^{\gamma-\lgamma}\ccM\otimes\ccE^{-1/x}).
\]
(The usefulness of the shift by $x^{-1}$ will appear later.) From now on, we will skip the term $\otimes\ccE^{-1/x}$, so we will write
\[
F_\Del^{\gamma}\wt\ccM\defin\sum_{k\geq0}(\partial_x+x^{-2})^kx^{-1} F^{\lgamma+k}V^{\gamma-\lgamma}\ccM.
\]
The sum above consists of a finite number of terms since, for $\beta\in(-1,0]$ fixed, $F^rV^\beta\ccM=0$ for $r\gg0$. Therefore, each $F_\Del^{\gamma}\wt\ccM$ is a locally free $\cO_X$-module of finite rank. Moreover, we clearly have the transversality property $(\partial_x+x^{-2})F_\Del^{\gamma}\wt\ccM\subset F_\Del^{\gamma-1}\wt\ccM$.

Let us show that the filtration is decreasing. Assume that $\gamma'\defin\beta'+p'\geq\gamma\defin\beta+p$, with $\beta,\beta'\in(-1,0]$ and $p,p'\in\ZZ$. The inclusion $F_\Del^{\beta'+p'}\subset F_\Del^{\beta+p}$ is clear if $\beta'\geq\beta$ (so~$p'\geq\nobreak p$). It remains to consider the case where $\beta'<\beta$ and $p'\geq p+1$ and it is enough to assume $p'=p+1$. Writing $1=(\partial_x+x^{-2})x^2-\partial_xx^2$, we get, for $k\geq0$,
\begin{multline*}
(\partial_x+x^{-2})^kx^{-1} F^{p+1+k}V^{\beta'}\ccM\subset (\partial_x+x^{-2})^{k+1}x^{-1} F^{p+1+k}V^{\beta'+2}\ccM\\
+(\partial_x+x^{-2})^k\partial_xxF^{p+1+k}V^{\beta'}\ccM.
\end{multline*}
In the right-hand side, the first term is contained in $F_\Del^{\beta+p}$ as \hbox{$\beta'+2\geq\beta$}. For the second one, we note that $\partial_xxF^{p+1+k}V^{\beta'}\subset F^{p+k}V^{\beta'}\subset x^{-1}F^{p+k}V^{\beta'+1}$, so the second term is also contained in $F_\Del^{\beta+p}$, as $\beta'+1\geq\beta$.

\begin{exemple}\label{ex:beta}
Let us assume, as in \cite{Deligne8406}, that \eqref{eq:FVM} holds and that $j^*F^\cbbullet\ccM$ has only one jump at $p=0$, so $j^*F^1\ccM=0$ and $j^*F^0\ccM=j^*\ccM$. Then, for $\beta\in(-1,0]$ and~$p\in\nobreak\ZZ$,
\[
F_\Del^{\beta+p}\wt\ccM=\begin{cases}
0&\text{if $p\geq1$},\\
x^{2p-1}V^\beta&\text{if $p\leq0$}.
\end{cases}
\]
Indeed, if $p=-1$ for instance,
\[
F_\Del^{\beta-1}\wt\ccM=(\partial_x+x^{-2})x^{-1}V^\beta+x^{-1}V^\beta=x^{-3}V^\beta,
\]
as $x^2\partial_x+1$ is invertible on $V^\beta\ccM$ (as $\ccM$ has a regular singularity at $x=0$, it is enough to check this on modules like $\cO_X\langle\partial_x\rangle/(x\partial_x-\alpha)^\mu$).
\end{exemple}

\subsection{Deligne's filtration on the de~Rham complex}\label{subsec:DelignefiltrdR}

We now consider the case where $(\ccM,F^\cbbullet)$ is a filtered $\cD$-module on the projective line $\PP^1$. We denote by~$t$ a fixed affine coordinate on the affine line $\Afu=\PP^1\moins\{\infty\}$. We define the Deligne filtration on $\ccM\otimes\ccE^{-t}$, with $\wt\ccM\defin \cO_{\PP^1}(*\infty)\otimes_{\cO_{\PP^1}}\ccM$, by the following formulas (for simplicity, we identify $\ccM\otimes\ccE^{-t}$ with the $\cO_{\PP^1}(*\infty)$-module $\wt\ccM$ with twisted connection $\nabla-dt$):
\begin{itemize}
\item
Away from $\infty$, we set $F_\Del^{\beta+p}\wt\ccM=F^p\ccM$ for $\beta\in(-1,0]$, and $p\in\ZZ$,
\item
near $\infty$, we set $x=1/t$ and use the definition of \S\ref{subsec:Delignefiltr}.
\end{itemize}

The de~Rham complex $\DR(\wt\ccM\otimes\ccE^{-t})$ is filtered by setting, for each $\gamma\in\RR$,
\[
F_\Del^{\gamma}\DR(\wt\ccM\otimes\ccE^{-t})=\{0\ra F_\Del^{\gamma}(\wt\ccM\otimes\ccE^{-t})\To\nabla\Omega_{\PP^1}^1\otimes F_\Del^{\gamma-1}(\wt\ccM\otimes\ccE^{-t})\ra0\},
\]
a complex which is also written as
\[
\{0\ra F_\Del^{\gamma}\wt\ccM\To{\nabla-dt}\Omega_{\PP^1}^1\otimes F_\Del^{\gamma-1}\wt\ccM\ra0\}.
\]

\begin{exemple}
Let us consider the case of Example \ref{ex:alpha}. Using the computation in Example \ref{ex:beta}, we find, on $\PP^1\moins\{t=0\}$ with the coordinate~$t'$,
\[
F_\Del^\gamma(\wt\ccM\otimes\ccE^{-t})=
\begin{cases}
0&\text{if }\gamma>0,\\
\CC[t']t^{\prime-2p}[1]&\text{if }\gamma=-p,\;p\in\NN,\\
\CC[t']t^{\prime-(2p+1)}[1]&\text{if }\gamma=-\alpha-p,\;p\in\NN.
\end{cases}
\]
In particular, the complex $\gr_{F_\Del}^{1-\alpha}\DR(\wt\ccM\otimes\ccE^{-t})$ reduces to the complex having only $\gr_{F_\Del}^{-\alpha}(\wt\ccM\otimes\ccE^{-t})$ in degree one, and we get \ref{th:Deligne}\eqref{th:Deligne2} in that case.
\end{exemple}

Theorem \ref{th:Deligne} is a consequence of the following proposition, together with Lemma \ref{lem:HodgeVsol} and \eqref{eq:nugammaDR}.

\begin{proposition}\label{prop:Deligne}
If $(\ccM,F^\cbbullet)$ underlies a polarized complex Hodge $\cD$-module on $\PP^1$, then the filtered de~Rham complex $\bR \Gamma\DR(\wt\ccM\otimes\ccE^{-t}, F_\Del^\cbbullet)$ is strict, that is, for any $\gamma\in\RR$, the natural morphism
\[
\HH^*\big(\PP^1,F_\Del^{\gamma}\DR(\wt\ccM\otimes\ccE^{-t})\big)\to \HH^*\big(\PP^1,\DR(\wt\ccM\otimes\ccE^{-t})\big)
\]
is injective, with image $V^{\gamma-1}\HH^1\big(\PP^1,\DR(\wt\ccM\otimes\ccE^{-t})\big)$ (when $*=1$).
\end{proposition}

\begin{proof}
We assume for simplicity that $R_F\ccM$ underlies a polarized Hodge $\cD$-module of weight~$0$, with polarization $(\id,\id)$, that we denote $\cT=(R_F\ccM,R_F\ccM,R_Fk)=(\cM,\cM,C)$ (\cf \cite{Bibi05}). Let us consider a new copy of the affine line, that we denote by $\Afuh$ with coordinate~$\tau$, and let us denote by $p:\PP^1\times\Afuh\to\PP^1$ the projection. Let us set $\cFcT=(\cFcM,\cFcM,\cFC)$ with
\[
\cFcM\defin \cE^{-t\tau/\hb}\otimes p^+\wt\cM
\]
and $\cFC$ defined as in \cite[\S3]{Bibi05b}. If $V^\cbbullet\wt\cM$ denotes the $V$-filtration along $t'=0$, we have $V_{t'}^\beta\wt\cM=R_FV_{t'}^\beta\ccM$ for $\beta>-1$. The $V$-filtration of $\cFcM$ along $\tau=0$ is computed in \cite{Bibi05b}. In a chart near $(t=\infty,\tau=0)$, setting $t'=1/t$, it is given by the formula (if~$\beta\in(-1,0]$)
\[
V_\tau^\beta\cFcM=
\sum_{k\geq0}\partiall_{t'}^k(p^*R_FV_{t'}^{\beta-1}\wt\ccM\otimes\cE^{-t\tau/\hb}).
\]
(There is also a formula for $V_\tau^\gamma\cFcM$ for any $\gamma\in\RR$, but it will be needed here.) This can be rewritten as
\[
V_\tau^\beta\cFcM=
\sum_{k\geq0}(1\otimes\hb\partial_{t'}+\tau\otimes t^{\prime-2})^k(\CC[\tau]\otimes_\CC t^{\prime-1}R_FV_{t'}^\beta\wt\ccM).
\]
Let us now consider the multiplication by $\tau-\hb$. It is clearly injective on $\cFcM$, hence on each $V_\tau^\beta\cFcM$. The cokernel is given by
\[
V_\tau^\beta\cFcM/(\tau-\hb)V_\tau^\beta\cFcM=
\sum_{k\geq0}(\partial_{t'}+t^{\prime-2})^kt^{\prime-1}\hb^k\CC[\hb]R_FV_{t'}^\beta\wt\ccM
\]
that is, for any $\beta\in(-1,0]$,
\begin{equation}\label{eq:VcFcMred}
V_\tau^\beta\cFcM/(\tau-\hb)V_\tau^\beta\cFcM=R_{F^{\beta+\cbbullet}_\Del}(\wt\ccM\otimes\ccE^{-1/t'}).
\end{equation}
In a chart away from $t=\infty$ (and near $\tau=0$), we have, for $\beta\in(-1,0]$,
\[
V_\tau^\beta\cFcM=\cFcM,
\]
and therefore \eqref{eq:VcFcMred} remains valid in this chart.

Let $\wh p:\PP^1\times\Afuh\to\Afuh$ denote the projection. Then, according to \cite[Th\ptbl3.3.15]{Bibi01c}, \cite[Prop\ptbl4.1(ii) and Cor\ptbl5.9]{Bibi05b} and \cite[Th\ptbl6.1.1]{Bibi01c}, the filtered complex $\wh p_+V_\tau^\cbbullet\cFcM$ is strict, that is, $\cH^j\wh p_+V_\tau^\beta\cFcM\to\cH^j\wh p_+\cFcM$ is injective for any $j$ and~$\beta$, therefore $\cH^j\wh p_+V_\tau^\beta\cFcM=0$ for any $j\neq0$ and any~$\beta$, and $\cH^0\wh p_+V_\tau^\beta\cFcM$ is identified with $V_\tau^\beta\cH^0\wh p_+\cFcM=V_\tau^\beta\wh{R_FM}=V_\tau^\beta\wh{R_FM}_\loc$ (because $\beta>-1$).

We thus have an exact sequence
\[
0\ra\cH^0\wh p_+V_\tau^\beta\cFcM\To{\tau-\hb}\cH^0\wh p_+V_\tau^\beta\cFcM\ra\cH^0\wh p_+R_{F_\Del^{\beta+\cbbullet}}(\wt\ccM\otimes\ccE^{-t})\ra0
\]
which is identified with the exact sequence (\cf \eqref{eq:VRF})
\[
0\ra V_\tau^\beta\wh{R_FM}_\loc\To{\tau-\hb}V_\tau^\beta\wh{R_FM}_\loc\ra R_{V^{\beta+\cbbullet}}H^1_{\DR}(\PP^1,\wt\ccM\otimes\ccE^{-t})\ra0
\]
The identification of the action of $\hb^2\partial_\hb$, \ie the grading with respect to the filtrations involved, gives
\[
\cH^0\wh p_+R_{F_\Del^{\beta+\cbbullet}}(\wt\ccM\otimes\ccE^{-t})=\HH^1\big(\PP^1,R_{F_\Del^{\beta+\cbbullet+1}}\DR(\wt\ccM\otimes\ccE^{-t})\big).
\]
Fixing the power of~$\hb$ shows therefore that, for any $p\in\ZZ$, the morphism
\[
\HH^1\big(\PP^1,F_\Del^{\beta+p}\DR(\wt\ccM\otimes\ccE^{-t})\big)\to\HH^1\big(\PP^1,\DR(\wt\ccM\otimes\ccE^{-t})\big)
\]
induces an isomorphism onto $V^{\beta+p-1}H^1_{\DR}(\PP^1,\wt\ccM\otimes\ccE^{-t})$, as was to be proved.
\end{proof}

\begin{remarque}\label{rem:DelignefiltG}
It is possible to define the Deligne filtration as a filtration on $G$ and not only on its fibre at $\theta=1$. For any $\gamma\in\RR$, one sets
\[
F_\Del^\gamma G=\CC[\theta,\theta^{-1}]\cdot (G_0\cap V_\theta^\gamma G)\subset G.
\]
One easily checks that Griffiths transversality holds for this filtration (\ie $t\cdot F_\Del^\gamma G\subset F_\Del^{\gamma-1} G$). Moreover, the limit filtration when $\theta\to0$ is the Hodge filtration on $\bigoplus_{\beta\in(-1,0]}\gr_{V_\theta}^\beta G$, that is, for any $\beta\in(-1,0]$,
\[
F_\Del^\gamma\gr_{V_\theta}^\beta G\defin (F_\Del^\gamma G\cap V_\theta^\beta G)/(F_\Del^\gamma G\cap V_\theta^{>\beta}G)
\]
jumps at most at $\gamma=\beta+p$, $p\in\ZZ$, and
\[
F_\Del^{\beta+p}\gr_{V_\theta}^\beta G=G^p\gr_{V_\theta}^\beta G\defin (G^p\cap V_\theta^\beta G)/(G^p\cap V_\theta^{>\beta}G),
\]
where we recall that $G^p=\theta^{\prime p}G_0$. These properties are checked by using a $V$-solution of the Birkhoff problem for $G_0$, as in Lemma \ref{lem:tau-hb}.
\end{remarque}

\section{The new supersymmetric index and the spectrum}

Let $(\cT,\cS)$ be a polarized complex Hodge $\cD$-module of weight~$w$ on~$\PP^1$ (\cf Definition \ref{def:polHodge}). Its exponentially twisted de~Rham cohomology $\cH^0a_+\FcT$ is an integrable polarized twistor structure of weight~$w$, according to Corollary \ref{cor:FcT}, which is identified to the fibre $(\wh\cT,\wh\cS)_1$ of the Fourier-Laplace transform $(\wh\cT,\wh\cS)$ at $\tau=1$.

Recall (\cf Appendix~\ref{sec:appendiceB}) that we have a rescaling action with respect to $\tau\in\CC^*$, that we denote by $\mu^*_\tau$, on integrable twistor structures. Applying it to $\wh\cT_1$, we get a family $\Susy_{\mu^*_\tau\wh\cT_1}(T)$ of polynomials in $T$ with coefficients depending on $\tau\in\CC^*$.

\begin{theoreme}\label{th:conjHertling}
If $(\cT,\cS)$ is a polarized complex Hodge $\cD$-module of weight~$w$, then
\begin{align*}
\lim_{\tau\to0}\Susy_{\mu^*_\tau\wh\cT_1}(T)&=\SP^\infty_{\wh\cT_1}(T),\\
\lim_{\tau\to\infty}\Susy_{\mu^*_\tau\wh\cT_1}(T)&=\SP^0_{\wh\cT_1}(T).
\end{align*}
\end{theoreme}

\begin{remarque}\label{rem:relations1-2}
It follows that the eigenvalues of the new supersymmetric index of $\mu_\tau^*\wh\cT_1$ interpolate, when~$\tau$ varies between~$0$ and $\infty$, between the spectrum at $\infty$, which gives, by exponentiating, the eigenvalues of the monodromy of the original variation of Hodge structure around $t=\infty$, and the spectrum at~$0$, which gives, by exponentiating, the eigenvalues of the monodromy of the original variation near the singular points at finite distance. In particular, in general, $\Susy_{\mu^*_\tau\wh\cT_1}(T)$ is far from being constant with respect to~$\tau$.

As we will indicate below, $\mu^*_\tau\wh\cT_1$ is nothing but the fibre $\wh\cT_\tau$. If $w=0$, $\wh\cT_{|\tau\neq0}$ is a variation of polarized pure twistor structure of weight~$0$, hence corresponds to a flat bundle with harmonic metric $\wh h$. The flat bundle is nothing else but $G^\an$ with its connection. Now, $\Susy_{\wh\cT_\tau}(T)$ is the characteristic polynomial of the selfadjoint operator $\cQ$ acting on the $C^\infty$ bundle associated to~$G$. Therefore, at each $\tau^o\in\CC^*$, the fibre $G_{\tau^o}$ has a $\wh h$-orthogonal decomposition indexed by the eigenvalues of $\cQ_{\tau^o}$. This decomposition does not give rise to a filtration of $G^\an$ indexed by a discrete set of $\RR$, however. On the other hand, Deligne's filtration introduced in Remark \ref{rem:DelignefiltG} is a Hodge-type filtration, but does not correspond, in general, to the previous `Hodge' decomposition. Nevertheless, this difference disappears asymptotically when $\tau\to0$, according to the theorem.
\end{remarque}

\begin{exemple}\label{ex:conjHertling}
Let $f$ be a cohomologically tame function on an smooth complex affine variety~$U$ as in Examples \ref{ex:cohomtame} and \ref{ex:cohomtameloc} and let $G_0$ be the corresponding Brieskorn lattice. In \cite[Th\ptbl4.10]{Bibi05} we have defined (following a conjecture of C\ptbl Hertling) a sesquilinear pairing $\wh C$ on $G_0$ and proved that $\cG\defin(G_0,G_0,\wh C)$ is a pure twistor structure of weight~$0$ polarized by $(\id,\id)$. It is moreover integrable, and can be obtained as the direct image by the constant map of the exponential twist of the minimal extension of a suitable variation of Hodge structure (namely, the intermediate direct image by~$f$ of~$\cO_U$, with a suitable Tate twist).

In such a case, there is a natural real structure coming from the real structure on the cohomology $H^{\dim U}(U,f^{-1}(t))$ for a regular value $t\in\CC$ of~$f$, and we can define $\cQ^\Hert$ as in Remark \ref{rem:QHert}, whose eigenvalues are symmetric with respect to~$0$. According to the symmetry, mentioned in Examples \ref{ex:cohomtame} and \ref{ex:cohomtameloc}, of the spectrum at the origin or at infinity with respect to $\tfrac12\dim U$, Theorem \ref{th:conjHertling} reads:
\begin{align*}
\lim_{\tau\to0}\Susy^\Hert_{\mu^*_\tau\cG}(T)&=\SP^\infty_{G_0}(T-\tfrac12\dim U),\\
\lim_{\tau\to\infty}\Susy^\Hert_{\mu^*_\tau\cG}(T)&=\SP^0_{G_0}(T-\tfrac12\dim U).
\end{align*}
\end{exemple}

\begin{proof}[Proof of Theorem \ref{th:conjHertling}]
It will have three steps.

\subsubsections{Step~1}
According to Proposition \ref{prop:Fourier-rescaling} in the appendix, $\mu^*_\tau\wh\cT_1$ is regarded as the fibre at~$\tau$ of the Fourier-Laplace transform $\wh\cT$ of $\cT$, whose definition is recalled in \S\ref{subsec:Fourier-Laplace}.

\subsubsections{Step~2}
According to \cite{Bibi05} (\cf \S\ref{subsec:localFL}), $\wh\cT$ is an integrable polarized regular twistor $\cD$-module of weight~$w$ on the analytic affine line with coordinate~$\tau$, and according to Theorem \ref{th:FourierlocalT} of the appendix \ref{sec:appendiceA}, it is an integrable wild twistor $\cD$-module of weight~$w$ at $\tau=\infty$ (this could also be deduced from \cite{Mochizuki08}). We can therefore apply Theorem \ref{th:limSusytame} at $\tau=0$ and Theorem \ref{th:limSusywild} at $\tau=\infty$ to compute the left-hand sides in Theorem \ref{th:conjHertling}.

\subsubsections{Step~3 for $\tau\to0$}
From Corollary \ref{cor:psinilporb} applied to the right-hand sides in \ref{eq:psiwtTT}, we conclude from Lemma \ref{lem:SP=Susy} that the $\Susy$ polynomials are equal to the corresponding $\SP^\infty$ polynomials. It follows that such a property holds for the left-hand sides, as the shift by~$\beta$ is the same for $\Susy$ and $\SP^\infty$. By Corollary \ref{cor:FcT} and Lemma \ref{lem:SP=Susy}, $\cH^0a_+\cT$ satisfies $\Susy=\SP^\infty$, hence so does the left-hand side in \ref{eq:psiwtTTP}. It follows that $(\Psi_\tau^0\wt\cT,\Psi_\tau^0\cS,\cN,\hb^2\partial_\hb)$ is $\Tw$ of a polarized complex mixed Hodge structure of weight~$w$, hence also satisfies $\Susy=\SP^\infty$. Therefore,
\begin{align*}
\prod_{\beta\in(-1,0]} \prod_{\ell\in\ZZ}\Susy_{\gr_\ell^\rM\Psi_\tau^\beta\wh\cT}(T)&=\prod_{\beta\in(-1,0]} \prod_{\ell\in\ZZ}\SP^\infty_{\gr_\ell^\rM\Psi_\tau^\beta\wh\cT}(T)\\
&=\SP^\infty_{G_0^{(F)}}(T)\quad\text{after Corollary \ref{cor:psiHodgeinfty}}\\
&=\SP^\infty_{\wh\cT_1}(T)\quad\text{by definition.}
\end{align*}
On the other hand, Theorem \ref{th:limSusytame} gives
\[
\lim_{\tau\to0}\Susy_{\wh\cT_\tau}(T)=\prod_{\beta\in(-1,0]} \prod_{\ell\in\ZZ}\Susy_{\gr_\ell^\rM\Psi_\tau^\beta\wh\cT}(T).\qedhere
\]
\end{proof}

\begin{proof}[Step~3 for $\tau\to\infty$]
We argue similarly at $\tau=\infty$. We apply Corollaries \ref{cor:psinilporb} and \ref{cor:phinilporb} to the right-hand sides in \ref{eq:psiwtTTinf} and get, according to Lemma \ref{lem:SP=Susy}, the equality between the $\Susy$ polynomial and the $\SP^0$ polynomial. This equality then also holds for the left-hand side, as the same shift of $-(\beta+1)$ applies to both. We conclude:
\begin{align*}
\prod_i\prod_{\beta\in(-1,0]}\prod_{\ell\in\ZZ} \Susy_{\gr^\rM_\ell\Psi_{\tau'}^{c_i/\tau',\beta}\wh\cT}&(T)=\prod_i\prod_{\beta\in(-1,0]}\prod_{\ell\in\ZZ} \SP^0_{\gr^\rM_\ell\Psi_{\tau'}^{c_i/\tau',\beta}\wh\cT}(T)\\
&=\prod_i\SP^0_{G_{i,0}^{(F)}}(T)\quad \text{after Corollary \ref{cor:psiHodge0}}\\
&=\SP^0_{\wh\cT_1}(T)\quad \text{by definition}.
\end{align*}
On the other hand, Theorem \ref{th:limSusywild} gives
\[
\lim_{\tau'\to0}\Susy_{\wh\cT_{\tau'}}(T)=\prod_i\prod_{\beta\in(-1,0]} \prod_{\ell\in\ZZ}\Susy_{\gr^\rM_\ell\Psi_{\tau'}^{c_i/\tau',\beta}\wh\cT}(T).\qedhere
\]
\end{proof}

\appendix
\section{Stationary phase formula for polarized twistor $\cD$-modules}\label{sec:appendiceA}

Let $(\cT,\cS)$ be a polarized regular twistor $\cD$-module of weight~$w$ (in the sense of \cite{Bibi01c} or \cite{Mochizuki07}) on $\PP^1$. The Fourier-Laplace transform $(\wh\cT,\wh\cS)$ is an object of the same kind on the \emph{analytic} affine line $\Afuh$ with coordinate~$\tau$, after \cite{Bibi04} and \cite{Bibi05b}. The purpose of this appendix \ref{sec:appendiceA} is to show that this Fourier-Laplace transform naturally extends as a wild twistor $\cD$-module (in the sense of \cite{Bibi06b}, \cf also \cite{Mochizuki08}) near $\wh\infty\in\wh\PP^1$ and to relate the corresponding nearby cycles with the vanishing cycles of $(\cT,\cS)$ at its critical points (stationary phase formula). While the first goal could directly be obtained from recent work of T\ptbl Mochizuki \cite{Mochizuki08}, we follow here the method of \cite{Bibi05b} in order to get in the same way the stationary phase formula. We will use the notation introduced in \S\ref{subsec:vartw} and in \S\ref{subsec:Fourier-Laplace}.

\begin{theoreme}\label{th:FourierlocalT}
Let $(\cT,\cS)$ be a polarized regular twistor $\cD$-module of weight~$w$ on~$\PP^1$. Then its Fourier-Laplace transform $(\wh\cT,\wh\cS)$ is a polarized wild twistor $\cD$-module of weight~$w$ on~$\wh\PP^1$ and, for any $c\in\CC$, we have functorial isomorphisms in $\RTriples(\pt)$ compatible with the polarizations induced by $\wh\cS$ and~$\cS$ respectively:
\begin{align*}
(\Psi_{\tau'}^{c/\tau',\beta}\wh\cT,\cN_{\tau'})&\simeq (\Psi_{t+c}^\beta\cT,\cN_{t+c})\quad\text{if $\reel\beta\in(-1,0)$},\\[-9pt]
\tag{$*$}\\[-9pt]
\gr_\bbullet^\rM\Psi_{\tau'}^{c/\tau',\beta}\wh\cT&\simeq \gr_\bbullet^\rM\Psi_{t+c}^\beta\cT\quad\text{if }\beta\in i\RR^*,\\
(\Psi_{\tau'}^{c/\tau',0}\wh\cT,\cN_{\tau'})&\simeq (\phi_{t+c}^{-1}\cT,\cN_{t+c}).\tag{$**$}
\end{align*}
\end{theoreme}

\begin{remarque}
In \cite{Bibi05b} and \cite[Appendix]{Bibi01c}, the distinction between the two lines in the analogue of \ref{th:FourierlocalT}$(*)$ was mistakenly forgotten in the corresponding statements (see Footnote \ref{foot:gr} below).
\end{remarque}

\begin{proof}
According to the results of \cite{Bibi04} and \cite{Bibi05b}, it is enough to prove the conditions on the wild specialization at $\wh\infty$, that is, $(*)$ and $(**)$ with possible ramification at $\tau'=0$. We will denote by $i_0$ the inclusion $\{0\}\hto\PP^1$. We can reduce to the case $w=0$ by Tate twist by $(w/2)$, and assume that $\cS=(\id,\id)$, so $\wh\cS=(\id,\id)$. The compatibility with polarizations will then be clear from the proof.

As in \cite[Prop\ptbl4.1]{Bibi05b}, we will denote by $D_\beta$ the divisor $1\cdot i$ if $\beta\in i\RR_-^*$ and $1\cdot (-i)$ if $\beta\in i\RR_+^*$, and $D_\beta=0$ otherwise. For a $\cR$-module $\cN$, the $\cR$-module~$\cN(D_\beta)$ is defined as usual as $\cO_{\Omega_0}(D_\beta)\otimes_{\cO_{\Omega_0}}\cN$. We denote the monodromy filtration of a nilpotent endomorphism by $\rM_\bbullet$.

\begin{lemme}\label{lem:FourierlocalM}
Let $\cM$ be a coherent $\cR_{\cP^1}$-module which is strictly specializable at $t=0$. Then the $\cR_\cZ(*\wh\infty)$-module $\cFcM$ is strictly specializable along $\tau'=0$ and we have natural functorial isomorphisms of $\cR_{\cP^1}$-modules with nilpotent endomorphism
\begin{align*}
(\Psi_{\tau'}^{0,\beta}\cFcM_{|\Delta_0},\rN_{\tau'})&\simeq i_{0,+}(\Psi_t^\beta\cM_{|\Delta_0},\rN_t)\quad\text{if $\reel\beta\in(-1,0)$},\\[-9pt]
\tag{$*$}\\[-9pt]
\gr_\bbullet^\rM\Psi_{\tau'}^{0,\beta}\cFcM_{|\Delta_0}&\simeq i_{0,+}(\gr_\bbullet^\rM\Psi_t^\beta\cM(D_\beta)_{|\Delta_0})\quad\text{if }\beta\in i\RR^*,\\
(\Psi_{\tau'}^{0,0}\cFcM,\rN_{\tau'})&\simeq i_{0,+}(\phi_t^{-1}\cM,\rN_t).\tag{$**$}
\end{align*}
\end{lemme}

Recall (\cf \cite[\S3.6.b]{Bibi01c}) that $\phi_t^{-1}\cM$ is also denoted $\psi_t^{-1}\cM$ (or $\psi_{t,0}\cM$ if one uses the increasing notation), but later we will extend the correspondence with sesquilinear pairings, where the distinction between $\phi$ and $\psi$ is important. Recall also (\cf \cite{Bibi06b}) that the notation $\Psi_{\tau'}^{0,\beta}\cFcM$ has the same meaning as $\Psi_{\tau'}^\beta\cFcM$ (as used on the right-hand side), but we mean here that the possible exponential factor is zero, for later use.

\begin{proof}
We will consider the two charts $(t,\tau')$ and $(t',\tau')$. Let us first start with the second one. Let $(m_i)_{i\in I}$ be a finite set of $\cR_{\cP^1,(t'_o,\hbo)}$-generators of $\cM_{t'_o,\hbo}$. Then, from the formulas (\cf \cite[(A.2.5)]{Bibi01c})
\begin{align*}
m/t'\tau'\otimes\ccE^{-t\tau/\hb}&=\tau'\partiall_{\tau'}(m\otimes\ccE^{-t\tau/\hb})\\
\partiall_{t'}(m\otimes\ccE^{-t\tau/\hb})&=(\partiall_{t'}m)\otimes\ccE^{-t\tau/\hb}+\tau'(\tau'\partiall_{\tau'})^2(m\otimes\ccE^{-t\tau/\hb}),
\end{align*}
we conclude that, in this chart, $\cFcM$ is $V_0\cR_\cZ$-coherent (where the $V$-filtration on $\cZ$ is relative to $\tau'=0$) and that it is strictly specializable along $\tau'=0$ with a constant $V$-filtration.

Let us now consider the chart $(t,\tau')$ and the corresponding formulas \cite[(A.2.4)]{Bibi01c}
\begin{align*}
\partiall_t(m\otimes\ccE^{-t\tau/\hb})&=[(\partiall_t-1/\tau')m]\otimes\ccE^{-t\tau/\hb},\\
\partiall_{\tau'}(m\otimes\ccE^{-t\tau/\hb})&=tm/\tau^{\prime2}\otimes\ccE^{-t\tau/\hb}.
\end{align*}
The proof is very similar to that of \cite[Prop\ptbl4.1]{Bibi05b}. It will be simpler to work with the algebraic version of $\cFcM$, that is, to consider the projection~$p$ in the algebraic sense, so $p^+\wt\cM(*\wh\infty)=\CC[\tau',\tau^{\prime-1}]\otimes_\CC\wt\cM$. Moreover, as we work in the (analytic) chart with coordinate~$t$, there is no difference between $\wt\cM$ and $\cM$. We will exhibit the $V$-filtration of $\cFcM$ along $\tau'=0$. As such a filtration is only locally defined with respect to~$\hb$, we fix~$\hbo$ and work with in some neighbourhood of~$\hbo$. We will forget~$\hbo$ in the notation of the $V$-filtration. For any $b\in\RR$, let us set
\[
U^b\cFcM\defin\sum_{\ell\geq0}\partiall_t^\ell\sum_{k\in\ZZ}\tau^{\prime-k}(V^{b+k}\cM\otimes\ccE^{-t/\hb\tau'})\subset\cM[\tau',\tau^{\prime-1}]\otimes\ccE^{-t/\hb\tau'},
\]
where $V^\cbbullet\cM$ is the $V$-filtration (near~$\hbo$) of $\cM$ along $t=0$. The following properties are easily checked:
\begin{itemize}
\item
$U^\cbbullet\cFcM$ is a decreasing filtration of $\cFcM$ by $\cR_{\cP^1}[\tau']\langle\tau'\partiall_{\tau'}\rangle$-modules.
\item
$\tau' U^b\cFcM=U^{b+1}\cFcM$ (so~$\tau'$ induces $\gr_U^b\cFcM\isom\gr_U^{b+1}\cFcM$).
\item
From the strict specializability of $\cM$ (\cf\cite[Def\ptbl3.3.8 \&\ Rem\ptbl 3.3.9(2)]{Bibi01c}) we get, for $b\in(-1,0]$ and $k\in\NN$, $V^{b+k}\cM=t^kV^b\cM$ and $V^{b-1-k}\cM=\sum_{j=0}^{k-1}\partiall_t^jV^{-1}\cM+\partiall_t^kV^{b-1}\cM$. If we write, according to \cite[(A.2.4)]{Bibi01c},
\begin{align*}
U^b\cFcM=\sum_{\ell\geq0}\partiall_t^\ell\Big(\sum_{k\geq0}(\tau'\partiall_{\tau'})^k(V^b\cM\otimes{}&\ccE^{-t/\hb\tau'})\\[-12pt]
&+\tau'\sum_{k\geq0}\tau^{\prime k}(V^{b-1-k}\cM\otimes\ccE^{-t/\hb\tau'})\Big)
\end{align*}
and, for $k\geq0$,
\begin{multline*}
\tau^{\prime k}(V^{b-1-k}\cM\otimes\ccE^{-t/\hb\tau'})=(\tau'\partiall_t+1)^k(V^{b-1}\cM\otimes\ccE^{-t/\hb\tau'})\\
+\sum_{j=0}^{k-1}\tau^{\prime j}(\tau'\partiall_t+1)^{k-j}(V^{-1}\cM\otimes\ccE^{-t/\hb\tau'}),
\end{multline*}
we find that $U^b\cFcM$ is $\cR_{\cP^1}[\tau']\langle\tau'\partiall_{\tau'}\rangle$-coherent, and locally generated as such by $m_i\otimes\nobreak\ccE^{-t/\hb\tau'}$ and $\tau'(n_j\otimes\nobreak\ccE^{-t/\hb\tau'})$, if $m_i$ (\resp $n_j$) are local $\cO_{\cP^1}\langle t\partiall_t\rangle$-generators of $V^b\cM$ (\resp $V^{b-1}\cM$).

\item
If $B_b(s)$ is the minimal polynomial of $t\partiall_t$ on $\gr^b_V\cM$, then, for any local section~$m$ of $V^b\cM$, as $(t\partiall_t+\tau'\partiall_{\tau'})(m\otimes\ccE^{-t/\hb\tau'})=(t\partiall_t m)\otimes\ccE^{-t/\hb\tau'}$ and as $\partiall_tt(m\otimes\nobreak\ccE^{-t/\hb\tau'})\in U^{b+1}\cFcM$, we find $B_b(\tau'\partiall_{\tau'}-1)(m\otimes\ccE^{-t/\hb\tau'})\in U^{>b}\cFcM$.
\end{itemize}

It follows from these properties that $V^b\cFcM\defin U^{b-1}\cFcM$ is a good candidate for being the $V$-filtration of $\cFcM$. It remains to check the strict specializability, by computing the graded modules.

For any $b\in\RR$, the map (where $\eta$ is a new variable)
\begin{align*}
V^b\cM[\eta]&\to U^b\cFcM=V^{b+1}\cFcM\\
\sum_pm_p\eta^p&\mto\sum_p\partiall_t^p(m_p\otimes\ccE^{-t/\hb\tau'})
\end{align*}
induces a mapping
\begin{equation}\label{eq:grUV}
(\gr^b_V\cM[\eta],t\partiall_t)\to(\gr^b_U\cFcM,\tau'\partiall_{\tau'}-1)
\isoTo{\tau^{\prime-1}}
(\gr^b_V\cFcM,\tau'\partiall_{\tau'}),
\end{equation}
and we have a commutative diagram
\begin{equation}\label{eq:comdiag}
\begin{array}{c}\xymatrix{
\gr^b_V\cM[\eta]\ar[r]\ar[d]_{\partiall_t}&\gr^b_V\cFcM\ar[d]^{\tau^{\prime-1}}\\
\gr^{b-1}_V\cM[\eta]\ar[r]&\gr^{b-1}_V\cFcM
}\end{array}
\end{equation}
Moreover, if we identify in a natural way $\gr^b_V\cM[\eta]$ with $i_{0,+}\gr^b_V\cM$, this map is a morphism of $\cR_{\cP^1}$-modules.

\begin{lemme}\label{lem:b<0}
If $b<0$, \eqref{eq:grUV} is an isomorphism of $\cR_{\cP^1}$-modules.
\end{lemme}

\begin{proof}
If $b\in<0$, we assert that
\[
U^b\cFcM=U^{>b}\cFcM+\sum_{\ell\geq0}\partiall_t^\ell(V^b\cM\otimes\ccE^{-t/\hb\tau'}),
\]
which implies that the morphism \eqref{eq:grUV} is onto. Indeed, on the one hand, using the formulas \cite[(A.2.4)]{Bibi01c} recalled above and iterating the inclusion
\begin{align*}
\tau^{\prime-1}(V^{b+1}\cM\otimes\ccE^{-t/\hb\tau'})&\subset\partial_t(V^{b+1}\cM\otimes\ccE^{-t/\hb\tau'})-(\partial_tV^{b+1}\cM\otimes\ccE^{-t/\hb\tau'})\\
&\subset U^{b+1}\cFcM+(V^b\cM\otimes\ccE^{-t/\hb\tau'}),
\end{align*}
we get
\[
\sum_{k\geq0}\tau^{\prime-k}(V^{b+k}\cM\otimes\ccE^{-t/\hb\tau'})\subset (V^b\cM\otimes\ccE^{-t/\hb\tau'})+U^{b+1}\cFcM.
\]
On the other hand, if $b<0$, for any $k\geq1$ we have $V^{b-k}\cM=\partiall_t^kV^{b}\cM+V^{>b-k}\cM$ and
\[
\tau^{\prime k}(V^{b-k}\cM\otimes\ccE^{-t/\hb\tau'})\subset (\tau'\partiall_t+1)^k(V^b\cM\otimes\ccE^{-t/\hb\tau'})+U^{>b}\cFcM.
\]
Notice then that, for $j\geq1$, $\partiall_t^j(V^b\cM\otimes\ccE^{-t/\hb\tau'})\subset U^b\cFcM$ and $\tau^{\prime j}\partiall_t^j(V^b\cM\otimes\ccE^{-t/\hb\tau'})\subset U^{b+j}\cFcM\subset U^{>b}\cFcM$.

The morphism \eqref{eq:grUV} is also injective: one remarks that, given local sections $n_j$ of~$\cM$, if $\sum_j\tau^{\prime j}n_j\otimes\nobreak\ccE^{-t/\hb\tau'}$ is a local section of $U^b\cFcM$, then the dominant coefficient with respect to $\tau^{\prime-1}$ belongs to $V^b\cM$ (by considering the dominant coefficient with respect to $\tau^{\prime-1}$ in an expression like $\sum_p\partiall_t^p(m_p\otimes\ccE^{-t/\hb\tau'})$); arguing as in \cite[Proof of Prop\ptbl4.1, \S(ii)(6)]{Bibi05b}, one gets the injectivity.
\end{proof}

At this point, we have proved the strict specializability of $\cFcM$ along $\tau'=0$. We now prove the existence of the isomorphisms $(*)$ and $(**)$ of Lemma \ref{lem:FourierlocalM}.

Let $\beta$ be such that $\reel\beta\in(-1,0]$. The morphism \eqref{eq:grUV} induces, near~$\hbo$, a morphism $(i_{0,+}\psi_t^\beta\cM,\rN_t)\to(\psi_{\tau'}^\beta\cFcM,\rN_{\tau'})$. One can show that these locally defined morphisms glue together. Setting $\ell_\hbo(\beta)=\reel\beta-(\im\hbo)(\im\beta)$ (\cf \cite[p\ptbl17]{Bibi01c}), if $\ell_\hbo(\beta)<0$, it is an isomorphism near~$\hbo$, according to the Lemma \ref{lem:b<0} with $b=\ell_\hbo(\beta)$. Let us first show that such remains the case if $\ell_\hbo(\beta)=0$. In this case, by definition of strict specializability (\cf \cite[Def\ptbl3.3.8(c)]{Bibi01c}, we know that $\partiall_t:\psi_t^\beta\cM\to\psi_t^{\beta-1}\cM$ is an isomorphism near~$\hbo$, so we conclude using \eqref{eq:comdiag}.

Let us now assume that $\ell_\hbo(\beta)>0$ and let us choose $k\in\NN$ such that $\ell_\hbo(\beta-k)\in(-1,0]$. Near~$\hbo$, we get from \eqref{eq:comdiag} a commutative diagram
\[
\xymatrix{
i_{0,+}\psi_t^\beta\cM\ar[r]\ar[d]_{i_{0,+}\partiall_t^k}&\psi_{\tau'}^\beta\cFcM\ar[d]^{\tau^{\prime-k}}\\
i_{0,+}\psi_t^{\beta-k}\cM\ar[r]&\psi_{\tau'}^{\beta-k}\cFcM
}
\]
where the right vertical map is an isomorphism, by definition, and the choice of $k$ implies that the lower horizontal map is an isomorphism. On the other hand, by the strict specializability of $\cM$ at $t=0$ (\cf \cite[Def\ptbl3.3.8(1b) and Rem\ptbl3.3.9(2)]{Bibi01c}) and the choice of $k$, the map $t^k:\psi_t^{\beta-k}\cM\to\psi_t^\beta\cM$ is an isomorphism, and $\partiall_t^kt^k:\psi_t^{\beta-k}\cM\to\psi_t^{\beta-k}\cM$ is equal to $\prod_{j=0}^{k-1}\big[(\beta-j)\star\hb+\rN_t\big]$. We note that $(\beta-j)\star\hb+\rN_t$ is invertible near~$\hbo$ unless $(\beta-j)\star\hbo=0$. With the conditions $\hbo\in\Delta_0$, $\beta\neq0$, $\reel(\beta-j)\leq0$, $j=0,\dots,k-1$, $\ell_\hbo(\beta-j)>0$, this vanishing only occurs if $j=0$, $\reel\beta=0$ and $\hbo=i$ if $\im\beta<0$, $\hbo=-i$ if $\im\beta>0$. Therefore, on the one hand, the natural morphism \eqref{eq:grUV} induces
\begin{equation}\label{eq:reelbetaneq0}
i_{0,+}\psi_t^\beta\cM_{|\Delta_0}\isom\psi_{\tau'}^\beta\cFcM_{|\Delta_0}\quad\text{if }\reel\beta\in(-1,0).
\end{equation}

On the other hand, if $\beta\in i\RR^*$, one checks similarly that, grading first\footnote{\label{foot:gr}This grading was forgotten in \cite{Bibi05b}.} by the monodromy filtration in order to kill $N_t$, \eqref{eq:grUV} induces an isomorphism
\begin{equation}\label{eq:reelbeta=0}
i_{0,+}\gr_\bbullet^\rM\psi_t^\beta\cM_{|\Delta_0}\isom\gr_\bbullet^\rM\psi_{\tau'}^\beta\cFcM_{|\Delta_0}(-D_\beta).
\end{equation}

Lastly, if $\beta=0$, we consider the isomorphism
\begin{equation}\label{eq:beta=0}
i_{0,+}\psi_t^{-1}\cM_{|\Delta_0}\isom\psi_{\tau'}^{-1}\cFcM_{|\Delta_0}\isoTo{\tau'}\psi_{\tau'}^0\cFcM_{|\Delta_0}.
\end{equation}

Notice that, by definition, $\Psi_{\tau'}^\beta\cFcM=\psi_{\tau'}^\beta\cFcM$. So, at this point, we have obtained \ref{lem:FourierlocalM}$(**)$. In order to get \ref{lem:FourierlocalM}$(*)$, it remains to compare $\psi_t^\beta\cM_{|\Delta_0}$ with $\Psi_t^\beta\cM_{|\Delta_0}$. Arguing in a way similar to that of \cite[Lemma 4.18]{Bibi05b}, we conclude that, for $\reel\beta\in(-1,0]$, the natural inclusion $\psi_t^\beta\cM_{|\Delta_0}\hto\Psi_t^\beta\cM_{|\Delta_0}$ is an isomorphism. This gives \ref{lem:FourierlocalM}$(*)$.

Finally, the functoriality of the isomorphisms \ref{lem:FourierlocalM}$(*)$ and $(**)$ is clear from the construction.
\end{proof}

Let $\cT=(\cM',\cM'',\cCS)$ be an object of $\RTriples(\PP^1)$, such that $\cM',\cM''$ are $\cR_{\cP^1}$-coherent and strictly specializable along $t=0$. Then $\cFcT$ is defined as $(\cFcM',\cFcM'',\cFCS)$, where $\cFcM',\cFcM''$ are as above and $\cFCS$ is defined in \cite[p\ptbl196]{Bibi01c}.

\begin{lemme}\label{lem:FourierlocalT}
Let $\cT$ be as above. Then the isomorphisms of Lemma \ref{lem:FourierlocalM} extend as isomorphisms
\begin{align*}
(\Psi_{\tau'}^{0,\beta}\cFcT,\cN_{\tau'})&\simeq i_{0,+}(\Psi_t^\beta\cT,\cN_t)\quad\text{if $\reel\beta\in(-1,0)$},\\[-9pt]
\tag{$*$}\\[-9pt]
\gr_\bbullet^\rM\Psi_{\tau'}^{0,\beta}\cFcT&\simeq i_{0,+}(\gr_\bbullet^\rM\Psi_t^\beta\cT)\quad\text{if }\beta\in i\RR^*,\\
(\Psi_{\tau'}^{0,0}\cFcT,\cN_{\tau'})&\simeq i_{0,+}(\phi_t^{-1}\cT,\cN_t).\tag{$**$}
\end{align*}
\end{lemme}

\begin{proof}
The point is to prove the compatibility of the corresponding sesquilinear pairings under \ref{lem:FourierlocalM}($*$) and ($**$), up to $\Gamma$-factors that we will analyse, as in \cite[Proof of Prop\ptbl5.8]{Bibi05b}.

Let us fix $\hbo\in\bS$ and let us work in the neighbourhood of~$\hbo$. For~$\beta\neq0$ with $\reel\beta\in(-1,0]$, let us set $\alpha=-\beta-1$ and $b=\ell_\hbo(\beta)$. Let $m',m''$ be local section near $(t=0,\hbo)$ of $V^b\cM',V^b\cM''$ inducing sections $[m'],[m'']$ of $\psi_t^\beta\cM',\psi_t^\beta\cM''$ on $\gr^b_V\cM',\gr^b_V\cM''$ (recall that $\psi_t^\beta\cM_{|\bS}=\Psi_t^\beta\cM_{|\bS}$, \cf \cite[Lemma 3.4.2(2)]{Bibi01c}). We regard $[m'],[m'']$ as sections of $i_{0,+}\psi_t^\beta\cM',i_{0,+}\psi_t^\beta\cM''$ (degree~$0$ with respect to $\eta$). Following \eqref{eq:grUV}, they correspond to sections $[\tau^{\prime-1}m'\otimes\ccE^{-t/\hb\tau'}],[\tau^{\prime-1}m''\otimes\ccE^{-t/\hb\tau'}]$ of $\Psi_{\tau'}^{0,\beta}\cFcM',\Psi_{\tau'}^{0,\beta}\cFcM''$.
By definition, we have, for any $C^\infty$ form $\varphi(t)$ of type $(1,1)$ on $\PP^1$ with compact support in the chart~$t$,
\begin{align*}
\big\langle\psi_{\tau'}^\beta&\cFCS([\tau^{\prime-1}m'\otimes\ccE^{-t/\hb\tau'}],\ovv{[\tau^{\prime-1}m''\otimes\ccE^{-t/\hb\tau'}]}),\varphi\big\rangle\\
&=\res_{s=\alpha\star\hb/\hb}\big\langle\cFCS(m'\otimes\ccE^{-t/\hb\tau'},\ovv{m''\otimes\ccE^{-t/\hb\tau'}}),\varphi\wedge |\tau'|^{2(s-1)}\wh\chi(\tau')\itwopi d\tau'\wedge d\ov{\tau'}\big\rangle\\
&=\res_{s=\alpha\star\hb/\hb}\big\langle\cCS(m',\ovv{m''}),I_{\wh\chi}(t,s,\hb)\varphi\big\rangle,
\end{align*}
where $\wh\chi$ is $C^\infty$ with compact support in $\wh\PP^1$, $\equiv1$ near $\tau'=0$ and $\equiv0$ far from $\tau'=0$, and $I_{\wh\chi}$ is defined for $\reel s>0$ as in \cite[\S3.6.b]{Bibi01c} by
\[
I_{\wh\chi}(t,s,\hb)=\int e^{\hb\ov t/\ov{\tau'}-t/\hb\tau'}\vert\tau'\vert^{2(s-1)}\wh\chi(\tau')\itwopi d\tau'\wedge d\ov{\tau'}.
\]
The last equality above means that $\big\langle\cCS(m',\ovv{m''}),I_{\wh\chi}(t,s,\hb)\varphi\big\rangle$ is holomorphic with respect to~$s$ for $\reel s>0$ and extends as a meromorphic function of~$s$, of which we take the residue at $s=\alpha\star\hb/\hb$.

One first proves, as in \cite[Lemma 3.6.6]{Bibi01c} that $\big\langle\cCS(m',\ovv{m''}),I_{\wh\chi}(t,s,\hb)\varphi\big\rangle$ has poles on sets $s=\gamma\star\hb/\hb$ ($\hb\in\bS$) with $\reel\gamma<\reel\alpha$ or $\gamma=\alpha$. Moreover, only the first case occurs if $\varphi$ vanishes along $t=0$, and we can thus assume that $\varphi\equiv\itwopi dt\wedge d\ov t$ near $t=0$. As the residue at $s=\alpha\star\hb/\hb$ does not depend on the such a $\varphi$, we can assume $\varphi=\chi^2\itwopi dt\wedge d\ov t$ with $\chi\equiv1$ near $t=0$. We will compare $\big\langle\cCS(m',\ovv{m''}),I_{\wh\chi}(t,s,\hb)\chi\big\rangle$ and $\big\langle\cCS(m',\ovv{m''}),\mt^{2s}\chi^2\big\rangle$ before taking their residue.

If we denote by $T$ the distribution $\chi\cCS(m',\ovv{m''})$, and by $\cF$ the Fourier transform with kernel $e^{\ov{t\tau}\hb-t\tau/\hb}\itwopi d\tau\wedge d\ov\tau$ (\cf \cite[Rem\ptbl3.6.17]{Bibi01c}), these functions are respectively written as
\[
\int\cF T(\tau,\hb)|\tau|^{-2(s+1)}\wh\chi(\tau)\itwopi d\tau\wedge d\ov\tau\quad\text{and}\quad\int\cF T(\tau,\hb)\wh I_\chi(\tau,s,\hb)\itwopi d\tau\wedge d\ov\tau
\]
with $\wh I_\chi(\tau,s,\hb)\defin\cF^{-1}(\mt^{2s}\chi)$ (this is analogous to (3.6.25) and (3.6.26) in \loccit).

We note that $(1-\wh\chi(\tau))\cF T(\tau,\hb)$ is $C^\infty$ with compact support, so its inverse Fourier transform $\eta$ is in the Schwartz class, and
\[
\int\cF T(\tau,\hb)\wh I_\chi(\tau,s,\hb)(1-\wh\chi(\tau))\itwopi d\tau\wedge d\ov\tau
\]
reads $\int\eta\chi\mt^{2s}\itwopi dt\wedge d\ov t$, so takes the form $\Gamma(s+1)h(s,\hb)$, where $h$ is entire with respect to~$s$. Using the computation in \cite[Lemma 5.14]{Bibi05b} for $\wh I_\chi$ (replacing $t$ there with $\tau$ here) and in particular \cite[(5.17)]{Bibi05b}, we finally find, if $\beta\neq0$,
\begin{multline*}
\frac{\Gamma(1+\alpha\star\hb/\hb)}{\Gamma(-\alpha\star\hb/\hb)}\res_{s=\alpha\star\hb/\hb}\big\langle\cCS(m',\ovv{m''}),I_{\wh\chi}(t,s,\hb)\chi\big\rangle\\
=\res_{s=\alpha\star\hb/\hb}\big\langle\cCS(m',\ovv{m''}),\mt^{2s}\chi\big\rangle.
\end{multline*}
Let us write $\sfrac{\Gamma(1+\alpha\star\hb/\hb)}{\Gamma(-\alpha\star\hb/\hb)}=\mu\ov\mu$ as in \cite[Lemma 5.5]{Bibi05b}, and $D_\mu=-D_\beta$. Using \eqref{eq:reelbetaneq0} or \eqref{eq:reelbeta=0} we find (grading only if $\beta\in i\RR^*$, that is, if $D_\mu\neq\emptyset$)
\[
(\gr_\bbullet^\rM)(\Psi_{\tau'}^\beta\cFcM'(D_\mu),\Psi_{\tau'}^\beta\cFcM''(D_\mu),\mu\ovv\mu\cFCS)
\simeq i_{0,+}(\gr_\bbullet^\rM)(\Psi_t^\beta\cM',\Psi_t^\beta\cM'',\cCS).
\]
\ref{lem:FourierlocalT}$(*)$ follows then from \cite[Lemma 5.6]{Bibi05b}.

Arguing similarly for $\beta=0$, we note that \ref{lem:FourierlocalT}$(**)$ is by definition (\cf \cite[(3.6.18) \& Rem\ptbl3.6.20]{Bibi01c}).
\end{proof}

Arguing as in \cite{Bibi05b} by applying \cite[\S6.3]{Bibi01c}, we get \ref{th:FourierlocalT}$(*)$ and $(**)$ for $c=0$. It is then not difficult to check that replacing $t$ with $t+c$ corresponds to twisting $\wh\cT$ by $\ccE^{c/\hb\tau'}$, and to get $(*)$ and $(**)$ for any $c\in\CC$ in the same way. One checks that \ref{th:FourierlocalT}$(*)$ and $(**)$ hold after ramification by using \cite[Rem\ptbl2.3.3]{Bibi06b}.
\end{proof}

\subsubsections{The integrable case}
Let us now assume that $(\cT,\cS)$ is integrable. We will describe the compatibility between the actions of $\hb^2\partial_\hb$ in Theorem \ref{th:FourierlocalT}. Recall that, with this assumption, the numbers~$\beta$ such that $\Psi_t^\beta\cT\neq0$ are real, and thus so are the numbers~$\beta$ such that $\Psi_t^\beta\wh\cT\neq0$. In the computation above, we can set $b=\beta$ and do not worry about the local dependence with respect to~$\hbo$ (\cf\cite[Chap\ptbl7]{Bibi01c}).

The morphism \eqref{eq:grUV} is compatible with the natural action of $\hb^2\partial_\hb$ on $i_{0,+}\gr_V^b\cM$ on the one hand, and the action of $\hb^2\partial_\hb-\tau'\partial_\tau'$ on $\gr_{b+1}^V\cFcM$ on the other hand (where the action of $\hb^2\partial_\hb$ comes from the natural action of $\hb^2\partial_\hb$ on $\cFcM$, \cf Lemma \ref{lem:hbdhbtdt}). As we set $\beta=b$, we thus have, for $\beta<0$,
\[
(i_{0,+}\psi_t^\beta\cM,\hb^2\partial_\hb)\isom(\psi_{\tau'}^{\beta+1}\cFcM, \hb^2\partial_\hb-(\beta+1)\hb-\rN_{\tau'}).
\]
As multiplication by~$\tau'$ commutes with $\hb^2\partial_\hb$ and $\rN_{\tau'}$, we obtain
\[
(i_{0,+}\psi_t^\beta\cM,\hb^2\partial_\hb)\isom
\begin{cases}
(\psi_{\tau'}^\beta\cFcM, \hb^2\partial_\hb-(\beta+1)\hb-\rN_{\tau'})&\text{if }\beta\in(-1,0),\\
(\psi_{\tau'}^0\cFcM, \hb^2\partial_\hb-\rN_{\tau'})&\text{if }\beta=-1.
\end{cases}
\]
In a way analogous to that of Lemma \ref{lem:FourierlocalT}, we conclude
\begin{align*}
(i_{0,+}\psi_t^\beta\cT,\hb^2\partial_\hb)&\isom (\psi_{\tau'}^\beta\cFcT, \hb^2\partial_\hb-(\beta+1)\hb-\rN_{\tau'})\quad\text{if }\beta\in(-1,0),\\
(i_{0,+}\phi_t^{-1}\cT,\hb^2\partial_\hb)&\isom (\psi_{\tau'}^0\cFcM, \hb^2\partial_\hb-\rN_{\tau'}).
\end{align*}
As $\psi_{\tau'}^\beta$ commutes with the direct image by $\wh p$ in our context, we get
\begin{align*}
(\Psi_t^\beta\cT,\hb^2\partial_\hb)&\isom(\Psi_{\tau'}^{0,\beta}\wh\cT, \hb^2\partial_\hb-(\beta+1)\hb-\rN_{\tau'})\quad\text{if }\beta\in(-1,0),\\
(\phi_t^{-1}\cT,\hb^2\partial_\hb)&\isom(\Psi_{\tau'}^{0,0}\wh\cT, \hb^2\partial_\hb-\rN_{\tau'}),
\end{align*}
and, grading with respect to $\rM_\bbullet$ kills $\rN_{\tau'}$ and gives, for any $\ell\in\ZZ$,
\begin{align}
(\gr_\ell^\rM\Psi_t^\beta\cT,\hb^2\partial_\hb)&\isom(\gr_\ell^\rM\Psi_{\tau'}^{0,\beta}\wh\cT, \hb^2\partial_\hb-(\beta+1)\hb)\quad\text{if }\beta\in(-1,0),\label{eq:PsiintMicro}\\
(\gr_\ell^\rM\phi_t^{-1}\cT,\hb^2\partial_\hb)&\isom(\gr_\ell^\rM\Psi_{\tau'}^{0,0}\wh\cT, \hb^2\partial_\hb).\label{eq:phiintMicro}
\end{align}
Lastly, we get a similar result after translating $t$ by $c\in\CC$.

\section{Rescaling}\label{sec:appendiceB}
In this appendix \ref{sec:appendiceB}, we recall the notion of rescaling of an integrable twistor structure considered in \cite[Def\ptbl4.1]{H-S06}. We also explain how the ``no ramification'' condition is related to a good behaviour of the rescaling.

Let $(\cH',\cH'',\cCS,\nabla)$ be an integrable twistor structure (\cf \S\ref{subsec:intvartw} with~$X$ reduced to a point). In order to clarify notation, we will denote by $\eta$ the coordinate denoted by~$\hb$ before. By integrability, $\nabla_{\eta^2\partial_\eta}$ acts on $\cH',\cH''$ in a way compatible with $\cCS$. For the sake of simplicity, we will denote by $\eta^2\partial_\eta$ this action. The bundles $\cH',\cH''$ are a priori defined on some open neighbourhood of $\{|\eta\leq1|\}$ but, using the gluing defined by $\cCS$ and its compatibility with $\nabla$, we can assume that they are defined, together with the action of $\eta^2\partial_\eta$, on the whole complex line~$\CC_\eta$ with coordinate $\eta$ and that $\cCS$ is the restriction to $\bS$ of a sesquilinear pairing compatible with $\nabla$
\[
\cC:\cH'_{|\CC_\eta^*}\otimes_{\cO_{\CC_\eta^*}}\ovv{\cH''_{|\CC_\eta^*}}\to\cO_{\CC_\eta^*}.
\]
Let us consider the map $\mu:\CC_{\tau'}\times\nobreak\Omega_0\to\CC_\eta$ defined by $\mu(\tau',\hb)=\eta=\tau'\hb$ (for $\tau'\neq0$, we will set $\tau=\tau^{\prime-1}$; this corresponds to the coordinate~$\tau$ in \S\ref{subsec:Fourier-Laplace}). The pull-backs $\mu^*\cH',\mu^*\cH''$ are holomorphic bundles on $\CC_{\tau'}\times\nobreak\Omega_0$. When restricted to the open set $\tau'\neq0$, they are equipped with a flat meromorphic connection having a pole of Poincar\'e rank one along $\hb=0$. We have
\[
\hb^2\partial_\hb(1\otimes m)=\tau^{\prime-1}(1\otimes\eta^2\partial_\eta m)\quad\text{and}\quad \tau^{\prime2}\partiall_{\tau'}(1\otimes m)=1\otimes\eta^2\partial_\eta m.
\]
For a fixed $\tau_o\in\CC^*$, denote by $\mu_{\tau_o}^*$ the composition of $\mu^*$ with the restriction to $\tau'=\tau_o^{-1}$. This defines $\mu_{\tau_o}^*\cH',\mu_{\tau_o}^*\cH''$.

In order to define the rescaling of the sesquilinear pairing, we need to be careful. Indeed, the rescaling is not compatible with twistor conjugation, as we have $\mu_{\tau_o}^*\ovv{\cH''_{|\CC^*}}=\ovv{\mu_{\ov\tau_o^{-1}}^*\cH''_{|\CC^*}}$. We therefore need an identification $\mu_{\tau_o}^*\cH''_{|\CC^*}\simeq\mu_{\ov\tau_o^{-1}}^*\cH''_{|\CC^*}$ in order to get a sesquilinear pairing $\mu_{\tau_o}^*\cC$ such that $(\mu_{\tau_o}^*\cH',\mu_{\tau_o}^*\cH'',\mu_{\tau_o}^*\cC)$ is a twistor structure. This identification is obtained through the parallel transport with respect to the holomorphic connection on $\mu^*\cH''_{|\CC^*\times\CC^*}$ from $\tau_o\times\nobreak\CC^*$ to $\ov\tau_o^{-1}\times\nobreak\CC^*$ along the segment between $\tau_o$ and $\ov\tau_o^{-1}$.

One can rewrite this definition in a way independent of $\tau_o\in\CC^*$. For that purpose, let us denote by $\cL',\cL''$ the local systems on $\CC^*_\eta$ determined by $(\cH',\nabla)$ and $(\cH'',\nabla)$. We get a pairing $\cL'_{\bS}\otimes\sigma^{-1}\ov{\cL''_\bS}\to \CC_\bS$ by restricting $\cCS$ to these local systems. We remark that, on $\bS$, $\sigma$ coincides with the involution $\iota:\eta\mto-\eta$, and we can extend (by parallel transport) in a unique way $\cCS$ as a pairing
\[
\wt\cC:\cL'\otimes\iota^{-1}\ov{\cL''}\to \CC_{\CC^*_\eta}.
\]
Taking the pull-back by $\mu$ commutes with $\iota$ (with respect to $\eta$ and to~$\hb$), and restricting to $\CC^*_{\tau'}\times\nobreak\bS$ gives a pairing
\[
\mu^{-1}\wt\cC:\mu^{-1}\cL'_{\CC^*_{\tau'}\times\bS}\otimes\iota^{-1}\ov{\mu^{-1}\cL''_{\CC^*_{\tau'}\times\bS}}\to \CC_{\CC^*_{\tau'}\times\bS}.
\]
Identifying now $\iota_{|\bS}$ and $\sigma_{|\bS}$ in the~$\hb$-variable gives the desired sesquilinear pairing $\mu^*\cCS$ at the level of local systems, and thus at the level of holomorphic bundles. Clearly, it is nondegenerate.

\begin{definition}[Rescaling]
Let $\cT=(\cH',\cH'',\cCS)$ be an integrable twistor structure. The rescaling $\mu^*\cT$ is the triple $(\mu^*\cH',\mu^*\cH'',\mu^*\cCS)$ defined as above.
\end{definition}

We have the following properties:
\begin{itemize}
\item
By construction (and because $\mu^*\cCS$ is nondegenerate), $\mu^*\cT$ is an integrable variation of twistor structure on $\CC^*_{\tau'}$.
\item
Functoriality: This mainly reduces to showing that, given $\cT=(\cH',\cH'',\cCS)$ with $\cH',\cH''$ defined on some neighbourhood of $\{\eta\leq1\}$, the extension of $\cH',\cH''$ to~$\CC_\eta$ by using the gluing is functorial. This is done as follows. The pairing $\cCS$ induces an isomorphism $\cH'\simeq\ovv\cH{}^{\prime\prime\vee}$ on some neighbourhood of $\{\eta=1\}$, which allows one to extend~$\cH'$ as a bundle on $\CC_\eta$. If $\varphi:\cT_1\to\cT_2$ is a morphism, the previous isomorphism is compatible with $\varphi':\cH'_2\to\cH'_1$ and $\ovv\varphi^{\prime\prime\vee}:\ovv\cH{}_2^{\prime\prime\vee}\to\ovv\cH{}_1^{\prime\prime\vee}$ by definition. Therefore, it extends to $\CC_\eta$.
\item
Compatibility with adjunction: Restricted to local systems on $\{\eta=1\}$, the adjoint $\cCS^*$ of $\cCS$ is $\sigma^{-1}\cCS^\dag$, where $\cCS^\dag$ is the adjoint with respect to the standard conjugation. So, working on local systems,
\[
\cCS^*=\sigma^{-1}\cCS^\dag=\iota^{-1}\wt\cC^\dag_{|\eta=1}
\]
and
\begin{align*}
\mu^*\cCS^*=(\mu^{-1}\iota^{-1}\wt\cC^\dag)_{|\hb=1}&=(\iota^{-1}\mu^{-1}\wt\cC^\dag)_{|\hb=1}\\
&=\sigma^{-1}(\mu^{-1}\wt\cC)^\dag_{|\hb=1}=(\mu^*\cCS)^*.
\end{align*}
\item
Compatibility with Tate twist: For $k\in\hZZ$,
\[
\cT(k)\defin(\cH',\cH'',(i\hb)^{-2k}\cCS),
\]
so
\[
\mu^*(\cT(k))=(\mu^*\cH',\mu^*\cH'',\tau^{\prime-2k}(i\hb)^{-2k}\mu^*\cCS).
\]
We therefore have an isomorphism
\[
\varphi:(\mu^*\cT)(k)\isom\mu^*(\cT(k)),
\]
with $\varphi=(\varphi',\varphi'')$ and $\varphi'=\tau^{\prime-2k}\id_{\mu^*\cH'}$, $\varphi''=\id_{\mu^*\cH''}$.
\end{itemize}

\begin{proposition}\label{prop:Fourier-rescaling}
If $(\cT,\cS)$ is the polarized twistor $\cD$-module of weight~$w$ associated to a variation of Hodge structure of weight~$w$ as in Proposition \ref{prop:RFMtw}, then, when restricted to $\tau\neq0,\infty$, the Fourier-Laplace transform $\wh\cT$ is identified with the rescaling of its fibre at $\tau=1$ as defined above.
\end{proposition}

\begin{proof}[Sketch of proof]
One first reduces to weight~$0$, by using the compatibility of rescaling and Fourier-Laplace transform with Tate twist by $w/2$. One can also assume that $\cS=(\id,\id)$. Then the result follows from \cite[\S2b \& 2c]{Bibi05} (in particular, Lemma~2.4 in \loccit).
\end{proof}

In the remaining part of this section, we explain why a good behaviour at $\tau'=0$ of the rescaled twistor structure imposes the ``no ramification'' condition of \S\ref{subsec:specorigin}. Instead of working on $\{\tau'\neq0\}$, we now work on the whole line $\CC_{\tau'}$. Let us set $\cH=\cH'\text{ or }\cH''$ and $\wt\cM=\mu^*\cH[\tau^{\prime-1}]$. If we set $X=\CC_{\tau'}$, then $\wt\cM$ is an integrable $\cR_\cX[\tau^{\prime-1}]$-module.

\begin{proposition}
If $\wt\cM$ is strictly specializable with ramification and exponential twist at $\tau'=0$ (in the sense of \cite{Bibi01c} and \cite{Bibi06b}), then $(\cH,\nabla_{\partial_\eta})$ has no ramification.
\end{proposition}

\begin{proof}
It will be easier to work in an algebraic framework. One can find a free $\CC[\eta]$-module $H$ with an algebraic connection having a double pole at $\eta=0$ and no other pole, such that $(\cH,\nabla_{\partial_\eta})=(\cO_{\CC_\eta}\otimes_{\CC[\eta]}\nobreak H,\nabla_{\partial_\eta})$. Notice then that $\wt\cM$ is the analytization of $\wt M\simeq\CC[\tau',\tau^{\prime-1}]\otimes_\CC\nobreak H$, where the action of~$\hb$ is defined as $\tau^{\prime-1}\otimes\eta$.

Let us try to find a Bernstein relation (in the sense of \cite{Bibi01c} or, more generally with parabolic structure, of \cite{Mochizuki07}) for elements of $\wt M$ at $\tau'=\nobreak0$. Let $m\in H$. The differential equation of minimal degree satisfied by~$m$ can be written as $b(\eta\partial_\eta)m=\eta P(\eta,\eta\partial_\eta)m$, where $b\in\CC[s]\moins\{0\}$ and~$P$ is an operator in $\eta\partial_\eta$ with coefficients in $\CC[\eta]$. Let us set $b(s)=\prod_{\beta\in\CC}(s-\nobreak\beta)^{\nu_\beta}$. One deduces
\[
\prod_{\beta\in\CC}(\tau'\partiall_{\tau'}-\beta\hb)^{\nu_\beta}(1\otimes m)=\hb^{\deg b}\cdot (\tau'\hb)\cdot P(\tau'\hb,\hbm\tau'\partiall_{\tau'})(1\otimes m).
\]
For such a relation to be a Bernstein relation in the sense of \cite{Bibi01c,Mochizuki07} two conditions must be fulfilled.
\begin{enumerate}
\item\label{enum:pentes}
The right-hand side should have no pole in~$\hb$; this is possible if and only if the smallest positive slope of the Newton polygon of the equation $b(\eta\partial_\eta)-\eta P(\eta,\eta\partial_\eta)$ is $\geq1$; but we assumed that the order of the pole of the connection is at most two, hence the biggest slope of the Newton polygon (Katz invariant) is $\leq1$. Both conditions imply that the Newton polygon has only the slopes~$0$ and~$1$.
\item
For any~$\beta$ with $\nu_\beta\neq0$, the function $\hb\mto\beta\hb$ should be written as $\hb\mto \gamma\hb^2+b\hb+\ov\gamma$ for some $\gamma\in\CC$ and $b\in\RR$ (\cf \cite{Mochizuki07}). This implies $\gamma=0$ and $\beta\in\RR$.
\end{enumerate}

We should apply these conditions to any exponentially twisted module $\wt\cM\otimes\cE^{-\varphi/\hb}$. Condition \eqref{enum:pentes} applied to any $\wt\cM\otimes\cE^{-c/\hb\tau'}$ implies
\begin{enumerate}\setcounter{enumi}{2}
\item
$(\cH,\nabla_{\partial_\eta})$ satisfies the ``no ramification'' condition at $\eta=0$.\qedhere
\end{enumerate}
\end{proof}

\backmatter
\newcommand{\SortNoop}[1]{}\def\cprime{$'$}
\providecommand{\bysame}{\leavevmode\hbox to3em{\hrulefill}\thinspace}
\providecommand{\MR}{\relax\ifhmode\unskip\space\fi MR }
\providecommand{\MRhref}[2]{%
  \href{http://www.ams.org/mathscinet-getitem?mr=#1}{#2}
}
\providecommand{\href}[2]{#2}

\end{document}